\documentclass[11pt]{amsart}

\usepackage[lmargin=1in,rmargin=1in, bmargin=1in, tmargin=1in]{geometry}

\usepackage{graphicx}              
\usepackage{amsmath}               
\usepackage{amsfonts}              
\usepackage{amsthm}                
\usepackage{amssymb}
\usepackage{amscd}
 \usepackage{listings}
\usepackage{times}
\usepackage{comment}
\usepackage{mathtools}
\usepackage{tikz}
\usepackage[all,cmtip]{xy}
\usepackage{bm}

\usepackage{stmaryrd}
\usepackage{xcolor}

\usetikzlibrary{arrows, matrix}
\usepackage{tikz-cd}
\usepackage{mathrsfs}
\usepackage{stmaryrd}
\usepackage{xcolor}
\usepackage[backref,hidelinks]{hyperref}
\hypersetup{
    colorlinks,
    citecolor=blue,
    filecolor=blue,
    linkcolor=blue,
    urlcolor=blue
}
\usepackage[strings]{underscore}
\usepackage[colorinlistoftodos,textsize=tiny]{todonotes}

\theoremstyle{definition}
\newtheorem{thm}{Theorem}[section]
\newtheorem{lem}[thm]{Lemma}
\newtheorem{prop}[thm]{Proposition}
\newtheorem{cor}[thm]{Corollary}
\newtheorem{conj}[thm]{Conjecture}
\newtheorem{example}[thm]{Example}

\newtheorem{defn}[thm]{Definition}

\newtheorem{q}{Question}

\newtheorem{rem}[thm]{Remark}

\newtheorem{mainthm}{Theorem}

\newcommand{\RR}{\mathbf{R}}      
\newcommand{\ZZ}{\mathbf{Z}}      

\newcommand{\Gm}{\mathbf{G}_m}
\newcommand{\kk}{\mathsf{k}}

\newcommand{\GSp}{\mathrm{GSp}}

\newcommand{\A}{\mathbb{A}}

\newcommand{\Z}{\mathbf{Z}}
\newcommand{\Q}{\mathbf{Q}}
\newcommand{\R}{\mathbf{R}}
\newcommand{\SL}{\mathrm{SL}}
\newcommand{\C}{\mathbf{C}}
\newcommand{\F}{\mathbf{F}}

\newcommand{\GL}{\mathrm{GL}}

\newcommand{\G}{\mathbf{G}}

\newcommand{\g}{\mathfrak{g}}

\newcommand{\spec}{\mathrm{Spec}\, }

\newcommand{\acts}{\curvearrowright}

\newcommand{\Ggr}{\mathbf{G}_{\mathrm{gr}}}

 \DeclareFontFamily{U}{wncy}{}
    \DeclareFontShape{U}{wncy}{m}{n}{<->wncyr10}{}
    \DeclareSymbolFont{mcy}{U}{wncy}{m}{n}
    \DeclareMathSymbol{\Sh}{\mathord}{mcy}{"58}

\usepackage{hyperref}
\hypersetup{
    colorlinks,
    citecolor=blue,
    filecolor=blue,
    linkcolor=blue,
    urlcolor=blue
}

\begin{document}

\author{Antonio Cauchi}
\address{Antonio Cauchi\newline University College Dublin\\ School of Mathematics and Statistics\\ Science Centre - South, Belfield, Dublin 4, Dublin, Ireland}
\email{antonio.cauchi@ucd.ie}

\author{Eric Yen-Yo Chen}
\address{Eric Yen-Yo Chen \newline École Polytechnique Fédérale de Lausanne, Lausanne, Switzerland}
\email{eric.chen@epfl.ch}

\author{Armando Gutierrez Terradillos}
\address{Armando Gutierrez Terradillos\newline 
Department of Mathematics, Aarhus University, Ny Munkegade 118, DK-8000, Aarhus, Denmark}
\email{armangute@math.au.dk}

\title{Relative Langlands duality of the Bump--Friedberg--Ginzburg $\mathrm{GSO}_6$-integral}
\begin{abstract}
 We provide a new instance of singular relative Langlands duality, underlying a Rankin--Selberg integral on $\mathrm{GSO}_6$ due to Bump--Friedberg--Ginzburg. We conclude that this integral represents an essentially self-dual object in the relative Langlands program, and we demonstrate that the Langlands dual automorphic integral computes a finite sum of $L$-functions, reflecting stacky structure on the spectral side. 
\end{abstract}

\maketitle

\setcounter{tocdepth}{2}
\tableofcontents

\newpage

\section{Introduction}
Let $F$ be a global field, and let $G$ be a reductive group over $F$ which we assume to be split for simplicity. The arithmetic form of the Langlands correspondence roughly stipulates that automorphic forms on $G$, i.e., $L^2$-functions on the adelic quotient
$$[G] = G(F) \backslash G(\mathbb{A}_F),$$
which appear in the spectral decomposition of the right $G(\mathbb{A}_F)$-action, can be labeled by $L$-parameters with image in the Langlands dual group $\check{G}$ of $G$. The Hecke eigenvalues of an eigenform $f$ on $G$ can be directly calculated from its $L$-parameter, and this compatibility constraint essentially characterizes the labeling. 

Typically, for applications one seeks a more quantitative description of the Langlands correspondence afforded by a \textit{relative Langlands correspondence}, in the recently introduced terminology of Ben-Zvi--Sakellaridis--Venkatesh \cite{BZSV}. Concretely, one collects numerical data from a Hecke eigenform via \textit{automorphic period integrals}, and wonders how the same numerical invariant may be extracted from its $L$-parameter. Traditionally, such numerical invariants took the form of \textit{special values of $L$-functions}, and successful cases of a relative Langlands correspondence can be understood as establishing diagrams of the following form:
\begin{equation}\label{equation period L function diagram}
    \begin{tikzcd} 
	{\text{Automorphic forms on } G} && {\check{G}\text{-valued } L\text{-parameters}} \\
	& {\mathbf{C}}
	\arrow[from=1-1, to=1-3]
	\arrow["{\text{automorphic periods}}"', from=1-1, to=2-2]
	\arrow["{\text{Langlands corresp.}}"', from=1-3, to=1-1]
	\arrow["{L\text{-values}}", from=1-3, to=2-2]
\end{tikzcd}
\end{equation}
A fundamental philosophical consequence of the relative Langlands program is that the diagram \eqref{equation period L function diagram} itself can be Langlands dualized, after which the automorphic and spectral roles of $G$ and $\check{G}$ are interchanged:
\begin{equation}\label{equation period L function diagram 2}
    \begin{tikzcd} 
	{\text{Automorphic forms on } \check{G}} && {G\text{-valued } L\text{-parameters}} \\
	& {\mathbf{C}}
	\arrow[from=1-1, to=1-3]
	\arrow["{\text{automorphic periods}}"', from=1-1, to=2-2]
	\arrow["{\text{Langlands corresp.}}"', from=1-3, to=1-1]
	\arrow["{L\text{-values}}", from=1-3, to=2-2]
\end{tikzcd}
\end{equation}
In other words, automorphic period integrals occur in Langlands dual pairs. With this in mind, a program was initiated by the second named author's thesis \cite{thesis} and following joint work with A. Venkatesh \cite{CV} to answer the following question:
\begin{q} \label{question}
    Given one of the many automorphic integrals in the literature, how do we compute its Langlands dual automorphic integral?
\end{q}
    
In the rest of the introduction, we discuss in more detail the precise formulation of Question \ref{question}, our current state of knowledge, and the key improvements, building upon previous work, made by the present project.

\subsection{Integral representations of \texorpdfstring{$L$}{L}-functions}

The theory of integral representations of $L$-functions produces concrete examples of diagram \eqref{equation period L function diagram}, in many cases predating the underlying Langlands correspondence itself. Among the illustrious instances of this theory are the following instructive examples:
\begin{itemize}
    \item (Riemann, Iwasawa--Tate). Let $\chi$ be an (adelic) Hecke character on $[\Gm]$ whose $L$-parameter we continue to denote by $\chi$. Then 
    $$\int_{[\Gm]} \, \chi(x)\sum_{\gamma \in F}\Phi(\gamma x) \, dx = L(1/2, \chi, \mathrm{std}),$$
    where $\Phi$ denotes a well-chosen Schwartz function on the adelic points of the 1-dimensional representation of $\Gm$, and $\mathrm{std}$ denotes the 1-dimensional standard representation of the Langlands dual copy of $\Gm$.
    \item (Hecke). Let $f$ be a cuspidal automorphic form on $G = \mathrm{PGL}_2$ with $L$-parameter $\varphi_f$, and let $T \subset G$ be the diagonal torus. Then 
    $$\int_{[T]} \, f(t) \, dt = L( 1/2,\varphi_f, \mathrm{std}),$$
    where $\mathrm{std}$ here denotes the 2-dimensional standard representation of the Langlands dual group $\check{G} = \mathrm{SL}_2$.
    \item (Jacquet--Shalika). Let $f$ be a cuspidal automorphic form on $G = \mathrm{PGL}_{2n}$ with $L$-parameter $\varphi_f$. Then the period
    $$\int_{[\mathrm{PGL}_n][\mathrm{Mat}_n]} \, f\left(\begin{bmatrix} 1 & X\\ 0&1\end{bmatrix}\begin{bmatrix} m & 0\\ 0&m \end{bmatrix}\right) \psi(\mathrm{tr}(X))  \, dX \, dm$$
    vanishes unless $\varphi_f$ is conjugate to a symplectic subgroup of $\check{G} = \mathrm{SL}_{2n}$, and in which case the period computes the $L$-value $L(1,\varphi_f,  \wedge^2_0)$ where $\wedge^2_0$ is the second fundamental representation of $\mathrm{Sp}_{2n}$.
\end{itemize}

A rather distinct class of integrals was first discovered independently by Rankin and Selberg, and reformulated by Jacquet--Piatetskii-Shapiro--Shalika (\cite{Jacquet-Shalika-RSConvolutions}, \cite{Jacquet-PS-Shalika-RSConvolutions}) in the adelic setting, which took as input two modular forms and produced the tensor product $L$-function as its principal output. 
\begin{itemize}
    \item (Rankin--Selberg, Jacquet--Piatetskii-Shapiro--Shalika). Let $f = f_1 \times f_2$ be cuspidal automorphic forms on $G = \mathrm{P}(\mathrm{GL}_2 \times \mathrm{GL}_2)$, with $L$-parameter $\varphi_f = \varphi_1 \times \varphi_2$ valued in $\check{G}$. Then
    $$\int_{[\mathrm{PGL}_2]}f_1(g)f_2(g)E(g,s) \, dg = L(s,\varphi_1 \times \varphi_2, \mathrm{std}_1 \otimes \mathrm{std}_2)/\zeta(2s),$$
    where $E(g,s)$ denotes a \textit{naïve} Eisenstein series induced from the Borel subgroup of $\mathrm{GL}_2$, and $\mathrm{std}_1 \otimes \mathrm{std}_2$ denotes the tensor product of the standard representations of $\check{G} = \mathrm{S}(\mathrm{GL}_2 \times \mathrm{GL}_2)$.
\end{itemize}
Note that, contrary to the three integrals in the first list, the right hand side of the formula in this case is not an $L$-function, but rather a \textit{ratio} of $L$-functions.

The four types of automorphic integrals displayed above are emblematic of the vast majority\footnote{Notably we are missing the rather large class of automorphic integrals involving $\theta$-series, such as the one considered by Shimura \cite{Shimura} and Bump--Ginzburg \cite{Bump-Ginzburg} in their representation of the symmetric square $L$-function. } of successful integral representations of $L$-functions that one encounters in the literature: they are obtained by a combination of (i) $L^2$-pairing against a Schwartz function on a linear representation, (ii) integrating over a reductive subgroup, (iii) taking certain Fourier--Whittaker coefficients, and (iv) $L^2$-pairing against an Eisenstein series. The first three building blocks (i)-(iii) of automorphic integrals are conveniently encoded by the formalism of \cite{BZSV}, through which one labels an automorphic integral on $[G]$ by certain \textit{Hamiltonian $G$-spaces} (with some additional extra structure).

We recall briefly the labeling proposed by \textit{op. cit.} in a form directly useful for us. In particular, we will specialize to the case of (twisted) polarized Hamiltonian spaces for the purposes of this introduction, where one encodes the automorphic integral via a $G$-variety $X$ (and an affine line bundle $\Psi$ in the twisted case). Given such data, the shape of the associated automorphic period is as follows
\begin{equation} \label{equation general BZSV integral}
    f \longmapsto \int_{[G]} \, f(g)\theta_{X}(g) \, dg =: \langle f, \theta_{X}\rangle
\end{equation}
where $\theta_{X}$ is a certain \textit{theta series} associated to the $G$-variety $X$ (and in the twisted case, combined with a partial Fourier coefficient encoded by $\Psi$). Indeed, the automorphic integral \eqref{equation general BZSV integral} specializes to the left hand sides of the Riemann/Iwasawa--Tate case (by taking the standard representation $X = \mathbf{A}^1$), the Hecke case (by taking $X = T \backslash \mathrm{PGL}_2$), and the Jacquet--Shalika case (by taking $(X,\Psi) = (\mathrm{Mat}_n\mathrm{PGL}_n \backslash \mathrm{PGL}_{2n}, \Psi_{\mathrm{tr}})$, where $\Psi_{\mathrm{tr}}$ is the affine line bundle encoded by the trace character on $\mathrm{Mat}_n$).

A fundamental observation of the \textit{relative Langlands program} \cite{BZSV} posits that, not only can automorphic periods be labeled by (Hamiltonian) actions, but so can $L$-values (for $L$-parameters into $\check{G}$). One defines the $L$-value labeled by some smooth $\check{G}$-variety $\check{X}$\footnote{We ignore the complications presented by a nontrivial $\Psi$ on the Galois side for now. For an all-inclusive definition, see the right hand side of Conjecture 14.3.5 in \cite{BZSV}.} for a $\check{G}$-valued $L$-parameter as follows:
\begin{equation}\label{equation general L function}
    \varphi \longmapsto \sum_{x \in \mathrm{Fix}(\varphi, \check{X})} \, L\big(\varphi, T_x\check{X}\big) =: L_{\check{X}}(\varphi)
\end{equation}
where $\mathrm{Fix}(\varphi, \check{X})$ denotes the set of fixed points of the $L$-parameter $\varphi$ on the $\check{G}$-space $\check{X}$, and $L\big(\varphi, T_x\check{X}\big)$ denotes the usual Langlands $L$-function\footnote{More precisely, the $L$-function evaluated at some specific point depending on finer structures on the action that we omit for the sake of exposition.} attached to the representation on the tangent space $T_x\check{X}$. Indeed, the $L$-value \eqref{equation general L function} specializes to the right hand sides of the Riemann/Iwasawa--Tate case (by taking the standard representation $\check{X} = \mathbf{A}^1$), the Hecke case (by taking the standard representation $\check{X} = \mathbf{A}^2$) and the Jacquet--Shalika case (by taking $\check{X} = \mathrm{Sp}_{2n} \backslash \mathrm{SL}_{2n}$). For an explicit explanation of the last case in particular, see Lemma \ref{lemma spectral induction}.

With the preceding reformulations in mind, one makes the following heuristic definition: a pair of actions $(G,X)$ and $(\check{G}, \check{X})$ are Langlands dual if 
\begin{equation}
    \langle f, \theta_X\rangle = L_{\check{X}}(\varphi_f)
\end{equation}
for cusp forms $f$ on $G$ with $L$-parameter $\varphi_f$ valued in $\check{G}$. However, once convinced that the two arrows in diagram \eqref{equation period L function diagram} can be labeled by equivalent data, one is led to the following natural prediction of relative Langlands duality, that integral representation formulae come in Langlands dual pairs: there should exist a \textit{Langlands dual automorphic period formula}
\begin{equation}
    \langle f', \theta_{\check{X}}\rangle \overset{?}{=} L_{X}(\varphi_{f'})
\end{equation}
for cusp forms $f'$ on $\check{G}$ with $L$-parameter $\varphi_{f'}$ valued in $G$. While a successful list of well-behaved examples can be found in Appendix B of \cite{CV} (and many more are expected if not implicit in the literature), the statement that the two formulae are logically equivalent seems to be out of reach at the moment. 

\subsection{Langlands dualizing Rankin--Selberg integrals}\label{dualizing Rankin}
Integrals of Rankin--Selberg type appear, at first glance, to be more difficult to capture using the same type of labels. This difficulty is already apparent in the Rankin--Selberg integral presented above: on neither the automorphic nor the Galois side are the numerical outputs of the form \eqref{equation general BZSV integral} or \eqref{equation general L function}, respectively. In fact, examples abound in the literature (such as \cite{Ginzburg-adjoint}, \cite{GinzburgRallisExteriorCube},  \cite{Ginzburg-Hundley},  \cite{Ginzburg-Hundley-Orthogonal},  \cite{Bump-Ginzburg-Spin9},   \cite{Bump-Ginzburg-Adjoint},  \cite{tower},  \cite{newtower} to cite a few) where integral representation formulae cannot be encoded by (Hamiltonian) $G$-actions considered by \cite{BZSV} and thus cannot be ``Langlands dualized" in a straightforward manner. 

A partial solution to this problem was proposed in the second named author's thesis \cite{thesis} and the subsequent work \cite{CV}: by extending the allowed labels of automorphic periods and $L$-values to include certain \textit{singular varieties}, one may recover the underlying Langlands duality for the integral representation of Garrett's triple product $L$-function \cite{Garrett} and Ginzburg's $\mathrm{SL}_3$ adjoint $L$-function \cite{Ginzburg-adjoint}. The key idea is to carefully redistribute the normalizing factors of the Eisenstein series and the $\zeta$-factors that appear on the right hand side of Rankin--Selberg-type integrals, in order to achieve an automorphic period (and $L$-value) which is associated to a spherical, albeit singular, variety. Since this is not an innovation of our present project but nonetheless informs our starting point, we simply state its consequences and refer the reader to \cite[\S 3.2.4]{CV} for a detailed discussion.

Experimentally, Rankin--Selberg type formulae often take the following shape: for a cuspidal automorphic form $f$ on $G$ and a well-chosen \textit{normalized} Eisenstein series $E^*(g,s)$ on $G$, we have
\begin{equation} \label{equation typical RS integral}
    \int_{[G]} \, f(g)E^*(g,s) \, dg = \frac{L(s,\varphi_f,V)}{\prod_i \, \zeta(d_is)}
\end{equation}
where $\varphi_f$ denotes the $\check{G}$-valued $L$-parameter of $f$, $V$ is a representation of $\check{G}$ whose Langlands $L$-function $L(s,\varphi_f,V)$ one seeks to represent, and the $d_i$'s are certain positive integers. In many cases\footnote{To the best of our knowledge, this numerical phenomenon was first pointed out by Ginzburg and Rallis in \cite{tower}.}, the $d_i$'s that appear are exactly the degrees of $\check{G}$-invariant polynomials on $V$. Write $g_1, \ldots, g_r$ for the homogeneous generators of $\C[V]^{\check{G}}$ of homogeneous degrees $d_1, \ldots, d_r$, then we define the \textit{nilcone} of $V$ to be the joint vanishing locus of $g_1, \ldots, g_r$. The first fundamental observation one makes is that (Proposition \ref{proposition quotient of L functions})
\begin{equation} \label{equation spectral period of nilcone}
    \text{RHS of \eqref{equation typical RS integral}} =\text{ spectral period of the nilcone of }V.
\end{equation}
On the other hand, on the left hand side the Eisenstein series $E^*(g,s)$ is written as an induction from some parabolic subgroup $P \subset G$ normalized by various $\zeta$-factors. We observe that in some cases, for some special value $s = s_0$, one can replace $E^*(g,s_0)$ by a $\theta$-series attached to the affine cone $\overline{U_PM_P'\backslash G}^{\mathrm{aff}}$ over the Grassmannian $P \backslash G$, where $P = U_P M_P$ is the Levi decomposition of $P$ and $M_P^{\mathrm{der}} \subseteq M_P' \subset M_P$ is a suitably chosen subgroup. In other words, 
\begin{center}
    LHS of \eqref{equation typical RS integral} = automorphic period of the affine cone over $P \backslash G$.
\end{center}
These observations allow us to rewrite \eqref{equation typical RS integral} in a form that is more compatible with the relative Langlands program. 

It is important to note that a large class of Rankin--Selberg integrals is obtained as the integration over a subgroup $H$ of a cuspidal automorphic form $f$ on $G$ and an Eisenstein series $E^*(g,s)$ on $H$. These can be encoded via the procedure of induction from $H$ to $G$, as we now explain. Replace $E^*(g,s_0)$ by the theta series $\theta_Y$ attached to the affine cone $Y=\overline{U_P M_P'\backslash H}^{\mathrm{aff}}$ over the Grassmannian $P\backslash H$. Then, if $X = \mathrm{Ind}_H^G(Y):=Y \times^H G$ denotes the induction of $Y$ from $H$ to $G$, we have  \begin{equation} \label{equation typical RS integral with restriction to H}
    \int_{[H]} \, f(h)E^*(h,s_0) \, dh = \text{ automorphic period of }X,
\end{equation}
under a certain cohomological assumption such as $\mathrm{H}^1(F, H) = 0$, see Lemma \ref{lemma automorphic induction}.  We invite the reader to consult  \S \ref{subsection period induction}, where we record various results on the effect of induction on automorphic and spectral periods. Most of these are known to experts, but we discuss them for the sake of completeness and for precise future reference.

\subsection{Main results}

Our present purpose is to construct another example of singular relative Langlands duality underlying an integral representation formula discovered by Bump--Friedberg--Ginzburg \cite{BFGsplitorthogonal} and later rivisited by \cite{CauchiGutiCS,CauchiGutiMVI} on $\mathrm{PGL}_4$. This is an emblematic example of a Rankin--Selberg-type integral, whose shape we describe below.  Let $f$ be a cuspidal automorphic form on $\mathrm{PGL}_4$ with $L$-parameter $\varphi_f$, and consider the period
\begin{equation} \label{equation CG integral}
    f \longmapsto \int_{[\mathrm{PGSp}_4]} \, f(g)E_{\mathrm{Siegel}}(g,s) \, dg,
\end{equation}
where $E_{\mathrm{Siegel}}(g,s)$ denotes the \textit{fully normalized} Siegel Eisenstein series on $\mathrm{PGSp}_4$. Then the period vanishes unless $\varphi_f$ is conjugate to the symplectic subgroup of $\mathrm{SL}_4$, and in which case the period computes the $L$-function $L(s,\varphi_f, \wedge_0^2)$, where $\wedge_0^2$ is the 5-dimensional second fundamental representation of $\mathrm{Sp}_4$. If we substitute $E_{\mathrm{Siegel}}(g,s)$ by the \emph{partially normalized} Eisenstein series $E^*_{\mathrm{Siegel}}(g,s) := \zeta(2s)^{-1} E_{\mathrm{Siegel}}(g,s)$, the result of \emph{loc. cit.} conveniently reads as $$\text{\eqref{equation CG integral}} = \frac{L( s,\varphi_f, \wedge_0^2)}{\zeta(2s)}.$$
It is now crucial to observe that the right hand side is precisely the spectral period of the affine cone $Y$ over the Lagrangian Grassmannian $\mathrm{LGr}(2,4)$ of 2-planes in the standard representation of $\mathrm{GSp}_4$. In view of the discussion in \S \ref{dualizing Rankin}, observe that the discrepancy between the spectral period of $Y$ and that of $\mathrm{Ind}_{\mathrm{GSp}_4}^{\GL_4}(Y)$ is, by Lemma \ref{lemma spectral induction},  the factor $$L( 1, \varphi_f, \wedge_0^2).$$ By a result of Jacquet--Shalika \cite[Proposition 1]{Jacquet-Shalika}, this special $L$-value is computed by the Shalika period of $f$, reflecting the fact that \eqref{equation CG integral} unfolds naturally to the automorphic period of \textit{loc. cit}.

With this in mind, our main result concerns giving an answer to Question \ref{question} in the case of \eqref{equation CG integral}. We discover that the Bump--Friedberg--Ginzburg integral is more or less Langlands self dual, up to an isogeny that reflects stacky structure on the dual side. 

Generalizing slightly, we work with the centrally extended group $\check{G} := \mathrm{GSO}_6$ of $\mathrm{PGL}_4$, allowing for nontrivial central characters in \eqref{equation CG integral}, whose Langlands dual group is $G = \mathrm{GSpin}_6$. There is a copy of the symplectic group $\mathrm{GSp}_4$ (resp. $\overline{\mathrm{GSp}}_4' := \mathrm{GSp}_4 \times \Gm/(z\mathrm{Id}, z^{-2})$) embedded in $G$ (resp. $\check{G}$).  We then consider an isogenous cover $\mathrm{GSp}_4'$ of  $\overline{\mathrm{GSp}}_4'$ with kernel $\mu_2$ and we let the affine cone  $Y$  be equipped with the actions of $\mathrm{GSp}_4$ and $\mathrm{GSp}_4'$ described in \S \ref{Section:The:Space:G:X} and \S \ref{section Gcheck Xcheck}. We state the resulting singular Langlands duality here in a rough form, and refer the reader to Theorems \ref{theorem main duality}, \ref{thm:Final:PX}, and \ref{thm PXcheck final} in the main text for a precise formulation:

\begin{mainthm} \label{theorem A}
   Consider the pair of Langlands dual groups $G = \mathrm{GSpin}_6$ and $\check{G} = \mathrm{GSO}_6$, with the actions
    $$G \acts X = \mathrm{GSp}_4 \backslash (Y \times G) \, \text{ and } \check{G} \acts \check{X} = \overline{\mathrm{GSp}}_4' \backslash ([Y/\mu_2] \times \check{G}),$$
    where the action of $\mathrm{GSp}_4$ and $\overline{\mathrm{GSp}}_4'$ on $Y \times G$ and $[Y/\mu_2] \times \check{G}$ are specified in \S \ref{Section:The:Space:G:X} and \S\ref{section Gcheck Xcheck}, respectively. Then we have a Langlands dual pair of period integrals
    \begin{equation}
        \langle f_1, \theta_X\rangle = \Delta^{-5/4} \cdot L_{\check{X}}(\varphi_1) \text{ and } \langle f_2, \theta_{\check{X}}\rangle = \Delta^{-5/4} \cdot L_X(\varphi_2)
    \end{equation}
    where $f_1$ (resp. $f_2$) is a suitably normalized, everywhere unramified tempered cusp form on $G$ (resp. $\check{G}$) with $L$-parameter $\varphi_1$ (resp. $\varphi_2$).
\end{mainthm}

We would like to emphasize that the appearance of the (singular) Deligne--Mumford stack $\check{X}$ arises naturally as the answer to Question \ref{question}. In particular, the generalized $L$-value $L_{\check{X}}$ is not a single Eulerian product, but rather a sum of quadratic twists of quotients of Langlands $L$-functions. The notions of automorphic periods and $L$-values attached to Deligne--Mumford stacks will be treated more carefully in \S \ref{subsection DM stacks} of the main text, and we refer the reader to the introductory section \S \ref{subsubsec DM staks intro} for a conceptual explanation of the appearance of such structures.

On the one hand, the automorphic period $\langle f_2, \theta_{\check{X}}\rangle$, which unfolds to the Shalika model on $\mathrm{GSO}_6$, is essentially equivalent to the Bump--Friedberg--Ginzburg integral \eqref{equation CG integral} and extends it to the case of cusp forms with arbitrary central character. On the other hand, our method of computation of the automorphic period $\langle f_1, \theta_X\rangle$ on $\mathrm{GSpin}_6$ hinges upon a natural unfolding to a nonstandard Shalika-type model of $f_1$. By leveraging the natural inclusion $\mathrm{GSpin}_6 \subset \mathrm{GL}_4 \times \mathrm{GL}_1$, this model can be spectrally expanded in order to directly compare with the standard Shalika model. These results - which may be of independent interest - are detailed in \S\ref{subsec:stackyJS} and \S\ref{subsection Shalika for GSpin6}, with the main statements given in Theorem \ref{conjecture stacky Shalika dual pair} and Proposition \ref{Shalika:GSpin:In:GL}.

\subsection{Difficulties and new tools}

In order to uncover the Langlands duality structure underlying the integral representation formula \eqref{equation CG integral}, there are notable difficulties and new phenomena which are typical, and they must be overcome in order to improve our understanding of Question \ref{question}. The new tools we develop, with a view towards compatibility with the ideas of relative Langlands duality à la Ben-Zvi--Sakellaridis--Venkatesh, form the technical heart of our current discussion, and we aim to tackle similar but more challenging examples in future work utilizing the techniques established here.

Besides the strategy proposed by \cite{thesis} and \cite{CV}, i.e., the replacement of the automorphic sides of Rankin--Selberg type integrals by certain \textit{affine cones} and the spectral sides by certain \textit{nilpotent cones}, we highlight the following key pieces of innovation realized in the course of the proof of Theorem \ref{theorem A}.

\subsubsection{Sums of Eulerian products} \label{subsubsec DM staks intro}

While the prototypically successful automorphic integral evaluates to a single Eulerian product which one seeks to compare to (products and quotients of) Langlands $L$-functions, examples abound where an automorphic integral evaluates to a sum, even an infinite sum, of Eulerian products. The fact that such phenomena should reflect \textit{spectral expansions} on the automorphic side, and the presence of \textit{stabilizers}, or \textit{stackyness}, on the spectral side was first informally discussed in A. Venkatesh's presentation at the Bernstein 75 Conference \cite{Bernstein75}.

To summarize the argument therein, suppose $G$ is a reductive group and $H \subset G$ is an algebraic subgroup which is not necessarily spherical, and we would like to evaluate the $H$-period of automorphic forms on $G$. Formally speaking, one can try to find some intermediate reductive subgroup $H \subset S \subset G$ such that $H$ \textit{is spherical} in $S$, and expand in the automorphic spectrum of $S$ to compute this period in terms of an infinite sum of $L^2$-inner products:
\begin{equation} \label{equation spectral expansion}
    \int_{[H]}f(h) \, dh = \sum_{\xi \in \mathcal{A}(S)} \, \langle \xi, f\rangle_{L^2([S])}\int_{[H]} \, \xi(h) \, dh,
\end{equation}
where $\mathcal{A}(S)$ denotes an ``orthonormal basis" of automorphic forms on $S$. If one is lucky, then each of the terms appearing in the preceding formula are now spherical, i.e., one may try to relate $\langle \xi, f\rangle$ and $\int_{[H]}\xi(h) dh$ separately to $L$-functions, the former in terms of the $L$-parameters of both $\xi$ and $f$, while the latter depends only on the $L$-parameter of $\xi$. Schematically, we may thus write
$$\eqref{equation spectral expansion} = \sum_{\varphi_\xi} \, L(\varphi_\xi \boxtimes \varphi_f)L(\varphi_\xi),$$
where $\varphi_\xi, \varphi_f$ denote the $L$-parameters of $\xi$ and $f$ respectively, and $L(\varphi_\xi \boxtimes \varphi_f)$ and $L(\varphi_\xi)$ denote some hypothetical $L$-functions. The fact that \textit{spectral expansion} as performed heuristically above can be useful in evaluating nonspherical periods is not a new idea, but the essential observation of \textit{loc. cit.} was that the right hand side of the preceding formula may be interpreted as the spectral period (Definition \ref{definition spectral period}) of an \textit{Artin stack} with stabilizer $\check{S}$, the Langlands dual group of $S$.

Following this line of thought, one is led to consider the simpler case of \textit{Deligne--Mumford stacks}, where the analytic difficulties in making sense of infinite expressions in spectral expansions disappear. The first such examples were proposed by the second named author \cite{toric}, which indeed arise from discrete spectral expansions in which the groups involved were tori. The example analyzed in this article presents the first \textit{nonabelian} example of automorphic period duality involving Deligne--Mumford stacks and lends credence to the fact that the Langlands dual of many existing automorphic integrals must be phrased in terms of \textit{singular Deligne--Mumford stacks}. 

\subsubsection{Normalization and discrepancies in period formulae}

The vast majority of automorphic integrals attempt to unfold to the Whittaker period in order to deduce Eulerianity of the integral on representation-theoretic grounds (via multiplicity one results on the local Whittaker model). This is not \textit{a priori} the only available strategy to compute automorphic integrals, as we now have access to various local models with multiplicity one properties. This strategy has not seen popular use, even after the appearance of Sakellaridis' general Casselman--Shalika formula \cite{SakellaridiSsphericalFunctions}, due to the relatively complicated translation between local spherical function evaluations and Weyl characters of the Langlands dual group. 

The integral \eqref{equation CG integral}, however, presents an interesting example of an emerging paradigm: there exists automorphic integrals which readily unfold to non-Whittaker, multiplicity one models, and such integrals of a cuspidal automorphic form can be evaluated in terms of its $L$-parameter via Sakellaridis' Casselman--Shalika formula as a \textit{nonabelian $L$-function} in the terminology of \cite{CV}. Besides our present example of interest which involves the \textit{Shalika model} (see \S \ref{section Shalika model} for a definition), other Eulerian integrals which appear naturally as Langlands dual partners of existing integral representations have also been discovered in \cite{thesis} (involving Ginzburg--Rallis' integrals for the exterior square $L$-function of $\mathrm{GSpin}_7$ \cite{tower} and for the exterior cube $L$-function of $\mathrm{GL}_6$ \cite{GinzburgRallisExteriorCube}).

With a view towards future work, we set up a convenient normalization scheme for automorphic periods unfolding readily to the Shalika model, and explain its utility towards establishing relative Langlands duality statements in \S \ref{subsection Shalika normalization}. Essentially, the numerical conjectures of \cite[Chapter 14]{BZSV} can be formulated without privileging a choice of normalization if one always works with \textit{ratios of periods}; the numerical conjectures of \textit{loc. cit.} and the weak numerical duality formulated in \cite{CV} thus privilege the Whittaker period (resp. the point $L$-function) by using it as a denominator against which all other automorphic periods (resp. $L$-values) are compared. 

Furthermore, we keep track of the contribution to discrepancies (see Definition \ref{definition weak numerical duality}) on the automorphic and spectral sides of a hypothetical period formula. This leads to an explicit formula, the \textit{discrepancy equation} (see Definition \ref{definition discrepancy equation}) measuring the  numerical deviation from the smooth case, where discrepancy-free duality formulae are expected.

\subsection{Notation and conventions}
\subsubsection{Groups}\label{Subsection:Groups:Notation}
$G$ and $\check{G}$ (or sometimes $G_1$ and $G_2$ for notational symmetry) will denote a pair of split Chevalley forms of Langlands dual reductive groups over $\Z$. We regard these reductive groups as equipped with standard pinning, with the letters $T$ and $\check{T}$ denoting Langlands dual maximal split tori in $G$ and $\check{G}$. 

We write $e^{2\rho} \in \mathrm{X}^*(T) = \mathrm{X}_*(\check{T})$ for the sum of positive roots of $G$, and $e^{2\check{\rho}} \in \mathrm{X}_*(T) = \mathrm{X}^*(\check{T})$ for the sum of positive coroots of $G$. 

We always write $\Gm := \spec \Z[t^{\pm}]$ as the group scheme of units defined over $\Z$. Identifying $\mathrm{X}^*(\mathbf{G}_m)$ with $\Z$, for each $d\in\Z$ we denote by $\varsigma_d$ the character of $\mathbf{G}_m$ defined by $\varsigma_d(x) := x^d$.

\subsubsection{Coefficient fields}
The letter $\F$ will be used for an \textit{automorphic side} coefficient field, i.e., a finite field of prime power size $q$ which we assume is odd. On the other hand, the letter $\kk$ will be used for a \textit{Galois side} coefficient field, i.e., the algebraic closure of $\ell$-adic rational numbers, or a copy of $\C$. Sometimes, we will use the letter $k$ to denote either $\F$ or $\kk$, to make uniform arguments that are useful for both the automorphic and the Galois side. 

\subsubsection{Group actions}
\label{subsubsection graded group actions}
$G$ (and $\check{G}$) will act by \textit{right actions} on varieties, inducing \textit{left actions} on functions by pullback. As a conseuence, by \textit{left translation} we will always mean the right action resulting from left translation by an inverse.

We denote by $\Ggr$ a distinguished copy of the multiplicative group. If $\Ggr$ acts on a variety $X$, then the ring of functions on $X$ is naturally graded according to $\Ggr$-weights. If $x \in X$ is a $\Ggr$-fixed point, then the completed ring of functions $\widehat{\mathcal{O}}_{X,x}$ at $x$ is equipped with a grading as well.

\subsubsection{Curves}
Let $\Sigma$ be a curve of genus $g$ over $\F$ with function field $F$. We write $\Delta$ for the discriminant of $\Sigma$, i.e., 
\begin{equation}
    \Delta := q^{2g-2}.
\end{equation}

\subsubsection{Langlands parameters}

We write $\Gamma$ for the unramified global Weil group of $\Sigma$, i.e., the preimage of integral powers of Frobenius in the \'etale fundamental group of $\Sigma$. All Langlands parameters will be understood as (equivalence classes of) homomorphisms from $\Gamma$ into some $\kk$-valued points of a reductive group. 

\subsubsection{Automorphic forms}
We denote by $\mathbb{A}$ the adele ring of $F$ and by $\mathfrak{o} \subset \mathbb{A}$ the maximal compact subring. Write
\begin{equation} \label{unramified adelic quotient}
    [G] := G_F \backslash G_\mathbb{A}/G_\mathfrak{o}
\end{equation}
for the unramified adelic quotient of $G$, i.e., the groupoid of isomorphism class of $\F$-rational $G$-bundles over $\Sigma$ when $G$ is a connected reductive group. Similar notation is used for nonreductive algebraic groups defined over $F$. Unless otherwise specified, by an automorphic form on $G$ we mean an unramified Hecke eigenform $f$ on $[G]$, whose central character we denote by $\omega_f$ and whose $L$-paramter we denote by $\varphi_f: \Gamma \to \check{G}(\kk)$. Similarly, if $\pi$ is an unramified automorphic representation, we denote by $\omega_\pi$ its central character and $\varphi_\pi$ its $L$-parameter.

\subsubsection{Local notation} We use $v$ to denote a place of $F$, and we write $F_v, \mathfrak{o}_v, \varpi_v$ for the completion of $F$, the ring of integers, and the uniformizer, at $v$, respectively. We write $q_v$ for the size of the residue field at $v$. The normalized valuation will be denoted $x \mapsto |x|_v$, and the product over all $v$ of the $| \, \cdot \, |_v$ will be denoted $| \, \cdot \, |: \mathbb{A}^\times \to \RR^\times$. If $\chi$ is a cocharacter of some torus $T$, we write $\varpi_v^\chi$ for the element $\chi(\varpi_v) \in T(F_v)$.

\subsubsection{Additive characters}\label{subsec_add_chars}
As in \cite{CV}, we choose a distinguished additive character as follows. By a theorem of Hecke, there exists a choice of $\F$-rational spin structure $K^{1/2}$ on $\Sigma$ and we fix a rational section $\nu$ of $K^{1/2}$. Its square $\omega = \nu^{\otimes 2}$ is a regular global 1-form, and we take $\psi$ to be
$$F \backslash \mathbb{A} \ni f \longmapsto \psi(f) := \mathrm{Res}(f\omega),$$
the latter regarded as a scalar in $\kk$ by a choice of primitive $\mathrm{char}(\mathbf{F})$th roots of unity. Note that for each $v$, there is an even integer $2m_v$ with the property that $\psi$ has conductor $\varpi_v^{-2m_v}\mathfrak{o}_v$, i.e., 
$$\psi(x) = \otimes_v \, \psi_v^{\mathrm{ur}}(\varpi_v^{2m_v}x)$$
where $\psi_v^{\mathrm{ur}}$ denotes the unramified local additive character. We write
$$\partial^{1/2} = (\varpi_v^{m_v}) \in \mathbb{A}^\times \text{ so that } |\partial| = \Delta^{-1}.$$

\subsubsection{Measures} For $G$ a reductive group over $\Z$, we shall normalize the Haar measure on $G(\mathbb{A})$ so that $G(\mathfrak{o}) = \prod_v G(\mathfrak{o}_v)$ has volume 1, and for a coordinate $g$ on $G(\mathbb{A})$ (resp. on $G(F_v)$) we write $dg$ for this probability normalized Haar measure.

For the additive group $\mathbf{G}_a$ there are two measures of interest. First, the analogue of the measure described above, which we denote by $dx$ for a coordinate $x$ on $\mathbb{A}$ (resp. on $F_v$), assigns volume 1 to $\mathbf{G}_a(\mathfrak{o})$ (resp. to $\mathfrak{o}_v$). The other is the ``self-dual" measure, denoted $d^\psi x$, with respect to which the adelic quotient $[\mathbf{G}_a]$ has a self-dual Fourier transform with respect to $\psi$. Locally, a self-dual lattice is given by $\varpi_v^{-m_v}\mathfrak{o}_v$, which has $d^\psi x$-volume 1 but $dx$-volume $q_v^{m_v}$. Thus, we see that
\begin{equation} \label{equation switching to self dual measure}
    dx = \Delta^{1/2} d^\psi x,
\end{equation}
a useful identity which we will use in our computations.

\subsubsection{Normalization of class field theory} \label{subsubsection CFT normalization}

We normalize local class field theory so that the modulus character $x \mapsto |x|_v$ on $F_v^\times$ is associated with the cyclotomic character of the Galois group of $F_v$. In particular, the uniformizer $\varpi_v$ is sent to geometric Frobenius $\mathrm{Fr}_v$, evaluating to $q^{-1}_v$ under the cyclotomic character.

\subsection{Acknowledgments}
This project was initiated and partially completed while the authors were visiting the Institute for Mathematical Sciences, National University of Singapore in 2025-2026. We wish to express our sincere gratitude to the organizers of the program ``Relative Langlands Program'' at NUS for the invitation and for fostering a stimulating environment that made this collaboration possible. A.C.'s research in this publication was conducted with the financial support  of Taighde \'{E}ireann -- Research Ireland under Grant number IRCLA/2023/849 (HighCritical). E.Y.C. was partially supported by the Swiss National Science Foundation No. 196960 and the JSPS Postdoctoral Fellowship during the completion of this project. A.G.T. was supported by VILLUM FONDEN research grant VIL54509 during the preparation of this paper.

\section{Periods and duality}\label{Section:Periods:and:Duality}

The first step towards regarding automorphic and spectral periods on equal footing is the repackaging of both concepts in terms of group actions. More precisely, we make the following fundamental definition.

\begin{defn}
    Let $G$ be a reductive algebraic group defined over $k$. A \textit{graded $G$-variety} is a $k$-variety $X$ with $G$-action, equipped with a commuting $\Ggr$-action on $X$, and an eigen-volume form $\omega_X$ on $X$.
\end{defn}
Given a graded $G$-variety $(X, \omega_X)$, we denote by $\eta_X: G \times \Ggr \to \Gm$ the eigenvalue character with which $G \times \Ggr$ acts on $\omega_X$:
$$(g, \lambda)^*\omega_X = \eta_X(g, \lambda) \omega_X$$
and we write $\varepsilon_X$ for the integer weight for which $\eta_X(1, \lambda) = \lambda^{\varepsilon_X}$. If the context is clear, we shall drop the subscripts $X$ for notational simplicity.
\begin{rem}
    In fact, the eigen-volume form $\omega_X$ is not strictly speaking necessary. We need, for all intents and purposes, only a \textit{rational} character $\eta_X$, which may be interpreted as the eigenvalue of a \textit{rational} volume form on $X$. 
\end{rem}

All graded $G$-varieties that we consider will be equipped with an integral form over some ring of $S$-integers (compatible with the Chevalley form of $G \times \Ggr$). In principle only the automorphic side depends on such structure, yet the dependence is subtle and crucial.

\begin{rem}[Actions v.s. Hamiltonian actions] Conceptually, the key objects of the relative Langlands program, i.e., the automorphic and spectral periods to be described below, are symplectic invariants - they are associated to \textit{Hamiltonian actions} and not arbitrary actions. However, for the purposes of computing numerical results, we have chosen to stick with the minimal setup possible: we will continue to work with graded $G$-actions whenever possible, and the underlying Hamiltonian action will usually be obtained by taking cotangent bundles. 

There are a few cases in which considering just actions is not quite enough. More precisely, we may consider graded $G$-varieties $X$ with a choice of affine line bundle $\Psi$, in which case the associated Hamiltonian action will be the $\Psi$-twisted cotangent bundle of $X$. This is notably the case in \S \ref{section Shalika model}, and we will discuss it in detail there.
\end{rem}

\subsection{Automorphic and spectral periods}

Given the definition of graded $G$-varieties, we recall the key notions of automorphic and spectral periods attached to them.

\subsubsection{Automorphic Periods}
The normalized theta series of $X$ on $G(\mathbb{A})$ is defined by 
\begin{equation} \label{equation theta series}
    \theta_X(g) := \Delta^{\frac{\mathrm{dim}(X) - \mathrm{dim}(G)}{4}}|\eta_X(\partial^{1/2})|^{1/2} \sum_{x \in \mathring{X}(F)} \, |\eta_X(g)|^{1/2}\Phi(x \partial^{1/2} g),
\end{equation}
where $\Phi = \otimes_v \, \mathbf{1}_{X(\mathfrak{o}_v)}$ is the pure tensor product of  indicator functions of the integral points of $X$ at all places. Using $\theta_X$, we can define the automorphic $X$-period for everywhere unramified automorphic forms on $G$. Since $|\eta_X(\partial^{1/2})|^{1/2} = \Delta^{-\varepsilon_X/4}$, we can set
\begin{equation}\label{equation betaX}
    C_X := \Delta^{\frac{\mathrm{dim}(X) - \mathrm{dim}(G) - \varepsilon_X}{4}}
\end{equation}
to be the global constant\footnote{In the notation of \cite[\S 10.4]{BZSV}, this is closely related to the constant $q^{-\beta_X}$, except we multiply by $|\eta_X(\partial^{1/2})|^{1/2}$ to simplify future calculations.} appearing in \eqref{equation theta series}, required to approximately $L^2$-normalize the theta series $\theta_X$. 

\begin{defn}[Automorphic period] \label{definition automorphic period}
    Let $f$ be an unramified automorphic form on $G$. The \textit{automorphic $X$-period of $f$} is defined as
    \begin{equation}
    P_X(f) := \int_{[G]} \, \theta_X(g) \, f(g) \, dg.
    \end{equation}
\end{defn}

The above definition of automorphic $X$-period can be regarded as a formal one, for which we make no promises of convergence. On the other hand, we observe that (modifying routinely the argument of Lemma 2.3 of \cite{CV}), we have the following.

\begin{lem} \label{lemma convergence of automorphic period}
Suppose that the center of $G$ acts \textit{conically}\footnote{By definition, an action by a torus $\Gm^n$ on some affine variety $X$ is \textit{conical} if the resulting action on $\mathcal{O}_X$ has nonnegative weights, and the 0th graded piece is isomorphic to the ground field.} on $X$. For $f$ a cusp form whose central character has real part large enough, the period $P_X(f)$ converges absolutely.
\end{lem}

The first hypothesis entails no essential loss for the automorphic integrals considered in the present discussion. For more general situations, see \S \ref{subsub central twists} below.

\subsection{Spectral periods}
Analogously to the automorphic side, a spectral period is defined as a certain numerical invariant of an $L$-parameter. To define it, suppose we have a $\check{G}$-valued $L$-parameter $\varphi$, and a graded $\check{G}$-variety $\check{X}$. Then $\Gamma_F$ acts on $\check{X}$ via $\varphi$, and for every point $x \in \check{X}$ fixed by this action, there is an induced action of $\Gamma_F$ on the completed ring of functions $\widehat{\mathcal{O}}_{\check{X},x}$ at the fixed point $x$. 

\begin{defn}[Spectral period] \label{definition spectral period}
    Let $\varphi$ be a $L$-parameter valued in $\check{G}$. Let $\check{X}$ be a graded $\check{G}$-variety with eigencharacter $\eta_{\check{X}}:\check{G} \to \Gm$. The \textit{spectral $\check{X}$-period of $\varphi$} is defined as
    \begin{equation}
        L_{\check{X}}(\varphi) := \mathfrak{z}_{\check{X}}(\varphi) \cdot  \Delta^{\frac{\varepsilon_{\check{X}} - \mathrm{dim}(\check{X})}{4}} \cdot \sum_{x \in \mathrm{Fix}(\varphi, \check{X})}  L_x(\check{X}, \varphi)
    \end{equation}
    where 
    \begin{itemize}
        \item $\varepsilon_{\check{X}} \in \Z$ is the $\Ggr$-weight on the eigenmeasure on $\check{X}$;
        \item $\mathfrak{z}_{\check{X}}(\varphi)$ is the value of $\eta_{\check{X}} \circ \varphi: \Gamma_F \to \kk^\times$ at $\partial^{-1/2}$, viewed as an element of $\Gamma_F^{\mathrm{ab}}$ via global class field theory; 
        \item for a fixed point $x$, we define
        $$L_x(\check{X}, \varphi) := \prod_v \, \mathrm{gtr}\left(\mathrm{Fr}_v \, \middle| \, \widehat{\mathcal{O}}_{\check{X}, x}\right)\footnote{The graded trace of $\mathrm{Fr}_v$ is defined as $\sum_{k \geq 0} \, q_v^{-k/2} \mathrm{tr}(\mathrm{Fr}_v \, |\widehat{\mathcal{O}}_{\check{X},x}[k])$ where $\widehat{\mathcal{O}}_{\check{X},x}[k]$ denotes the $k$th weighted piece.},$$
        which we term the \textit{spectral period at $x \in \check{X}$}.
    \end{itemize}
\end{defn}
Note that the transfer of $L$-parameters $\varphi \mapsto \eta_{\check{X}} \circ \varphi$ from $\check{G}$-valued parameters to $\Gm$-valued parameters is Langlands dual to restricting the central character of an automorphic form along the central cocharacter $\check{\eta}_{\check{X}}: [\Gm] \to [G]$ dual to $\eta_{\check{X}}$ (not to be confused with $\eta_X$). By our normalization of class field theory (see \S \ref{subsubsection CFT normalization}) we may calculate the scalar $\mathfrak{z}_{\check{X}}(\varphi_f)$ in terms of $f$ as
\begin{equation}
    \mathfrak{z}_{\check{X}}(\varphi_f) = \omega_f(\check{\eta}_{\check{X}}(\partial^{-1/2}))
\end{equation}
where $\varphi_f$ is the $L$-parameter of an automorphic form $f$ on $G$.

\begin{example}
    When $\check{X} = \mathbf{A}^n$ and the action of $\check{G}$ is linear arising from a representation $\rho : \check{G} \to \GL_n$ and $\Ggr$ acts by scaling, then \cite[Lemma 3.1]{CV} gives 
    $$L_{\check{X}}(\varphi) = \mathrm{det}(\rho)(\partial)^{-1/2} \cdot L(1/2,\varphi,\rho),$$
    where the latter is the value at $1/2$ of the Langlands $L$-function attached to $\rho$ and $\varphi$, provided that $L_{\check{X}}(\varphi)$ converges.
\end{example}

\begin{rem}[Convergence of spectral periods] \label{rem convergence of spectral period}
    In the case when $\check{X}$ is conical - which will cover all cases of interest for us - the convergence of the associated spectral periods is guaranteed by the discussion following \cite[Lemma 3.1]{CV}. 
\end{rem}

\subsection{Induction of periods} \label{subsection period induction}

Let $H \subset G$ be a reductive subgroup, and let $(Y, \omega_Y)$ be a graded $H$-variety. We consider an induction procedure which turns $Y$ into a graded $G$-variety. Let 
$$X = \mathrm{Ind}_H^G(Y) := Y \times^H G$$
be equipped with the $G \times \Ggr$-action
$$(y,g') \cdot (g, \lambda) := (y\lambda, g'g)$$
for $g, g' \in G, \lambda \in \Ggr$, and $y \in Y$. Note that $X \to H \backslash G$ is a $Y$-fiber bundle over the homogeneous space $H \backslash G$. 

While it is always possible to write down a volume form $\omega_X$ on $X$ which restricts on each $Y$-fiber to the volume form $\omega_Y$, it is not always possible to ensure that the volume form $\omega_X$ is a $G$\textit{-eigenform}. Indeed, this is only possible if $\eta_Y: H \to \Gm$ extends to an algebraic character $\eta_X: G \to \Gm$, in which case any such extension corresponds to a choice of volume form on $X$ restricting to $\omega_Y$ on each $Y$-fiber. On the other hand, for the purposes of calculating automorphic and spectral periods, it will become evident that only the existence of $\eta_X$ as a \textit{rational} character $\eta_X \in \mathrm{Hom}(G, \Gm) \otimes_\Z \Q$ such that $\eta_X|_H = \eta_Y$ is required. 

Indeed, note that on the automorphic side, the eigenform $\eta_X$ contributes only in the recipe defining $\theta_X$, where what we needed was the \textit{real-valued half-density} $|\eta_X|^{1/2}: G(\mathbb{A}) \to \R_{>0}$. This real valued character makes sense even if $\eta_X$ is only a rational character. 

Correspondingly on the spectral side, note that in defining the spectral period $L_X$ via Definition \ref{definition spectral period}, we only need $\eta_X$ to be a rational character. Indeed, $\eta_X$ participates only in the construction of the scalar $\mathfrak{z}_X(\varphi)$, and we can make sense of the scalar $\mathfrak{z}_X(\varphi) = (\eta_X \circ \varphi)(\partial^{-1/2})$ as a complex number with a choice of branch of the complex logarithm. 

When the choice of $\eta_X$ is unique as a rational character (as will always be the case for us), we will abuse terminology slightly to call $X = \mathrm{Ind}_H^G(Y)$, equipped with its commuting $\Ggr$-action and rational character $\eta_X \in \mathrm{Hom}(G,\Gm) \otimes_\Z \Q$ the \textit{graded $G$-variety induced from $Y$.} By definition, we have
$$\eta_X|_{H \times \Ggr} = \eta_Y.$$

The procedure of induction incurs predictable consequences on automorphic and spectral periods, most of which are known to experts. For the sake of completeness and for precise future reference, we record them as lemmas below.

\subsubsection{Induction of automorphic periods} The only aspect of induction which requires some care on the automorphic side is the induction of \textit{integral structures} which we always assume that $Y$ carries per our convention. To spell this out carefully, we write $\mathbf{G}, \mathbf{H}$ for our fixed Chevalley forms of the reductive groups $G,H$ respectively, and let $\mathbf{Y}$ denote a chosen integral structure for the $H$-action, i.e., we have an action $\mathbf{H} \times \mathbf{Y} \to \mathbf{Y}$ of $\mathfrak{o}_v$-schemes so that passing to the generic fiber gives the action of $H$ on $Y$. To construct an integral form of $X$, it suffices to choose an $\mathfrak{o}_v$-subalgebra of $\mathcal{O}_X$ stable under the $\mathfrak{o}_v$-linear $\mathcal{O}_{\mathbf{G}}$-coaction. We define simply $\mathcal{O}_{\mathbf{X}} := (\mathcal{O}_{\mathbf{Y}} \boxtimes \mathcal{O}_{\mathbf{G}})^{\mathbf{H}} $
and set
$$\mathbf{X} := \spec \mathcal{O}_{\mathbf{X}},$$
which we regard as the \textit{integral structure induced from $Y$.} For a dominant weight $\lambda$ of the maximal torus $T$ of $G$, we write $V_\lambda$ for the highest weight $G$-module of highest weight $\lambda$, and let $\mathbf{V}_\lambda \subset V_\lambda$ be the $\mathfrak{o}_v$-lattice defined by the integral form $\mathbf{G}$. Then the ring of functions on $\mathbf{X}$ can be presented explicitly as
\begin{equation} \label{equation integral functions on X}
    \mathcal{O}_{\mathbf{X}} \simeq \oplus_\lambda \, (\mathcal{O}_{\mathbf{Y}} \otimes \mathbf{V}_\lambda^*)^\mathbf{H} \boxtimes \mathbf{V}_\lambda
\end{equation}
using the Peter--Weyl decomposition of functions on $\mathbf{G}$, where $\lambda$ ranges through dominant weights of $T$.

\begin{lem} \label{lemma automorphic induction}
    Let $Y$ be a graded $H$-variety, where $H \subset G$ is a reductive subgroup satisfying the cohomological assumption 
    $$\mathrm{K}^1_{H,G} := \mathrm{Ker}\big(\mathrm{H}^1(F, H) \to \mathrm{H}^1(F, G)\big) = 0.$$
    Let $X = \mathrm{Ind}_H^G(Y)$ be the induction of $Y$ from $H$ to $G$. Then the automorphic $X$-period of an automorphic form $f$ on $G$ can be computed as
    \begin{equation} \label{equation automorphic induction}
        P_X(f) = \int_{[H]} \theta_Y(h)f(h) \, dh.
    \end{equation}
\end{lem}
\begin{proof}
    Note that the power of $\Delta$ involved in normalizing the theta series of $X$ and $Y$ are related by
    $$\frac{\mathrm{dim}(X) - \mathrm{dim}(G)}{4} = \frac{\mathrm{dim}(Y) - \mathrm{dim}(H)}{4}.$$
    Similarly, the prefactor involving $\partial^{1/2}$ in the definitions of $\theta_X$ and $\theta_Y$ can also be related simply as 
    $$|\eta_X(\partial^{1/2})|^{1/2} = |\eta_Y(\partial^{1/2})|^{1/2}$$
    because the induced $\Ggr$-action acts trivially on the Haar measure on $H \backslash G$. Thus, we will ignore these prefactors in the following. 
    
    By definition (ignoring the aforementioned prefactors), we compute directly that
    \begin{align*}
        P_X(f) &= \int_{[G]} \, f(g)\, |\eta_X(g)|^{1/2} \, \sum_{x \in X(F)} \,  \Phi_X(x \partial^{1/2}g) \, dg\\
        &= \int_{[G]} \, f(g) \,|\eta_X(g)|^{1/2} \sum_{(y,x) \in (Y(F) \times G(F))/H(F)} \, \Phi_X(y\, \partial^{1/2},xg) \, dg\\
        &= \sum_{\bm{y} \in Y(F) / H(F)} \int_{{\rm Stab}_{H(F)}(\bm{y}) \backslash G(\mathbb{A})} \, f(g)|\eta_X(g)|^{1/2} \Phi_X(\bm{y}\partial^{1/2},g) \, dg\\
        &= \int_{H(\mathbb{A}) \backslash G(\mathbb{A})} \sum_{\bm{y} \in Y(F) / H(F)} \int_{{\rm Stab}_{H(F)}(\bm{y}) \backslash H(\mathbb{A})}\!\! \!\!\!\!\!\!\!\!\!\!\!\!\!\!\!\!\!f(hg)|\eta_X(hg)|^{1/2} \Phi_X(\bm{y}\partial^{1/2},hg) \, dh \, dg\\
        &= \int_{H(\mathbb{A}) \backslash G(\mathbb{A})} \int_{[H]} \, f(hg)|\eta_X(hg)|^{1/2} \sum_{y \in Y(F)} \, \Phi_X(y\partial^{1/2}h,g) \, dh \, dg.
    \end{align*}
    Note that from the first to the second line, we relied on the Galois cohomological assumption on $H$ to interpret the $F$-rational points of $X$ as those of $Y \times G$ up to $H(F)$-action. Moreover, observe that the set of $G(F)$-orbits on $(Y(F) \times G(F))/H(F)$ is isomorphic to $Y(F)/ H(F)$. Thus, unfolding with respect to the $G(F)$-action explains the third line. 

    We would now like to evaluate the Schwartz function $\Phi_X$ in terms of $\Phi_Y$. The integrality of the point $(y\partial^{1/2}h,g)$ for $y \in Y(F), h \in H(\mathbb{A})$ and $g \in H(\mathbb{A}) \backslash G(\mathbb{A})$ is determined by the $\mathfrak{o}_v$-integrality (for all places $v$) of the evaluation of every function in $\mathcal{O}_{\mathbf{X}}$. To this end, we fix a place $v$ and appeal to equation \eqref{equation integral functions on X}. Select a dominant weight $\lambda$ for which there are nonzero $\mathbf{H}$-invariant vectors in $\mathcal{O}_\mathbf{Y} \otimes \mathbf{V}_\lambda^*$ among which we pick an arbitrary $\xi$, and take an arbitrary vector $\mathbf{v} \in \mathbf{V}_\lambda$. The matrix coefficient formed by $\xi$ (regarded as a function on $\mathbf{Y}$ valued in $\mathbf{V}_\lambda^*$) and $\mathbf{v}$ thus represent a generic function in $\mathcal{O}_{\mathbf{X}}$, so we examine the integrality of the expression
    \begin{equation} \label{equation integrality condition}
        \langle \xi(y\partial^{1/2}h),\mathbf{v}g\rangle = \langle \xi(y \partial^{1/2}) \cdot h, \mathbf{v}g\rangle \in \mathfrak{o}_v,
    \end{equation}
    for $y\partial^{1/2}h \in Y(F_v)$ and $ g \in H(F_v) \backslash G(H_v)$. First, we consider those $\xi$ that lie in $(\mathbf{V}_\lambda^*)^{\mathbf{H}}$; they may be regarded as constant functions on $\mathbf{Y}$ valued in $\mathbf{V}_\lambda^*$. Varying over $\lambda$ such that $(\mathbf{V}_\lambda^*)^{\mathbf{H}} \neq 0$, $\xi \in (\mathbf{V}_\lambda^*)^{\mathbf{H}}$, and $\mathbf{v} \in \mathbf{V}_\lambda$, we see that the integrality condition \eqref{equation integrality condition} reads
    $$g \longmapsto \langle \xi, \mathbf{v}g\rangle \in \mathfrak{o}_v$$
    which is exactly the condition that $g \in H(F_v)G(\mathfrak{o}_v)$. Returning to the integral, we thus obtain (again up to normalizing constants that we suppressed since the beginning)
    \begin{equation} \label{equation integrality step 1}
        P_X(f) = \int_{[H]} \, f(h)|\eta_X(h)|^{1/2} \sum_{y \in Y(F)}\Phi_X(y\partial^{1/2}h,1) \, dh
    \end{equation}
    by integrating out the region $g \in H(\mathbb{A})G(\mathfrak{o})$ in the outer integral. Thus we are reduced to considering the integrality of points of the form $(y\partial^{1/2}h,1) \in X(F_v)$. The integrality constraint \eqref{equation integrality condition} now reads as follows: for all $\lambda$ dominant, $\xi \in (\mathcal{O}_{\mathbf{Y}} \otimes \mathbf{V}_\lambda^*)^{\mathbf{H}}$, and $\mathbf{v} \in \mathbf{V}_\lambda$, we require
    $$\langle \xi(y\partial^{1/2}h), \mathbf{v}\rangle \in \mathfrak{o}_v.$$
    Note that the pairing $\mathbf{V}_\lambda^* \otimes \mathbf{V}_\lambda \to \mathfrak{o}_v$ induces a contraction $(\mathcal{O}_{\mathbf{Y}} \otimes \mathbf{V}_\lambda^*)^{\mathbf{H}} \otimes \mathbf{V}_\lambda \to \mathcal{O}_{\mathbf{Y}}$ which we denote by $\xi \otimes \mathbf{v} \mapsto f_{\xi \otimes \mathbf{v}}$. Then the previous condition is simply that $f_{\xi \otimes \mathbf{v}}(y\partial^{1/2}h) \in \mathfrak{o}_v$. First we observe that contraction is surjective, i.e., any function in $\mathcal{O}_{\mathbf{Y}}$ can be written as $f_{\xi \otimes \mathbf{v}}$ for some $\xi$ and $\mathbf{v}$. To see this, decompose $\mathcal{O}_{\mathbf{Y}}$ as an $\mathbf{H}$-representation and take a function $f \in \mathcal{O}_{\mathbf{Y}}$ which lies in some irreducible $\mathbf{H}$-module $\mathbf{W} \subset \mathcal{O}_{\mathbf{Y}}$. We then take any $\mathbf{G}$-module $\mathbf{V}_\lambda$ such that $\mathbf{V}_\lambda^*|_{\mathbf{H}}$ contains $\mathbf{W}$, and take two dual vectors $\mathbf{v}^* \in \mathbf{V}_\lambda^*$ and $\mathbf{v} \in \mathbf{V}_\lambda$ satisfying the equation $\langle \mathbf{v}^*, \mathbf{v}\rangle = 1$. Then $f \otimes \mathbf{v}^*$ is an $\mathbf{H}$-invariant vector of $\mathcal{O}_{\mathbf{Y}} \otimes \mathbf{V}_\lambda^*$, and $f$ is the contraction of $(f \otimes \mathbf{v}^*) \otimes \mathbf{v}$. Thus we see that the integrality of $f_{\xi \otimes \mathbf{v}}(y\partial^{1/2}h)$ for all $\xi$ and $\mathbf{v}$ is equivalent to the integrality of $f(y\partial^{1/2}h)$ for all $f \in \mathcal{O}_{\mathbf{Y}}$, which is in turn equivalent to the integrality of the point $y\partial^{1/2}h \in Y(F_v)$. 
    
    Finally we observe that by definition of the induced volume form on $X$, the eigencharacter in the above integrand can be evaluated as $|\eta_X(hg)|^{1/2} = |\eta_Y(h)|^{1/2}$ since the Haar measure on $H \backslash G$ is right $G$-invariant. Replacing $\eta_X$ and $\Phi_X(y\partial^{1/2}h,1)$ by $\eta_Y$ and $\Phi_Y(y\partial^{1/2}h)$ in equation \eqref{equation integrality step 1}, respectively, we obtain the statement of the lemma.
\end{proof}

\begin{rem} \label{remark unramified orbits}
    Suppose we relax the cohomological assumption $\mathrm{K}^1_{H,G}  = 0$. Then the first step in the above unfolding, in which we used the assumption that the rational points $X(F)$ are identified with the quotient $(Y(F) \times G(F))/H(F)$ is no longer valid. However, we do have an exact sequence (of pointed sets)
    $$(Y(F) \times G(F))/H(F) \longrightarrow X(F) \longrightarrow \mathrm{K}^1_{H,G}.$$
    In other words, in decomposing the set $X(F)$ into $H(F)$-orbits, there is a \textit{neutral orbit} corresponding to $(Y(F) \times G(F))/H(F)$, and the other orbits can be assigned a cohomological invariant in $\mathrm{K}^1_{H,G} \subseteq \mathrm{H}^1(F, H)$. Each orbit may contribute an extra term to Lemma \ref{lemma automorphic induction}, expressed as the automorphic period over an inner form of $H$.
\end{rem}

\begin{rem} \label{remark geometric interpretation of automorphic induction}
    The preceding lemma admits a geometric description, explained in \cite[\S 10.2]{BZSV}, through which its validity is evident. Recall that given an $H$-variety $Y$, one may consider the relative moduli space
    $$\mathrm{Bun}_H^Y(\Sigma) := \mathrm{Map}(\Sigma, [Y/H])$$
    over $\mathrm{Bun}_H(\Sigma)$ parametrizing $H$-bundles with an associated $Y$-section. The induced $G$-variety $X = Y \times^H G$ allows one to form
    $$\mathrm{Bun}_G^X(\Sigma) \simeq \mathrm{Bun}_H^Y(\Sigma)$$
    induced by the identification of stacks $[X/G] \simeq [Y/H]$. As a moduli space over $\mathrm{Bun}_G(\Sigma)$, the right hand side $\mathrm{Bun}_H^Y(\Sigma)$ parametrizes a principal $G$-bundle, a reduction of structure group to $H$, and an associated $Y$-section of the $H$-reduction. The period $P_X$ in Lemma \ref{lemma automorphic induction} may be geometrized as (i.e., is the function-sheaf partner of) the $!$-pushforward of the constant $\ell$-adic sheaf on $\mathrm{Bun}_G^X(\Sigma) \to \mathrm{Bun}_G(\Sigma)$, while the right hand side of equation \eqref{equation automorphic induction} is precisely geometrized by $\mathrm{Bun}_H^Y(\Sigma) \to \mathrm{Bun}_G(\Sigma)$.

    In the case when we drop the cohomological assumption $\mathrm{K}^1_{H,G} = 0$ as in the preceding remark, a similar geometric interpretation can be made. It suffices, without loss of generality, to consider the special case when $Y = \mathrm{pt}$ and $X = H \backslash G$: the period distribution $P_X$ counts reductions of a given $G$-bundle to \textit{pure inner forms} of $H$ that appear in $\mathrm{K}^1_{H,G}$.
\end{rem}

\subsubsection{Induction of spectral periods} We now turn our attention to the effect of induction on spectral periods. In classical language, a spectral period induced from $H$ to $G$ distinguishes $L$-parameters landing in $H$ (with possibly even smaller image). The following computation can be found in Lemma 2.7.1 of \cite{thesis}, but we reproduce it here for completeness and clarity. 
\begin{lem} \label{lemma spectral induction}
    Let $Y$ be a graded $H$-variety, and $H \subset G$ be a reductive subgroup. Let $X = \mathrm{Ind}_H^G(Y)$ be the induction of $Y$ from $H$ to $G$. Then the spectral $X$-period of an unramified $G$-valued $L$-parameter $\varphi: \Gamma_F \to G$ can be computed as
    \begin{equation} \label{equation spectral induction}
        L_X(\varphi) = \begin{cases}
        \, 0 \, \text{ if } \mathrm{Im}(\varphi) \text{ is not in a conjugate of }H, \\
        \, \Delta^{\frac{\mathrm{dim}(H)-\mathrm{dim}(G)}{4}} \cdot L(1, \varphi, \mathfrak{g}/\mathfrak{h}) \cdot\sum_{c \in C_\varphi}\,  L_Y(\varphi^c) \text{ if } \mathrm{Im}(\varphi) \text{ is in a conjugate of }H,
    \end{cases}
    \end{equation}
    where we have used the following notation:
    \begin{itemize}
        \item $C_\varphi$ parametrizes equivalence classes of lifts of the parameter $\varphi$ to $H$, i.e., 
        $$C_\varphi := \big\{g \in G: \mathrm{Ad}_g(\mathrm{Im}(\varphi)) \subset H\big\}/H,$$
        \item for an element $c \in C_\varphi$ we wrote $\varphi^c$ for the $c$-conjugate of $\varphi$,
        \item and $L(1,\varphi, \mathfrak{g}/\mathfrak{h})$ is the Langlands $L$-function
    $$L(1,\varphi, \mathfrak{g}/\mathfrak{h}) = \prod_v \, \frac{1}{\mathrm{det}(1- \varphi(\mathrm{Fr}_v)q_v^{-1}|\mathfrak{g}/\mathfrak{h})}.$$
    \end{itemize}
\end{lem}
As usual, we regard the right hand side of \eqref{equation spectral induction} as invalid if it fails to converge for any reason: if $|C_\varphi|$ is infinite, if $L(s,\varphi, \mathfrak{g}/\mathfrak{h})$ has a pole at $s = 1$, or if $\varphi$ has infinitely many fixed points on $Y$.

\begin{proof}
    By definition, the spectral $X$-period of $\varphi$ is 
    $$L_X(\varphi) = (\eta_X \circ \varphi)(\partial^{-1/2}) \Delta^{\frac{\varepsilon_X-\mathrm{dim}(X)}{4}}\sum_{x \in \mathrm{Fix}(\varphi,X)} \, L_x(X, \varphi).$$
    To evaluate the preceding expression from first principles, let us write down explicitly the definition of $\mathrm{Fix}(\varphi, X)$: a fixed point for the $\Gamma$ action on $X$ via $\varphi$ is the data of a pair $(y, g) \in Y \times G$ and for every $\gamma \in \Gamma$ an element $h_\gamma \in H$ satisfying the equation
    \begin{equation} \label{equation fixed point equation}
        (yh_\gamma, h_\gamma^{-1} g) = (y, g \varphi(\gamma)).
    \end{equation}
    Note that if one is able to find a solution to \eqref{equation fixed point equation}, then (i) the association $\gamma \mapsto h_\gamma$ is a group homomorphism $\Gamma \to H$, (ii) we have $h_\gamma = g \varphi(\gamma)^{-1} g^{-1}$ so that the homomorphism of (i) is obtained from $\varphi$ by taking a $g$-conjugate, and (iii) the point $y \in Y$ is a fixed point under the action of $\varphi^g$. Assuming that $\varphi$ is indeed conjugate to a homomorphism with image inside $H$ so that $\mathrm{Fix}(\varphi, X)$ is nonempty, we may further assume without loss of generality that $\varphi$ itself has image contained in $H$ by replacing $\varphi$ by a conjugate. Then we may compute from \eqref{equation fixed point equation} the following description of the fixed point set:
    $$\mathrm{Fix}(\varphi,X) = \bigsqcup_{c \in C_\varphi} \, \mathrm{Fix}(\varphi^c, Y_c)$$
    where $c$ ranges through $C_\varphi$ (which we assume to be finite), and $Y_c$ is the fiber above $Hc \in H \backslash G$, on which $\varphi^c$ acts (note that $Y_c \simeq Y$ as varieties for all $c \in C_\varphi$).

    Since $\eta_X|_{H \times \Ggr} = \eta_Y$, with the preceding observations in mind we may rewrite the spectral $X$-period for an $L$-parameter $\varphi$ with image valued in $H$ as
    \begin{equation} \label{equation spectral induction step 1}
        (\eta_Y \circ \varphi)(\partial^{-1/2})\Delta^{\frac{\varepsilon_Y - \mathrm{dim}(X)}{4}} \sum_{c \in C_\varphi} \, \sum_{y \in \mathrm{Fix}(\varphi^c, Y_c)} \, L_{(y,c)}(X, \varphi^c).
    \end{equation}
   We concentrate on the contribution at the $c = \mathrm{id}$ coset and evaluate the expression $L_{(y,\mathrm{id})}(X, \varphi)$; all other $c \in C_\varphi$ lead to a similar conclusion. As a graded $H$-representation, the completed local ring at $(y, \mathrm{id}) \in X$ can be computed as
    $$\widehat{\mathcal{O}}_{X,(y,\mathrm{id})} \simeq \widehat{\mathcal{O}}_{Y,y} \,  \widehat{\otimes} \, \widehat{\mathcal{O}}_{H \backslash G, \mathrm{id}} \simeq \widehat{\mathcal{O}}_{Y,y} \, \widehat{\otimes} \, \widehat{\mathrm{Sym}}((\mathfrak{g}/\mathfrak{h})^*)$$
    where $\widehat{\mathcal{O}}_{Y,y}$ is graded via the grading on $Y$, and $\widehat{\mathrm{Sym}}((\mathfrak{g}/\mathfrak{h})^*)$ is graded by polynomial degree. The graded trace of a tensor product being a product of graded traces, we see that
    $$L_{(y, \mathrm{id})}(X, \varphi) = L_y(Y, \varphi) \cdot L_0(\mathfrak{g}/\mathfrak{h}, \varphi)$$
    where $0 \in \mathfrak{g}/\mathfrak{h}$ is the origin. Returning to \eqref{equation spectral induction step 1}, and using the dimensional relation $\mathrm{dim}(X) = \mathrm{dim}(Y) + \mathrm{dim}(G) - \mathrm{dim}(H)$, we have exactly the desired expression in the statement of the lemma.
\end{proof}

\subsection{Central twists}\label{subsub central twists}

 For the sake of conceptual completeness, we discuss in some detail the normalization scheme presented in \cite[\S 4.1]{CV} and point out some subtleties. While not strictly necessary, it is nevertheless motivating in order to follow the constructions of \S \ref{subsec JS period} and \S \ref{subsec:stackyJS}.
 
 Let $G$ be a reductive group, with a 1-dimensional central torus $z: \Gm \subset G$. Let $X$ be a spherical $G$-variety from which we have constructed a theta function $\theta_X$ and the corresponding period  $P_X(f)$
for a cuspidal automorphic form $f$ on $G$. We let $\delta: G \to \Gm$ be a primitve character. In applications one often considers central twists of $P_X$ of the form
\begin{equation}\label{equation PX with twist}
    f \times \chi \longmapsto \int_{[G]} f(g)\theta_X(g)\chi(\delta(g)) \, dg,
\end{equation}
where $\chi: [\Gm] \to \kk$ is a varying adelic character. Formally we may factor out the central $[\Gm]$-integral to obtain the expression
\begin{equation} \label{equation factoring out the center with twist}
     \bigg[\int_{[\Gm]} \omega_f(z)\chi^{\langle z,\delta\rangle}(z) \, dz\bigg]\bigg[\int_{Z_G(\mathbb{A})\backslash [G]} \, f(g)\theta_X(g)\chi(\delta(g)) \, dg\bigg].
\end{equation}
where the $[\Gm]$-integral distinguishes\footnote{One may interpret this divergent expression representation-theoretically, as is common in the renormalization of automorphic periods. In this case, the expression $\int_{[\Gm]}\omega_f(z)\chi^{\langle z,\delta\rangle}(z) \, dz$ is regarded as a vector in $\mathrm{Hom}_{\Gm(\mathbb{A})}(\omega_f \chi^{\langle z,\delta\rangle}, \mathbf{1})$, which vanishes unless $\omega_f\chi^{\langle z,\delta\rangle} = \mathbf{1}$.} central characters of the form $\omega_f\chi^{\langle z,\delta\rangle} = \mathbf{1}$, in which case the integrand $f(g)\theta_X(g)\chi(\delta(g))$ descends to the quotient $Z_G(\mathbb{A})\backslash [G]$ and the second bracketed integral makes sense; otherwise, $P_{X}(f \times \chi) = 0$. In the case of distinction, the $[\Gm]$-integral contributes a divergent factor of $\zeta(1)$. To address this convergence issue, we thus pass to the action $G^{\rm ad} \acts X$, which has the effect of formally canceling the divergent factor $\zeta(1)$.

In terms of (hyper)spherical actions, we may regard the period \eqref{equation PX with twist} as encoded by the \textit{Hamiltonian induction} of $M_0 := T^*X$ along the inclusion 
\begin{equation}
    G \hookrightarrow G \times \Gm \text{ by } g \mapsto (g, \delta(g)).
\end{equation}
from which we obtain the Hamiltonian $G \times \Gm$-space
$$M_1 := \mathrm{hInd}_{G}^{G \times \Gm}(T^*X) := T^*X \times_{\mathfrak{g}^*}^G T^*(G \times \Gm) \simeq T^*(X \times^G (G \times \Gm)).$$
Indeed, setting $X_1 = X \times^G (G \times \Gm)$, one verifies readily that the formula \eqref{equation PX with twist} is given by the period $P_{X_1}$ on $G \times \Gm$. To obtain a numerically convergent expression, one considers the \textit{Hamiltonian reduction} by the central $\Gm$:
$$M_{\mathrm{tw}} := \{0\} \times_{\mathfrak{g}_{m,z}^*}^{\mathbf{G}_{m,z}}M_1$$
where we denoted by $\mathbf{G}_{m,z}$ the central torus (to distinguish it from the recipient of the character $\delta$). The resulting symplectic space $M_{\mathrm{tw}}$ is a Deligne--Mumford stack equipped with a Hamiltonian action of the twisted group 
$$G_{\mathrm{tw}} := (G \times \Gm)/\mathbf{G}_{m,z}.$$
Note that a cusp form $f \times \chi$ on $G \times \Gm$ as above with $\omega_f \chi^{\langle z, \delta\rangle} = \mathbf{1}$ naturally descends to a cusp form on $G_{\mathrm{tw}}$, and the automorphic $M_{\mathrm{tw}}$-period of such an $f \times \chi$ is essentially the second bracketed period in \eqref{equation PX with twist}.

Now suppose we have knowledge of the hyperspherical dual $\check{G} \acts \check{M}$, and we seek to construct the Hamiltonian space whose spectral period ought to encode \eqref{equation PX with twist}, i.e., we would like to construct the hyperspherical dual $\check{M}_{\mathrm{tw}}$ of $M_{\mathrm{tw}}$. Then we first consider the Langlands dual Hamiltonian induction of $\check{M}$ along the inclusion 
$$\check{G} \hookrightarrow \check{G} \times \check{\mathbf{G}}_m \text{ by } g \mapsto (g,z(g))$$
where we now regard the central cocharacter $z$ of $G$ as a character of $\check{G}$. From this we obtain 
$$\check{M}_1 := \mathrm{hInd}_{\check{G}}^{\check{G} \times \check{\mathbf{G}}_m}(M)$$
as the hyperspherical dual to $G \times \Gm \acts M_1$. To obtain $\check{M}_{\mathrm{tw}}$, the hyperspherical dual of $M_{\mathrm{tw}}$, first note that the Langlands dual group $\check{G}_{\mathrm{tw}}$ of $G_{\mathrm{tw}}$ is naturally presented as the kernel in the following short exact sequence
$$1 \longrightarrow \check{G}_{\mathrm{tw}} \longrightarrow \check{G} \times \Gm \overset{z}{\longrightarrow} \Gm \longrightarrow 1.$$
To extract from $\check{M}_1$ the Hamiltonian $\check{G}_{\mathrm{tw}}$-action hyperspherical dual to $M_{\mathrm{tw}}$, one needs to recognize that $\check{M}_1$ is a Hamiltonian induction of some $\check{G}_{\mathrm{tw}} \acts \check{M}_{\mathrm{tw}}$ (see \cite[\S 3.1.1]{BZSV}). 

Note crucially that neither $G_{\mathrm{tw}} \acts M_{\mathrm{tw}}$ nor $\check{G}_{\mathrm{tw}} \acts \check{M}_{\mathrm{tw}}$ are necessarily hyperspherical: while they do satisfy the dimension estimate (i.e., that there exists a Borel subgroup with a dense orbit in the Lagrangian base of the hyperspherical action), they may be Deligne--Mumford stacks and/or have disconnected generic stabilizers.  

\subsection{Numerical duality}

In \S 4 of \cite{CV}, a notion of automorphic and spectral period duality was proposed, generalizing the numerical conjectures of \cite{BZSV}, to encompass certain singular reductive group actions. On the one hand, the definitions of \textit{loc. cit} apply \textit{a priori} to arbitrary varieties (modulo convergence issues); on the other hand, they were envisioned to be most useful in treating automorphic integrals in the existing literature which do not admit a simple description in terms of hyperspherical varieties, but which do present themselves as attached to certain singular conical varieties. 

Besides placing emphasis on numerical consequences (as opposed to verifying a geometric, or sheaf theoretic correspondences relying on versions of the geometric Langlands correspondence), the duality proposed in \cite{CV} requires certain period identities to hold only on the \textit{cuspidal part} of the automorphic spectrum. This is why we deem the proposed duality a \textit{weak numerical duality}, and we review this notion here for the reader's convenience. 

\begin{defn} \label{definition weak numerical duality}
    We say that a pair $(G_1,X_1), (G_2,X_2)$ (defined over $\Q$) are \textit{numerically weakly dual} if, for a suitable $N$, there are integral models $\mathbf{X}_1, \mathbf{X}_2$ over $\Z[1/N]$ such that for any finite field $\F$ of size $q$ with characteristic not dividing $N$ and for any projective smooth curve $\Sigma$ over $\F$ with discriminant $\Delta$, we have equalities
    \begin{equation} \label{equation weak numerical duality}
        P_{X_1}(f_1) = \Delta^{a_{12}/4}L_{X_2}(\varphi_1) \, \text{ and } \, P_{X_2}(f_2) = \Delta^{a_{21}/4} L_{X_1}(\varphi_2),
    \end{equation}
    where:
    \begin{itemize}
        \item $f_1$ (resp. $f_2$) is any Whittaker normalized, everywhere unramified cuspidal tempered automorphic form on $G_1$ (resp. $G_2$) with Langlands parameter $\varphi_1$ valued in $G_2$ (resp. $G_1$).
        \item $a_{12}, a_{21} \in \Z$ are integers known as the \textit{discrepancy} of the pair $(G_1,X_1)$ and $(G_2,X_2)$.
    \end{itemize}
\end{defn}
When $X_1$ and $X_2$ are smooth, we generally expect that numerical duality should hold \textit{without discrepancies}; in other words, it is a correction term expected for cases of singular duality. While the precise significance of these discrepancies is not completely clear at the moment, see \S \ref{subsection discrepancy} for a discussion of our current understanding.  

\begin{rem}[Convergence in numerical duality statements] \label{remark convergence in numerical duality}
    We explain briefly how to handle the convergence issues in the preceding definition, which encompasses the main example considered here (and those of \cite{CV}).

    By Lemma \ref{lemma convergence of automorphic period}, if $X_1, X_2$ are \textit{conical} actions, then $P_{X_1}(f_1)$ and $P_{X_2}(f_2)$ converge absolutely whenever the real part of the central characters of $f_1, f_2$ are large enough, respectively. Implicitly, we may let $z$ be this real parameter. It is established separately (Remark \ref{rem convergence of spectral period}) that the spectral sides of \eqref{equation weak numerical duality} converge whenever $z \gg 1$, and one must interpret this equation in the sense of analytic continuation at $z = 1/2$. In our current examples, we will in fact have analytic continuation for $z \in \C$, but in general we expect a natural boundary at $z > 0$ for nonabelian $L$-functions. 
\end{rem}
\begin{rem}[Hyperspherical duality and numerical duality]
    The protagonists of \cite{BZSV} are \textit{hyperspherical varieties}: certain smooth affine Hamiltonian actions which are small relative to the acting group. While it is not at all necessary to understand them to parse the results here, we provide a brief discussion of their relation to the preceding Definition \ref{definition weak numerical duality} lest the reader be confused with the two notions. 

    In \textit{op. cit.}, a pair of \textit{hyperspherical dual} actions of a pair of Langlands dual groups $G_1 \acts M_1$ and $G_2 \acts M_2$, conjecturally gives rise to a myriad of matching structures predicating on varying contexts of Langlands duality. In particular, at the numerical level, the authors of \textit{op. cit.} proposed numerical conjectures \cite[\S 14]{BZSV}, which lead to the two equations in \eqref{equation weak numerical duality} without discrepancy, assuming that $M_1 = T^*X_1$ and $M_2 = T^*X_2$ are both polarized. In particular, every hyperspherical dual pair leads (conjecturally) to a numerically dual pair, while the objects discussed by Definition \ref{definition weak numerical duality} are much more general than hyperspherical actions. 
\end{rem}

\subsection{Extensions to Deligne--Mumford stacks} \label{subsection DM stacks}

An extension of weak numerical duality to certain Deligne--Mumford stacks was proposed in \cite{toric}, when the reductive group in question is a torus. Such an extension was introduced with the intention of encompassing period formulae that exhibit the following behavior: 
\begin{itemize}
    \item the automorphic integral unfolds to a \textit{finite sum} of Eulerian integrals, and
    \item the spectral period computes a \textit{finite sum} of (nonabelian) $L$-values.
\end{itemize}
These two phenomena are evidently one and the same, regarded from the automorphic or the Galois side. In \textit{op. cit.}, it was proposed that we can encode this situation as an extended relative Langlands duality:
\begin{itemize}
    \item the (Hamiltonian) action on the automorphic side has generically disconnected stabilizers, being Langlands dual to
    \item Deligne--Mumford stabilizers of the (Hamiltonian) action on the Galois side. 
\end{itemize}
In both cases -- having generically disconnected stabilizers or Deligne--Mumford stabilizers -- disqualifies a (Hamiltonian) action from being (hyper)spherical. Nonetheless, we observe, following \textit{op. cit.}, that at the numerical level (and the geometrical level), it is straightforward to generalize the definition of automorphic and spectral periods. 

We will avoid exceedingly general considerations, and simply focus on the following situation which seems to us an appropriate generality for applications to our program. Let $(Y, \omega_Y)$ be a graded $H$-variety for a reductive group $H$, and suppose we are given a group homomorphism $\iota: H \to G$ which is a \textit{central isogeny onto a subgroup}. We consider the induction from $H$ to $G$ of $Y$ via the same formula as in \S \ref{subsection period induction}:
$$X := \mathrm{Ind}_H^G(Y) := [(Y \times G)/H].$$
which is again a $Y$-fiber bundle over $[H \backslash G]$. Alternatively, if we write $\mu := \mathrm{ker}(\iota)$ and $\overline{H} := \mathrm{im}(\iota) \subset G$, we may present $X$ as
\begin{equation}\label{equation two presentations}
    X  \simeq [Y/\mu] \times^{\overline{H}} G.
\end{equation}

We pay some explicit attention to the computation of the eigencharacter associated to an eigenvolume form\footnote{Here we will only use volume forms on Deligne--Mumford stacks. Recall that the tangent complex of the smooth locus $X^{\mathrm{sm}}$ of $X$ is \textit{underived} and may be identified with the tangent space of its coarse quotient via the (\'etale) quotient map. Consequently, the sheaf of volume forms on $X^{\mathrm{sm}}$ is identified with the pullback of that on the coarse quotient.}. Taking the formula of $\omega_X$ \textit{verbatim} as in \S \ref{subsection period induction}, the action of an element $h \in \overline{H} \subset G$ on the fiber over the identity coset in $\overline{H}\backslash G$ is given by $y \cdot h = y \widetilde{h}$, where $\widetilde{h}$ is an arbitrary lift of $h$ to $H$, so we have
\begin{equation} \label{equation eigencharacter on cover}
    |\eta_X(h)| = |\eta_Y(\widetilde{h})| \text{ for } h \in \overline{H}.
\end{equation}

We note that Definitions \ref{definition automorphic period} and \ref{definition spectral period} extend almost \textit{verbatim} to $X$, and so do the statements of Lemmas \ref{lemma automorphic induction} and \ref{lemma spectral induction}. For the sake of clarity, let us spell out explicitly these consequences.

\subsubsection{Automorphic periods}
We briefly recall the geometric setup leading to BZSV's period sheaves as in \cite[\S 10]{BZSV} in order to motivate our definition. Given an arbitrary $G$-space $X$, one forms the following relative mapping stacks\footnote{We will ignore subtleties arising from twists by spin structures and normalizations for this informal discussion, focusing only on the important qualitative aspects.}
\begin{equation}\label{equation BunGX}
p: \mathrm{Bun}_G^X(\Sigma) := \mathrm{Map}(\Sigma, [X/G]) \longrightarrow \mathrm{Bun}_G(\Sigma).
\end{equation}
The $\theta$-function $\theta_X$ is geometrized by the BZSV period sheaf $\Theta_X := p_!\overline{\Q}_{\ell}$ (regarded as an $\ell$-adic sheaf on $\mathrm{Bun}_G(\Sigma)$), and the automorphic period $P_X$ is geometrized by the Hom-functor $\mathrm{Hom}(\Theta_X, -)$. But note that $p$ factors as 
$$\mathrm{Bun}_G^X(\Sigma) \simeq \mathrm{Bun}_H^Y(\Sigma) \overset{p'}{\to} \mathrm{Bun}_H(\Sigma) \overset{\iota}{\to}\mathrm{Bun}_G(\Sigma)$$
in the case when $X = \mathrm{Ind}_H^G(Y)$, so the aforementioned Hom-functor can be computed by adjunction as
$$\mathrm{Hom}(\Theta_X, -) \simeq \mathrm{Hom}(\Theta_Y, \iota^*(-)).$$
Taking Frobenius traces, we recover function-theoretic analogues of $\Theta_X, \Theta_Y$ as distributions on the groupoid of $\F$-rational points of $\mathrm{Bun}_G, \mathrm{Bun}_H$, respectively. If the groups in consideration are \textit{connected}, then by Weil's uniformization we may express these groupoids as the unramified adelic quotients $[G],[H]$, hence the formulation of Definition \ref{definition automorphic period}. 

However, if $H$ is disconnected (as it will be in our main cases of interest), it is no longer the case that $[H]$ is identified with $\mathrm{Bun}_H(\F)$ since a finite \'etale cover may not be trivializable on a Zariski cover, and it is the description of the following definition that remains faithful to the geometric interpretation of automorphic periods according to the setup of \eqref{equation BunGX}. 

\begin{defn} \label{definition DM automorphic period}
    Let $H, Y, \iota: H \to G$, and $X = \mathrm{Ind}_H^G(Y)$ be as above.  We assume that the image $\overline{H} \subset G$ of the isogeny $\iota$ satisfies the Galois cohomological condition $\mathrm{K}^1_{\overline{H},G} = 0$. For an automorphic form $f$ on $G$, we define the automorphic $X$-period of $f$ as
    \begin{equation}
        P_X(f) := \int_{\mathrm{Bun}_H(\F)} \, f(\iota(h)) \, \theta_Y(h) \, dh,
    \end{equation}
    where for an $H$-bundle $h \in \mathrm{Bun}_H(\F)$, the theta function $\theta_Y$ is evaluated as
    $$\theta_Y(h) := C_Y \cdot |\mathrm{Sect}(\Sigma, Y \times^{H \times \Ggr} hK^{1/2})(\F)| \, |\eta_Y|^{1/2}(h),$$
    with global constant $C_Y = \Delta^{\frac{\mathrm{dim}(Y) - \mathrm{dim}(H) - \varepsilon_Y}{4}}$ as in Definition \ref{definition automorphic period}, and $dh$ is the groupoid measure on $\mathrm{Bun}_H(\F)$ divided by the number of connected components of $\mathrm{Bun}_H$.
\end{defn}
The definition is compatible with Lemma \ref{lemma automorphic induction} when $\iota$ is an inclusion: indeed, when $H$ is connected, under Weil's uniformization $\mathrm{Bun}_H(\F) \simeq [H]$, the two definitions of $\theta_Y$ are compatible. Indeed, the theta series $\theta_Y$ of \eqref{equation theta series} can be regarded as the counting function of everywhere regular rational sections of the associated bundle $Y \times^{H \times \Ggr}hK^{1/2}$, hence the counting function of sections.

\begin{rem}
    The Galois cohomological assumption $\mathrm{K}^1_{\overline{H},G} = 0$ is made in order to avoid conflicting definitions in the case when $\iota$ is an inclusion but $\mathrm{K}^1_{\overline{H},G} \neq 0$. Indeed, in this case Definitions \ref{definition automorphic period} and \ref{definition DM automorphic period} would both apply, but they do not agree.
\end{rem}

In fact, the most important case for us will be when $\iota$ is a \textit{split isogeny}, in the sense that $H = \overline{H} \times \mu$ for some finite abelian group $\mu$, and $\iota: H \to \overline{H} \subset G$ is the projection composed with inclusion, with $\overline{H}$ connected.
\begin{lem}\label{example DM stack periods split isogeny}
    Suppose the isogeny $\iota : H \to G$ splits as $\iota: H \simeq \overline{H} \times \mu \subset G$, for some finite abelian group $\mu$. Then
    $$P_X(f) = \int_{[H]} \, f(\iota(h))\theta_Y(h) \, dh.$$
\end{lem}
\begin{proof}
    Definition \ref{definition DM automorphic period} indicates that the associated period is an integral over $$(\mathrm{Bun}_{\overline{H}}\times \mathrm{Cov}_\mu)(\F)  \simeq [\overline{H}] \times \mathrm{Cov}_\mu(\F),$$ where $\mathrm{Cov}_\mu$ denotes the (discrete) moduli stack of unramified Galois $\mu$-covers of $\Sigma$. Since the automorphic form $f$ in question is pulled back along $\iota$, it is constant in the $\mathrm{Cov}_\mu(\F)$-direction.  Hence, the automorphic period reduces to 
    \begin{equation} \label{equation sum over Covmu}
        P_X(f) = \frac{1}{|\mathrm{Cov}_\mu(\F)|}\sum_{\alpha \in \mathrm{Cov}_\mu(\F)} \int_{[\overline{H}]} \, f(h) \, \theta_Y(h \times \alpha) \, dh =  \int_{[\overline{H}]} \, f(h) \, \theta_Y(h \times 1) \, dh,
    \end{equation}
    since all the $\alpha$'s are necessarily integral, where now both the $dh$'s that appear in the preceding equation are groupoid measures. Now, noticing that $\iota$ induces an isomorphism $[H] \simeq [\overline{H}]$ of groupoids (since the unramified adelic quotient \eqref{unramified adelic quotient} of finite groups are trivial), we can rewrite the integral over $[H]$ as in the statement of the lemma.
\end{proof}

\begin{rem}
One may wish to obtain a description of the period $P_X$ independent of the induction presentations of \eqref{equation two presentations}. If one regards $X$ as obtained from induction along the subgroup $\overline{H} \subset G$, one is led to write down
$$P'_X(f) := \int_{[\overline{H}]} \, f(h) \, \theta_{[Y/\mu]}(h) \, dh,$$
where the definition of the theta series $\theta_{[Y/\mu]}$ requires some explanation. We can copy \eqref{equation theta series} to write
\begin{equation} \label{equation stacky theta series}
    \theta_{[Y/\mu]}(h) :=  C_{[Y/\mu]} \cdot \sum_{x \in [Y/\mu](F)} \, \Phi(x \cdot h \partial^{1/2})|\eta_{[Y/\mu]}|^{1/2}(h),
\end{equation}
and note that the $F$-points $x$ of $[Y/\mu]$ can be partitioned into isomorphism classes of \'etale $\mu$-algebras $K/F$, each contributing $Y(K)^\mu$ where $\mu$ acts by the Galois action on $K$ and its defining action on $Y$. A meaningful way to renormalize the infinite sum over \eqref{equation stacky theta series} as written is to consider only those \'etale $\mu$-algebras that are unramified over $F$, as done in \cite{toric}. With this regularization, since unramified \'etale $\mu$-algebras are in one-to-one correspondence with $\mu$-covers of $\Sigma$, we would form the \textit{regularized theta series}
$$\theta^{\mathrm{reg}}_{[Y/\mu]}(h) := C_{[Y/\mu]} \cdot \sum_{\alpha \in \mathrm{Cov}_\mu(\F)} \, \Phi(\alpha \cdot h \partial^{1/2})|\eta_{[Y/\mu]}|^{1/2}(h).$$
Using $\theta_{[Y/\mu]}^{\mathrm{reg}}$ instead of $\theta_{[Y/\mu]}$ in the definition of $P'_X$, we obtain a regularized period that coincides with $P_X$ thanks to \eqref{equation sum over Covmu}.
\end{rem}

\subsubsection{Spectral periods}
Analogously on the Galois side, given an arbitrary $G$-space $X$ one forms the relative mapping stack
\begin{equation}\label{equation LocGX}
q:\mathrm{Loc}_G^X(\Sigma) := \mathrm{Map}(\Sigma_{\mathrm{dR}}, [X/G]) \longrightarrow \mathrm{Loc}_G(\Sigma).
\end{equation} 
The spectral period $L_X$ is geometrized by the BZSV $L$-sheaf $\mathcal{L}_X := q_*\omega_{\mathrm{Loc}_G^X}$ (regarded as an ind-coherent sheaf on $\mathrm{Loc}_G(\Sigma)$), and the $L$-value of a given $L$-parameter is geometrized by the Hom-functor $\mathrm{Hom}(-, \mathcal{L}_X)$. Since $q$ factors as
$$\mathrm{Loc}_G^X(\Sigma) \simeq \mathrm{Loc}_H^Y(\Sigma) \overset{q'}{\to}\mathrm{Loc}_H(\Sigma) \overset{\iota}{\to} \mathrm{Loc}_G(\Sigma),$$
the aforementioned Hom-functor can be computed by adjunction as
$$\mathrm{Hom}(-, \mathcal{L}_X) \simeq \mathrm{Hom}(\iota^*(-), \mathcal{L}_Y).$$
Let $\varphi \in \mathrm{Loc}_G$ be an $L$-parameter, and let $\delta_{\varphi}$ be the skyscraper sheaf supported at $\varphi$\footnote{We will assume for simplicity that the centralizer of the image of $\varphi$ is trivial in this motivational discussion.}. Evaluating the preceding formula on $\delta_{\varphi}$, we see that the spectral $X$-period of $\varphi$ can be computed first by applying $\iota^*$ to obtain $\iota^*\delta_\varphi = \oplus_{\widetilde{\varphi}: \iota \circ \widetilde{\varphi} = \varphi} \, \delta_{\widetilde{\varphi}}$, and then by calculating the usual $Y$-spectral period for each lift $\widetilde{\varphi}$. These geometric considerations motivate the following definition.

\begin{defn}\label{defn_spectral_period_DM_stack}
 Let $H, Y, \iota: H \to G$, and $X=\mathrm{Ind}_H^G(Y)$ be as above. The spectral $X$-period of a $G$-valued $L$-parameter $\varphi: \Gamma_F \to G$ distinguishes parameters with image in a conjugate of $\overline{H}$, and in the distinction case is defined as
        \begin{equation}
            L_X(\varphi) = \Delta^{\frac{\mathrm{dim}(H)-\mathrm{dim}(G)}{4}} \cdot L(1,\varphi, \mathfrak{g}/\overline{\mathfrak{h}}) \cdot \sum_{\substack{\widetilde{\varphi}: \Gamma_F \to H \\ \text{ such that } \iota  \circ \widetilde{\varphi} = \varphi}} \mathfrak{z}_Y(\widetilde{\varphi}) \cdot L_Y(\widetilde{\varphi}),
        \end{equation}
    where $\widetilde{\varphi}$ ranges through equivalence classes of \textit{unramified} lifts of $\varphi$ to $H$.    
\end{defn}
This definition is compatible with Lemma \ref{lemma spectral induction} in the case when $\iota$ is an inclusion, and in fact agrees \textit{verbatim} with Definition \ref{definition spectral period} if one interprets the fixed points $\mathrm{Fix}(\varphi, X)$ in a stack-theoretic sense. In more classical terms, in the distinction case $L_X$ evaluates to a sum over \textit{unramified} $\mu = \mathrm{Ker}(\iota)$-twists of an $L$-function for $\overline{H}$.

\section{The Shalika model} \label{section Shalika model}

\subsection{The Jacquet--Shalika period}\label{subsec JS period}

We start by recalling a hyperspherical dual pair proposed by \cite[Example 4.5.1]{BZSV}, based on the pair of Langlands dual groups $G=\mathrm{GL}_{2n}$ and $\check{G}=\mathrm{GL}_{2n}$. Consider the unipotent subgroup
$$U = \bigg\{\begin{bmatrix} 1 & x\\ 0 & 1\end{bmatrix} \text{ for } x \in \mathrm{Mat}_n\bigg\} \subset G$$
with the trace character $\Psi: [U] \to \C$ given by $$\begin{bmatrix} 1 & x\\ 0 & 1\end{bmatrix} \mapsto \psi({\rm Tr}(x)).$$ We write $2\rho_\Sh$ for the sum of positive roots appearing in the unipotent subgroup $U$. Also, let $\mathrm{GL}_n^\Delta$ denote the subgroup of $G$ given by $$ \bigg\{\begin{bmatrix} g & 0 \\ 0 & g\end{bmatrix} \text{ for } g \in \mathrm{GL}_n\bigg\}.$$ 
Then we consider the pair of (Hamiltonian actions associated to the) group actions
\begin{equation} \label{equation Shalika dual pair}
    \mathrm{GL}_{2n} \acts \Sh := (U \mathrm{GL}_n^\Delta \backslash G, \Psi) \, \text{ and } \, \mathrm{GL}_{2n} \acts \check{\Sh} := \mathrm{Sp}_{2n} \backslash \check{G}
\end{equation}
with the following data (see \cite[\S 3.2.1]{BZSV}):
\begin{itemize}
    \item On the $G$-side, we consider the $\Ggr$-action on $\Sh$ given by left multiplication by the cocharacter
    $$e^{-\check{\lambda}_\Sh}: \Ggr \to G$$
    $$t \longmapsto \begin{bmatrix} t \cdot \mathrm{Id}_n & 0\\ 0 & t^{-1} \cdot \mathrm{Id}_n\end{bmatrix}.$$
    The $\Ggr$-eigenform has weight $\varepsilon_\Sh := \langle \check{\lambda}_\Sh, 2\check{\rho}_\Sh\rangle = 2n^2$.
    \item On the $\check{G}$-side, we consider the trivial $\Ggr$-action. 
\end{itemize} 
Then it is stipulated in \textit{loc. cit} that $(G,\Sh)$ and $(\check{G}, \check{\Sh})$ define hyperspherical dual (Hamiltonian) actions. 

Let us spell out carefully the period duality predicted by the hyperspherical dual pair \eqref{equation Shalika dual pair}, and we remark before continuing that it must be interpreted in terms of the renormalization scheme explained in \cite[\S 4.1]{CV}. Let $f$ be a Whittaker normalized cusp form on $\mathrm{GL}_{2n}$ with $L$-parameter $\varphi_f$ valued in the Langlands dual $\mathrm{GL}_{2n}$. Writing out the definition of $P_{\Sh}(f)$ and $L_{\check{\Sh}}(\varphi_f)$, we obtain the period identity
\begin{align} \label{equation vanilla Jacquet Shalika period}
        &\Delta^{-\frac{2n^2 + \varepsilon_{\Sh}}{4}} \cdot \int_{[U][\mathrm{GL}_n]} \, f(um \cdot e^{\check{\lambda}_\Sh}(\partial^{1/2}))\Psi(u) \, du \, dm\\
        &= \begin{cases}
            \, 0 \, \text{ if } \, \varphi_f \text{ is not symplectic, and }\\
            \, \Delta^{\frac{n-2n^2}{4}}\cdot L(1,\varphi_f, \check{\g}/\mathfrak{sp}_{2n}) \, \text{ otherwise}. 
        \end{cases}
\end{align}
Formally, the above period distinguishes trivial central characters $\omega_f = \mathbf{1}$, which corresponds to distinguishing $\varphi_f$ with image inside $\mathrm{SL}_{2n} \subset \mathrm{GL}_{2n}$. In the case of distinction, the automorphic side contributes a divergent factor of $\zeta(1)$ as the central torus part of the integral, matching an identical divergent factor of the spectral side:
$$L(1,\varphi_f, \check{\g}/\mathfrak{sp}_{2n}) = \zeta(1) \cdot L(1, \varphi_f, \wedge_0^2) \text{ for } \mathrm{Im}(\varphi_f) \subset \mathrm{Sp}_{2n}.$$
Formally ``cancelling $\zeta(1)$'s" (and multiplying by a factor of $\Delta^{1/4}$ on the right hand side for dimension reasons, switching from $\check{G}$ to $\mathrm{SL}_{2n}$), we obtain the following hyperspherical dual pair: 
\begin{equation}
    \mathrm{PGL}_{2n} \acts (U\mathrm{PGL}_n^\Delta\backslash \mathrm{PGL}_{2n}, \Psi) \, \text{ and } \, \mathrm{SL}_{2n} \acts \mathrm{Sp}_{2n} \backslash \mathrm{SL}_{2n}
\end{equation}

\subsection{Central twists of Jacquet--Shalika periods on \texorpdfstring{$\GL_4$}{GL4}}  For applications to our automorphic integrals of interest, we will restrict to $n = 2$ (although the general case is similar), and we must allow central twists of \eqref{equation vanilla Jacquet Shalika period}. While this procedure itself is a slight generalization of the induction and restriction of hyperspherical dual actions along central tori \cite[\S 4.1]{CV}, we refrain from discussing things in an unnecessary generality. Instead, we summarize briefly in Remark \ref{remark central twist} below and leave a careful discussion to future work.

\begin{rem}[Central twists and induction] \label{remark central twist}
As discussed in \S \ref{subsub central twists}, expanding on the more terse \cite[\S 4.1]{CV}, it is often convenient to be able to twist automorphic periods by a character and compute the hyperspherical dual formally. More generally, if $X$ is induced from $Y$ along an inclusion $H \subset G$ (or more generally, Whittaker induced along $H \subset G$ and a unipotent subgroup $(U,\Psi)$), then we may twist the automorphic $X$-period over $[G]$ (after unfolding one step as in Lemma \ref{lemma automorphic induction}) by 
\begin{equation} \label{equation centrally twisted period}
    f \longmapsto \int_{[H][U]} \, f(uh)\chi(\delta(h))\theta_Y(h)\Psi(u) \, du \, dh.
\end{equation}
for an automorphic form $f$ on $G$ and a character $\delta: H \to \Gm$ of $H$. Then the period \eqref{equation centrally twisted period} is associated to the Hamiltonian induction along $H(U,\Psi) \subset G \times \Gm$ where $H$ is embedded via $h \mapsto (h, \delta(h))$. The (numerically convergent part of the) centrally twisted period is now labeled by the (Whittaker) induction of $Y_{\mathrm{tw}}$ along $H_{\mathrm{tw}}(U,\Psi) \subset G_{\mathrm{tw}}$, which we denote by $G_{\mathrm{tw}} \acts X_{\mathrm{tw}}$. Note that the subscript $(-)_{\mathrm{tw}}$ now refers to reduction by the central $\Gm \subset H$ embedded in $G \times \Gm$ by $z \mapsto (z,\delta(z))$. 

Now suppose we have the knowledge of the hyperspherical dual action $\check{H} \acts \check{Y}$, from which we have formed $\check{H}_{\mathrm{tw}} \acts \check{Y}_{\mathrm{tw}}$ as in \S \ref{subsub central twists}. By the general principle that ``relative duality commutes with composition of boundary conditions" (see, for instance, \cite{Nakajima}), we may compute the relative dual of $(G, X_{\mathrm{tw}})$ by applying to $\check{Y}_{\mathrm{tw}}$ the Langlands dual operation of Whittaker induction along $H_{\mathrm{tw}}(U,\Psi) \subset G_{\mathrm{tw}}$. This latter procedure is not currently understood in satisfying generality, and may produce non-hyperspherical actions, as will be the case for us.
\end{rem}

Let us now denote the Langlands dual pair of groups
\begin{equation}\label{first appearance GSO and GSpin}
    \check{G} := \frac{\mathrm{GL}_{4} \times \Gm}{(z, z^{-2})} \, \, \text{ and } \, G := \bigg\{(g,\lambda) \in \mathrm{GL}_{4} \times \Gm: \mathrm{det}(g) = \lambda^{2}\bigg\}.
\end{equation}
In view of Remark \ref{remark central twist}, we perform a central twist of the Jacquet--Shalika dual pair of hyperspherical actions along the determinant character $\mathrm{det}: \mathrm{GL}_2^\Delta \to \Gm$ in the following sense. 
Consider the actions\footnote{We apologize for the preposterous convention for having a check versus not; it comes from a conflict of wanting to have ``traditionally" automorphic things without a check, while wanting to let $G$ be the group on which new automorphic phenomena is computed.}
\begin{equation}\label{def_M}
    \check{G} \acts M := (U \mathrm{PGL}_2^\Delta \backslash \check{G} ,\Psi) \, \text{ and } \, G \acts \check{M} := \mathrm{GSp}_{4} \backslash G,
\end{equation}
where $\mathrm{PGL}_2^\Delta$ denotes the subgroup of  $\check{G}$ given by $$ \bigg\{\left(\begin{bmatrix} g & 0 \\ 0 & g\end{bmatrix}, {\rm det}(g)^{-1}\right) \text{ for } g \in \mathrm{PGL}_2\bigg\},$$ and where $\mathrm{GSp}_{4}$ is embedded in $G$ via $g \longmapsto (g, \nu(g)).$ The grading on $M$ and $\check{M}$ are inherited from those of $\Sh$ and $\check{\Sh}$, respectively:
\begin{itemize}
    \item We consider the $\Ggr$-action on $M$ by left multiplication by the character $e^{-\check{\lambda}_\Sh}$, while
    \item on $\check{M}$ we consider the trivial $\Ggr$-action. 
\end{itemize}

We may thus interpret the main statements \cite[Proposition 1 of \S 5.1, Theorem 1 of \S 8]{Jacquet-Shalika} as (one direction of) the weak numerical duality of central twists of the hyperspherical example \eqref{equation Shalika dual pair}.
\begin{thm}[Jacquet--Shalika]\label{theorem Jacquet-Shalika}
    Let $f \times \chi$ be a Whittaker normalized cusp form on $\check{G}$, i.e., $f$ is a Whittaker normalized cusp form on $\check{G}$ with central character $\omega_\Pi$, and $\chi$ an adelic character satsifying $\omega_\Pi = \chi^2$. Then the automorphic period 
    $$P_M(f \times \chi) = \Delta^{-\frac{7+\varepsilon_\Sh}{4}} \cdot \int_{[U][\mathrm{PGL}_2]}f( u\left[\begin{smallmatrix} g& \\ &g\end{smallmatrix}\right] \cdot e^{\check{\lambda}_\Sh}(\partial^{1/2}))\chi^{-1}(\mathrm{det}(g))\Psi(u) \, d(u,g)$$
    computes the spectral period
    $$L_{M}(\varphi_f \times \chi) = \begin{cases} \, 0 \, \text{ unless } \mathrm{Im}(\varphi_f \times \chi) \subset \mathrm{GSp}_{4}, \\ \, 
    \Delta^{-\frac{5}{4}}  \cdot L(1, \varphi_f, \wedge^2_0) \, \text{ in the distinction case.}
    \end{cases}$$
\end{thm}
Note that by our embedding of $\mathrm{GSp}_{4}$ in $G$, we are distinguishing parameters where $\varphi_f$ is valued in symplectic similitudes and whose similitude character agrees with $\chi$.

\subsection{A stacky Jacquet--Shalika period}\label{subsec:stackyJS}

We retain the notation introduced in \eqref{first appearance GSO and GSpin}. We now switch the automorphic and spectral roles of the Langlands dual pair of groups $G$ and $\check{G}$, i.e., now we consider automorphic forms on $G$ whose $L$-parameters are valued in $\check{G}$. We are interested in the (Hamiltonian action associated to the) $G$-space
\begin{equation}\label{eq action P W}
   G \acts W := (U \mathrm{SL}_2\backslash G,\Psi),
\end{equation}
where $\mathrm{SL}_2$ is embedded in $G$ via
$$m \longmapsto \bigg(\begin{bmatrix} m& \\ &m\end{bmatrix}, \,  1\bigg),$$ and $U \simeq \mathrm{Mat}_2$, with the $\Ggr$-action via the cocharacter $e^{-\check{\lambda}_\Sh}$ as above\footnote{We extend $e^{-\check{\lambda}_\Sh}$ to a cocharacter of $G$ by sending
    $t \mapsto \left (\left [\begin{smallmatrix} t \cdot \mathrm{Id}_2 & 0\\ 0 & t^{-1} \cdot \mathrm{Id}_2\end{smallmatrix}\right], 1 \right).$}. The (Hamiltonian) space $W$ encodes the Jacquet--Shalika type period  
$$P_W(f) =  \Delta^{-\frac{7+\varepsilon_\Sh}{4}} \cdot \int_{[U][\mathrm{SL}_2]} \, f\left( u\left[\begin{smallmatrix} m& \\ &m\end{smallmatrix}\right] \cdot e^{\check{\lambda}_\Sh}(\partial^{1/2})\right)\Psi(u) \, du \, dm$$
for automorphic forms $f$ on $G$. We propose that the relative Langlands dual of $(G, W)$ would be a Deligne--Mumford stack, of the type discussed in \S \ref{subsection DM stacks}. To describe it, we start with the prequotient $\mathrm{GL}_{4} \times \Gm$ defining $\check{G}$, and consider the subgroup
$$\mathrm{GSp}_{4} \times \Gm \longrightarrow \mathrm{GL}_{4} \times \Gm.$$
Passing to the central quotient by $(z\mathrm{Id}, z^{-2})$, we obtain an inclusion
\begin{equation} \label{equation GSp2nbarprime}\overline{\mathrm{GSp}}_{4}' :=  \frac{\mathrm{GSp}_{4} \times \Gm}{(z\mathrm{Id}, z^{-2})} \longrightarrow \frac{\mathrm{GL}_{4} \times \Gm}{(z\mathrm{Id}, z^{-2})} = \check{G}.
\end{equation}
Note that $\check{G}$ has a character given by
$$\lambda(g,x) = \mathrm{det}(g)x^2$$
and we can consider a covering group $\widetilde{G}$ obtained by taking a square root of $\lambda$:
\begin{equation*} 
    \pi: \widetilde{G} := \big\{(g,x;y) : (g,x) \in \check{G} \text{ and } \lambda(g,x) = y^2\big\} \longrightarrow \check{G}.
\end{equation*}
Pulling back $\pi$ along the inclusion \eqref{equation GSp2nbarprime}, we obtain an isogeny
\begin{equation} \label{equation GSp4 prime}
    \iota: \mathrm{GSp}_{4}' \to \overline{\mathrm{GSp}}_{4}' \subset \check{G}.
\end{equation}
Note that $\mathrm{GSp}_{4}' \simeq \overline{\mathrm{GSp}}_{4}' \times \mu_2$ splits as a product, since $\lambda$, when restricted to the subgroup $\overline{\mathrm{GSp}}_{4}'$, acquires a natural square root given by
$$\lambda^{1/2}:(g,x) \longmapsto \nu(g)x.$$
We consider the (Hamiltonian action associated to the) $\check{G}$-space
$$\check{G} \acts \check{W} := [\mathrm{GSp}_{4}' \backslash \check{G}]$$
with trivial $\Ggr$-action. Note that $\check{W}$, as a Deligne--Mumford stack, can be understood as the homogeneous space $\overline{\mathrm{GSp}}_{4}'\backslash \check{G}$ modulo the trivial $\mu_2$-action.

As the notation suggests, we are proposing a numerical level duality between the actions $G \acts W$ and $\check{G} \acts \check{W}$. 
\begin{thm} \label{conjecture stacky Shalika dual pair}
    Let $f$ be a Whittaker normalized cuspidal automorphic form on $G$ with $L$-parameter $\varphi_f$ valued in $\check{G}$. Then 
    \begin{equation} \label{equation PW and LcheckW}
        P_W(f) = L_{\check{W}}(\varphi_f).
    \end{equation}
\end{thm}
We will provide the proof of this auxiliary result in \S \ref{subsection Shalika for GSpin6} as Corollary \ref{theorem stacky Shalika dual pair} after having examined more carefully the non-standard Shalika model on $G$. The proof is conceptually clear: one considers a spectral expansion along the embedding $G \subset \mathrm{GL}_4 \times \Gm$ to relate the Shalika period on $G$ to a finite sum of twisted Shalika periods on $\mathrm{GL}_4 \times \Gm$, and we use the formula of Jacquet--Shalika (Theorem \ref{theorem Jacquet-Shalika}) to conclude.
\begin{rem}
    Note that Theorem \ref{conjecture stacky Shalika dual pair} is not an instance of BZSV's numerical conjectures, as both $W$ and $\check{W}$ fail to be hyperspherical albeit in a very mild way. Indeed, we may understand the space $W$ as obtained by Hamiltonian restriction, dual to central twisting as discussed in \S \ref{subsub central twists} and Remark \ref{remark central twist}. 
\end{rem}

\begin{rem}
    In principle, one can expect the full numerical duality, i.e., that $P_{\check{W}} = L_W$ as well, but the statement \eqref{equation PW and LcheckW} is the important one for us in the following. 
\end{rem}

Let us explicate the spectral period $L_{\check{W}}(\varphi_f)$ in classical terms, i.e., in terms of usual Langlands $L$-functions and distinction properties. First of all, we have that $L_{\check{W}}(\varphi_f)$ vanishes unless $\varphi_f$ can be lifted to an unramified $L$-parameter factoring through $\iota$; in particular, $\varphi_f$ must have image inside $\overline{\mathrm{GSp}}_{4}'$ inside $\check{G}$. 

In the distinction case, we may pick a preferred lift $$\widetilde{\varphi}_f := \varphi_f \times (\lambda^{1/2} \circ \varphi_f): \Gamma_F^{\mathrm{ur}} \to \mathrm{GSp}_{4}'$$
induced by the preferred square root $\lambda^{1/2}$, and the set of unramified lifts will be a torsor under $\mathrm{Hom}(\Gamma_F^{\mathrm{ur}}, \mu_2)$:
\begin{align}\label{torsors_L_Wcheck}
    \big\{\text{Lifts of }\varphi_f \text{ to }\mathrm{GSp}_{4}'\big\} \longleftrightarrow \big\{\text{square roots of } \lambda \circ \varphi_f\big\} \longleftrightarrow\mathrm{Hom}(\Gamma_F^{\mathrm{ur}}, \mu_2)
\end{align}
$$\widetilde{\varphi}_f \cdot \varepsilon \longleftrightarrow (\lambda^{1/2} \circ \varphi_f) \cdot \varepsilon \longleftrightarrow \varepsilon$$
so we may write (see Definition \ref{defn_spectral_period_DM_stack})
\begin{align}
    L_{\check{W}}(\varphi_f) &=  \Delta^{-\frac{5}{4}} \sum_{\varepsilon \in \mathrm{Hom}(\Gamma_F^{\mathrm{ur}}, \mu_2)} L(1,\check{\mathfrak{g}}/\overline{\mathfrak{gsp}}_{4}', \widetilde{\varphi}_f \cdot \varepsilon) \nonumber\\
    &= \Delta^{-\frac{5}{4}} \cdot |\mathrm{Hom}(\Gamma_F^{\mathrm{ur}}, \mu_2)| \cdot L(1,\wedge^2_0, \varphi_f), \label{spectral L W check}
\end{align}
where we use the fact that the adjoint action of $\widetilde{\varphi}_f \cdot \varepsilon$ on the Lie algebra quotient $\check{\mathfrak{g}}/\overline{\mathfrak{gsp}}_{4}'$ factors through $\iota$, and hence depends only on $\varphi_f$.

\subsection{Shalika normalization}
\label{subsection Shalika normalization} 
A fundamental annoyance in period formulae of the form \begin{center}
    ``$P_X(f) = L_{\check{X}}(\varphi_f)$"
\end{center} for an automorphic form $f$ with $L$-parameter $\varphi_f$ is that of normalization: which scaling of $f$ should one take so that the period formula is true on the nose? 

From the perspective of BZSV's numerical conjectures (conjectures in \S 14 of \cite{BZSV}) and even in the singular examples of \cite{CV}, one has taken the preference that $f$ should be \textit{Whittaker normalized}, i.e. the normalized Whittaker integral of $f$ should be 1. Such a normalization is advantageous for period integrals that unfold to an adelic integral of Whittaker functions (which encompass the vast majority of examples\footnote{With notable exceptions, such as the ones unfolding to non-unique models, see \cite{BumpFurusawaGinzburgNonUnique}.}), since the local integrals can directly be expressed in terms of irreducible characters of $\check{G}$ via the appropriately shifted Casselman--Shalika formulae (see \S 2.2 of \textit{op. cit.}). 

On the other hand, Whittaker normalization is \textit{disadvantageous} for non-generic forms, or for period integrals that do not easily unfold to the Whittaker model. For instance, the Godement--Jacquet period \cite{Godement-Jacquet} naturally unfolds to the $L^2$-period (i.e. group case period) of the automorphic forms involved. The $L^2$/group case-period being understood (via the formula of Lapid--Mao \cite{Lapid-Mao} and explicated in \S B.2 of \cite{CV}), it is more advantageous in this case to normalize our automorphic forms to have unit $L^2$-norm (as was done in \S B.3.1 of \textit{op. cit.}).

More generally, without preferentially distinguishing a normalization, the numerical conjectures of \\ \cite{BZSV} can be enunciated in the following form. 
\begin{conj}[Conjecture 14.2.1 \cite{BZSV}, tempered polarized case.] \label{conjecture ratio of periods} Let $f$ be an \textit{arbitrarily normalized} unramified tempered automorphic form on $G$ with $L$-parameter $\varphi_f$. For any pair of hyperspherical dual varieties $(M_1, \check{M}_1 = T^*\check{X}_1)$ and $(M_2, \check{M}_2 = T^*\check{X}_2)$, we have an equality of ratios
    \begin{equation}
        \frac{P_{M_1}(f)}{P_{M_2}(f)} \overset{\cdot}{=} \frac{L_{\check{X}_1}(\varphi_f)}{L_{\check{X}_2}(\varphi_f)},
    \end{equation}
where the equality ($\overset{\cdot}{=}$) is to be understood if $P_{M_2}(f) \neq 0$ and $L_{\check{X}_2}(\varphi_f) \neq 0$, and undefined otherwise.
\end{conj}
In practice, one is interested in some period $P_{M_1}$ and is free to choose a convenient $P_{M_2}$ for normalization purposes. It is particularly advantageous to choose an $M_2$ which appears as an open subvariety of $M_1$, in which case the the global prefactors in Definition \ref{definition automorphic period} cancel (assuming that $M_1$ and $M_2$ are both hyperspherical and hence in particular \textit{smooth}). 

The spectral side analogue of the preceding paragraph is, in our experience, the following phenomenon. If one chooses $M_2 \subset M_1$ an open subvariety, it is likely that $\check{X}_1$ is a fiber bundle over $\check{X}_2$ (in the hyperspherical case, in fact vector bundles). In particular, if the fiber is some variety $Y$ and $L_{\check{X}_2}(\varphi_f) \neq 0$, the global prefactors in Definition \ref{definition spectral period} simplify to those of the spectral period for $Y$ itself. For instance, if $\check{X}_2 = H \backslash \check{G}$ for some reductive subgroup $H \subset \check{G}$ and $\check{X}_1 = \mathrm{Ind}_{H}^{\check{G}}(Y)$ for some $H$-space $Y$, then
\begin{equation} \label{equation quotient of spectral periods}
    \frac{L_{\check{X}_1}(\varphi)}{L_{\check{X}_2}(\varphi)} = L_{Y}(\varphi),
\end{equation}
whenever both sides converge. 

\begin{example} \label{example Whittaker normalization}
Let us explicate the preceding discussion in the simplest nontrivial example possible. Consider the group $G = \mathrm{PGL}_2$ with 
$$M_1 = T^*(T \backslash G) \text{ and } M_2 = \mathrm{Whittaker}.$$
where $T \subset G$ is the diagonal torus. Note that $M_2$ is a dense open subvariety in $M_1$, via the following explicit parametrization:
$$M_2 \simeq \begin{bmatrix} 0 & \ast \\ 1 & 0\end{bmatrix} \times G \hookrightarrow \begin{bmatrix} 0 & \ast \\ \ast & 0\end{bmatrix} \times^T G \simeq M_1.$$
Dually, we have chosen the $\check{G} = \mathrm{SL}_2$ spaces
$$\check{X}_1 = \mathbf{A}^2 \, \text{ and } \, \check{X}_2 = \mathrm{pt},$$
with $\check{G}$ acting by the standard (right) representation on $\mathbf{A}^2$ (represented by row vectors). Indeed, $\check{X}_1$ is a vector bundle over $\check{X}_2$. Conjecture \ref{conjecture ratio of periods} in this case reads as follows: for $f$ an unramified cusp form on $G$ with $L$-parameter $\varphi_f$, \textit{arbitrarily normalized}, we have
\begin{equation}\label{equation example of normalization}
    \frac{\Delta^{\frac{-\mathrm{dim}(T)}{4}}\int_{[T]}f(t) \, dt}{\Delta^{-\frac{\mathrm{dim}(U) + \langle 2\rho, 2\check{\rho}\rangle}{4}} \int_{[U]}f(u\partial^{1/2}) \psi(u)du} \overset{\cdot}{=} \frac{\mathfrak{z}_{\check{X}_1}(\varphi_f) \cdot   L(1/2, \varphi_f,\mathrm{std})}{1}
\end{equation}
where the equality ($\overset{\cdot}{=}$) is to be understood when $f$ is \textit{generic}. Now it is indeed for those \textit{Whittaker normalized forms} $f$ (i.e., those $f$ for which the denominator on the left hand side of \eqref{equation example of normalization} equals 1) for which the equation 
\begin{equation}
    \Delta^{\frac{-\mathrm{dim}(T)}{4}}\int_{[T]}f(t) \, dt = \mathfrak{z}_{\check{X}_1}(\varphi_f) \cdot L(1/2, \varphi_f,\mathrm{std})
\end{equation}
holds on the nose. It is straightforward to check that this notion coincides with the \textit{Whittaker normalization} in \cite{CV}; for more details on the rest of this computation, see \S B.4.1 of \textit{op. cit}. On the other hand, we will return to this example in Example \ref{example Whittaker normalization 2}, to analyze the (lack of) discrepancy in this integral.

This example can be easily generalized to an analogous relation between the Bump--Friedberg integral \cite{Bump-Friedberg} and the Jacquet--Shalika integral. More examples of this type have been discovered in the context of Hitchin moduli spaces, see \cite[\S 4.3]{CHY}.
\end{example}

With the above discussion in mind, for automorphic periods unfolding to the Shalika models $P_M$ and $P_W$ (which will be presently of interest to us), it is advantageous to consider the following normalization. 
\begin{defn}\label{def shalikanormalized}
    Let $f$ be an unramified automorphic form on $\check{G}$ (resp. $G$). Then we say that $f$ is \textit{Shalika normalized} if 
    $$P_M(f) = 1 \quad \text{(resp. } P_W(f) = 1\text{).}$$
    In particular, as an immediate consequence of Theorem \ref{theorem Jacquet-Shalika}, an unramified automorphic form $f$ can be Shalika normalized only if its $L$-parameter $\varphi_f$ is conjugate to the symplectic subgroup of the Langlands dual group of $\check{G}$ and $G$, respectively. 
\end{defn}

\subsubsection{Shalika normalization, explicitly}\label{subsubsecn=2}  Recall that the groups involved are
$$\check{G} = \frac{\mathrm{GL}_{4} \times \Gm}{(z, z^{-2})}  \, \, \text{ and } \, G = \bigg\{(g,\lambda) \in \mathrm{GL}_{4} \times \Gm: \mathrm{det}(g) = \lambda^{2}\bigg\}.$$ 
First, we write
$$a_\partial := {\rm diag}(\partial,\partial,1,1) \in \mathrm{GSp}_4(\mathbb{A}),$$
Then the following element of $\check{G}(\mathbb{A})$ (resp. $G(\mathbb{A})$) play a crucial role:
\begin{equation}\label{Definition:a0}a_0 := (a_\partial^{-1} , 1) \text{ (resp. }a_0 := (a_\partial^{-1} , \partial^{-1})\text{)}.\end{equation} 
They are characterized by either (a) being the idèle whose local components are the most antidominant points in the diagonal torus of  $\check{G}(F_v)$ (resp. of $G(F_v)$) at which the local Shalika functions are nonvanishing, or (b) effectuating the same action as that of $\partial^{1/2}$ on a preferred base point appearing in the definition of the theta series \eqref{equation theta series}.

 Let $\Pi$ be a cuspidal unramified automorphic representation of $\check{G}(\A)$. We can identify $\Pi$ with $\pi \otimes \chi$, where  $\pi$ is a cuspidal unramified automorphic representation of $\GL_4(\A)$ and $\chi:[\mathbf{G}_m]\to \C^\times$ is an unramified character  such that $\omega_\pi = \chi^2$. Then we may compute, for a factorizable Shalika normalized cusp form $f \times \chi$ in the space of $\Pi$ and a toral element $a = (a_1,a_2) \in \check{G}(\A)$ with which we translate,
\begin{align} 
    P_M(\Pi(a)&(f \times \chi)) \nonumber \\ &= \Delta^{-\frac{7+\varepsilon_\Sh}{4}} \cdot \int_{[U][\mathrm{PGL}_2]}f( u\left[\begin{smallmatrix} g& \\ &g\end{smallmatrix}\right] \cdot e^{\check{\lambda}_\Sh}(\partial^{1/2})a_1)\chi^{-1}(\mathrm{det}(g))\Psi(u) \, d(u,g) \nonumber \\  &= \Delta^{-\frac{7+\varepsilon_\Sh}{4}} \cdot \int_{[U][\mathrm{PGL}_2]}f( u\left[\begin{smallmatrix} g& \\ &g\end{smallmatrix}\right] a_\partial^{-1} \partial^{1/2} a_1)\chi^{-1}(\mathrm{det}(g))\Psi(u) \, d(u,g) \nonumber\\  &= \Delta^{-\frac{7+\varepsilon_\Sh}{4}} \cdot \int_{[U][\mathrm{PGL}_2]}f( u\left[\begin{smallmatrix} g& \\ &g\end{smallmatrix}\right] a_\partial^{-1} a_1) \chi(\partial)\chi^{-1}(\mathrm{det}(g))\Psi(u) \, d(u,g) \nonumber\\ &= \Delta^{-\frac{7+\varepsilon_\Sh}{4}}\Delta^{\frac{\mathrm{dim}(U)}{2}}  \chi(\partial) \int_{[U][\mathrm{PGL}_2]} f(u \left[\begin{smallmatrix} g& \\ &g\end{smallmatrix}\right]  a_\partial^{-1} a_1)\chi^{-1}(\mathrm{det}(g)a_2)\Psi(u) \, d^\psi u  \,dg \nonumber \\
    &= \Delta^{-\frac{7}{4}} \cdot \Omega_f^\circ \cdot \chi^{-1}(a_2 \partial^{-1})\prod_v \, \Omega_{f,v}^{\mathrm{ur}}(a_{1,v}), \label{equation expliciting Shalika normalization}
\end{align}
where, in the third equality, we used that $\omega_\pi=\chi^2$,  $\Omega_{f,v}^{\mathrm{ur}}$ is the local spherical function which will appear in Proposition \ref{CS:Formula:Shalika:simplified}, in which they are related to Weyl characters of the dual group via a Casselman--Shalika type formula, and $\Omega_f^\circ$ is the globally defined constant\footnote{If the reader wishes to, they may compare this with $W_f^0$ in \cite[\S 2.2]{CV}.}
\begin{equation}
     \Omega_f^\circ := \int_{[U][\mathrm{PGL_2}]} f(u \left[\begin{smallmatrix} g& \\ &g\end{smallmatrix}\right] a_\partial^{-1})\chi^{-1}(\mathrm{det}(g))\Psi(u) \, d^\psi u \, dg. 
\end{equation}
For $f$ to be Shalika normalized, we must have that both sides of equation \eqref{equation expliciting Shalika normalization} are equal to 1 when we set $a$ to be the identity; in particular, we see that for $f \times \chi$ to be $P_M$-normalized is to say that $f$ satisfies the condition
\begin{equation} \label{equation Omegaf0}
    \Omega_f^\circ = \Delta^{\frac{7}{4}} \chi^{-1}(\partial).
\end{equation}
The derivation of the analogue of \eqref{equation Omegaf0} for the period $P_W$ is similar but considerably more involved, and we will return to it in \S \ref{subsection Shalika for GSpin6}, where an Euler factorization formula will be given as Proposition \ref{Shalika:GSpin:In:GL}.

\section{Geometric preliminaries}

In this section, we introduce in detail the varieties with action of interest to us. There is no need, at the level of the discussion in this section, to bias between the automorphic or spectral roles of $G = \mathrm{GSpin}_6$ and $\check{G} = \mathrm{GSO}_6$. We compute in Proposition \ref{proposition spectral X period} and Proposition \ref{proposition spectral Xcheck period} their spectral periods by applying the following key computational observation \cite[\S 3.2.4]{CV}.

\begin{prop} \label{proposition quotient of L functions}
    Let $V$ be an irreducible $G$-representation, and let $Z \subset V$ be a conical, $G$-equivariantly embedded affine variety defined by the complete intersection of $G$-semiinvariant homogeneous polynomials $f_1, \ldots, f_r$ of degrees $d_1, \ldots, d_r$ with eigencharacters $\chi_1, \ldots, \chi_r$. Equip $V$ with the scaling action of $\Ggr$, then for a $G$-valued $L$-parameter $\varphi$ which fixes only the origin in $V$, its spectral $Z$-period is computed in terms of Langlands $L$-functions as
    \begin{equation}
        L_0(Z,\varphi) = \frac{L(1/2,\varphi,V )}{\prod_{i=1}^r \, L(d_i/2, \varphi , \chi_i)}
    \end{equation}
    where $L_0(Z,\varphi)$ is the contribution of the fixed point $0 \in Z$ to the spectral period as in Definition \ref{definition spectral period}.
\end{prop}
It will turn out that the spectral periods involved in our calculations can all be expressed as (finite sums of finite quotients of) Langlands $L$-functions. 

\subsection{The affine cone of \texorpdfstring{$\mathrm{LGr}(2,4)$}{LGr(2,4)}}
Let $\mathrm{GSp}_4$ be the symplectic group over $k$ with respect to the the sympectic form given by $J=\left[\begin{smallmatrix}  & J_2\\-J_2 &\end{smallmatrix}\right]$, with $J_2 = \left[\begin{smallmatrix}  & 1\\1 &\end{smallmatrix}\right]$. We denote by $W_4$ the four dimensional standard representation of $\mathrm{GSp}_4$ and fix its standard basis $\{e_1,e_2,f_2,f_1 \}$. Consider the $\mathrm{GSp}_4$-space
$$Y = \text{ affine cone over the Lagrangian Grasmannian }\mathrm{LGr}(2,4) \text{ of isotropic 2-planes in } k^4.$$
Concretely, we may realize $Y$ as the union of a cone point $y_0$, whose complement $\mathring{Y}$ is explicitly described as \[\mathring{Y} = \{(L,\xi)\,:\, L \in \mathrm{LGr}(2,4),\, 0 \ne \xi \in \wedge^2L\}.\]
Considering the 5-dimensional second fundamental representation of $\mathrm{GSp}_4$
\begin{equation}\label{equation 5dim representation V}
    V:= V_{\omega_2} = \mathrm{Ker}(\wedge^2 W_4 \overset{J}{\longrightarrow} \wedge^4 W_4),
\end{equation}
we see that the Plücker embedding of the affine cone $Y$ defines $Y$ as a conical subvariety of $V$, defined by a concrete homogeneous equation. 
\begin{lem}\label{Plucker:embedding}
    Under the Plücker embedding $i: Y \hookrightarrow \wedge^2 W_4$ sending the cone point $y_0$ to $0$, and on the complement sending an isotropic 2-plane $L \subset W_4$ with orientation $\alpha \in \wedge^2 L$ to $\alpha \in \wedge^2 W_4$, we have the following geometric facts:
    \begin{enumerate}
        \item The image of $i$ is contained in the subspace $ V$, and
        \item $i(Y) \subset  V$ is defined by the vanishing locus of the quadratic form
        \begin{equation} \label{equation quadratic form Q}
            Q:  V \otimes  V \to \wedge^2 W_4 \otimes \wedge^2 W_4 \overset{\wedge}{\longrightarrow} \wedge^4 W_4.
        \end{equation}
    \end{enumerate}
\end{lem}
\begin{proof}
  Firstly, the image of $\mathring{Y}$ via the embedding
    \begin{align*}
        i:\mathring{Y}\to \wedge^2W_4,\,
        (L,\alpha)\mapsto \alpha
    \end{align*}
    is contained in $V$ as  $\alpha \in \wedge^2 W_4 \setminus \{0\}$ is the orientation of an isotropic $2$-plane in $W_4$. Moreover, any such $\alpha$ is trivially isotropic with respect to the quadratic form $Q$. Finally, $i(\mathring{Y})$ surjects onto the vanishing locus in $V \setminus \{0\}$ of $Q$, since any non-zero vector $v \in V$ such that $v \wedge v = 0$ is decomposable, i.e. it can be written as $v = u_1 \wedge u_2$, with $u_1,u_2 \in W_4$ isotropic vectors that define a Lagrangian subspace $\langle u_1,u_2 \rangle \subset W_4$. The result follows.
\end{proof} 

Lemma \ref{Plucker:embedding} let us endow $Y$ with a natural integral structure. Indeed, let $\mathrm{GSp}_{4/\mathcal{O}_F}$ be the symplectic group over $\mathcal{O}_F$ acting on the rank 4 module $\mathbf{W}_4$ of basis $\{e_1,e_2,f_2,f_1\}$. Then, the maximal admissible $\mathcal{O}_F$-lattice $\mathbf{V}$ of $V$, invariant under the action of $\mathrm{GSp}_{4/\mathcal{O}_F}$, induces an integral structure $\mathbf{Y}$ on $Y$, with compatible $\mathrm{GSp}_4$-action, such that
\begin{equation}\label{Integral:Str:Y}\mathbf{Y}(\mathcal{O}_F) = \{\xi \in \mathbf{V} :\, Q(\xi,\xi)=0\}.\end{equation}

\subsection{The space \texorpdfstring{$(G,X)$}{(G,X)} and its spectral period}\label{Section:The:Space:G:X}

We consider the split form $G= \mathrm{GSpin}_{3,3}$ of $\mathrm{GSpin}_{6}$, defined by the short exact sequence 
\begin{equation}\label{ses_G}
    1 \to G  \to  \mathrm{GL}_4 \times \mathbf{G}_m \to \mathbf{G}_m \to 1,
\end{equation}
where the map $\mathrm{GL}_4 \times \mathbf{G}_m \to \mathbf{G}_m$ is given by $(g, \lambda) \mapsto \mathrm{det}(g) \lambda^{-2}$. Set
\begin{align}\label{GSp4embedsintoGSpin33}
    H = \mathrm{GSp}_4 = \big\{(g, \nu(g)): g \in \mathrm{GSp}_4\big\} \subset G
\end{align}
where $\nu$ is the similitude character of $\mathrm{GSp}_4$, and we equip the affine cone $Y$ over $\mathrm{LGr}(2,4)$ with an $H$-action as follows.

Observe that we may realize $\mathring{Y}$ as a homogeneous space for $H$. Choosing the base point $f_2 \wedge f_1 \in \mathring{Y}$, we see that its stabilizer $P_0$ is a subgroup of the Siegel parabolic subgroup $P = N M \subset H$ (where we have written $N$ for the unipotent radical of $P$ and $M \simeq \mathrm{GL}_2 \times \mathrm{GL}_1$ is the Levi subgroup of $P$) given by
$$P_0 = NM_0, \text{ where } M_0 = \left\{m_0(m,\lambda): =\left[\begin{smallmatrix}  m& \\ &\lambda J_2{}^tm^{-1}J_2\end{smallmatrix}\right]: \mathrm{det}(m)\lambda^{-2} = 1\right\}    \subset M.$$

\begin{prop}\label{Quotient:Equals:First:Variety}
 We have
    \[P_0\backslash H\simeq \mathring{Y}.\]
\end{prop}
\begin{proof}
     We have an isomorphism $P\backslash H\simeq \mathrm{LGr}(2,4)$ given by $g\mapsto L_0 g $, where $L_0 := \langle f_2,f_1\rangle$ is the isotropic plane in $W_4$ stabilized by $P$. Since
     \[\mathring{Y} = \{(L,\xi)\,:\, L \in \mathrm{LGr}(2,4),\, 0 \ne \xi \in \wedge^2L\},\]
     the map 
     \begin{align*}
         H\to \mathring{Y},\,\,g\mapsto (L_0 g,f_2\wedge f_1 g),
     \end{align*}
     is surjective, with kernel $P_0$. 
\end{proof}
\begin{lem}
    The homogeneous $H$-space $\mathring{Y}$ is quasi-affine, and 
    $$Y \simeq \spec \Gamma(\mathring{Y}, \mathcal{O}_{\mathring{Y}})$$
    is the affine closure of $\mathring{Y}$.
\end{lem}
    \begin{proof} By Lemma \ref{Plucker:embedding}, we see that $\mathring{Y}$ is an open subscheme of the one defined by the vanishing locus in $V$ of the quadratic form $Q$. As the latter is affine, we conclude that $\mathring{Y}$ is quasi-affine. Thus, there is an open immersion of $\mathring{Y}$ into $\spec\Gamma(\mathring{Y}, \mathcal{O}_{\mathring{Y}})$. Since $Y$ is normal, we have an identification of global functions (see for instance \cite[Exercise I.3.20]{Hartshorne}) 
    $$ \Gamma(\mathring{Y}, \mathcal{O}_{\mathring{Y}}) \simeq \Gamma(Y, \mathcal{O}_{Y}).$$
\end{proof}

As an affine cone, there is a natural scaling action on $Y$ restricted from the linear scaling action on $\wedge^2 W_4$; we will let this be the $\Ggr$-action on $Y$. With respect to these actions, we compute the weight of an eigenvolume form on $\mathring{Y}$.

\begin{lem}\label{eigen:char:G}
     There is an $H \times \Ggr$-eigenvolume form $\omega$ on $\mathring{Y}$ such that 
    \[(g,\lambda)^*\omega = \eta_Y(g,\lambda) \omega = (\nu(g)\lambda)^{3}\omega,\]
    where $g\in H$ and $\lambda\in \Ggr$.
\end{lem}
\begin{proof}
    Note that $P_0$ acts on the top exterior power of the cotangent fiber at $v_0$ by the modulus character $\delta_{P_0}$. Indeed, the tangent space at $v_0$ may be identified with the Lie algebra quotient $T_{v_0}\mathring{Y} \simeq \mathfrak{h}/\mathfrak{p}_0$, on which $P_0$ acts by the \textit{right} adjoint action. Taking the top exterior power and dualizing, we see that $P_0$ acts by $\delta_{P_0}$. 

    Now the character $\eta_Y: H \to \Gm$ defined by $\eta_Y(g) = \nu(g)^3$ satisfies $\eta_Y|_{P_0} = \delta_{P_0}$. Therefore according to \cite[\S 4.1]{Sakellaridis:Venkatesh}, we have an eigenvolume form on $\mathring{Y}$ with eigencharacter $\eta_Y$. Moreover, since the grading torus $\Ggr$ acts as the square root of the center, the result follows. 
\end{proof}

Finally, we define the $G$-space
$$X := \mathrm{Ind}_H^G(Y) := Y \times^H G$$
which inherits the structure of a graded $G$-variety (including an integral structure) from the structure of $Y$ as a graded $H$-variety. One verifies readily that there exists a $G$-eigenform $\omega_X$ on $X$ whose eigencharacter $\eta_X$ can be described as 
\begin{equation}
    \eta_X(g,a) = a^{3}  \text{ for } (g,a) \in G
\end{equation}
restricting to $\eta_Y$ on the subgroup $H$.

As an induced space, the spectral $X$-period is completely determined by the spectral $Y$-period by applying Lemma \ref{lemma spectral induction}, so we compute the latter. 
\begin{prop} \label{proposition spectral X period}
    The spectral $X$-period distinguishes $L$-parameters that factor through $H$, and for $\varphi: \Gamma_F \to H$ an unramified $L$-parameter with big enough image so that its only fixed point on $Y$ is the cone point, the spectral $X$-period is given by
    $$L_X(\varphi) = (\nu \circ \varphi)(\partial)^{-3/2} \cdot \Delta^{-3/2} \cdot L(1, \varphi, \wedge^2_0) \cdot \frac{L(1/2, \varphi, \wedge_0^2 \otimes \nu)}{L(1, \varphi, \nu^{2})},$$
    where  $\wedge_0^2:= V \otimes \nu^{-1}$, with $V$  the 5-dimensional second fundamental representation \eqref{equation 5dim representation V} of $H = \mathrm{GSp}_4$. 
\end{prop}
\begin{proof}
By Lemma \ref{lemma spectral induction}, we have 
$$L_X(\varphi) = \Delta^{\frac{\mathrm{dim}(H)-\mathrm{dim}(G)}{4}} \cdot L(1, \varphi, \mathfrak{g}/\mathfrak{h})   L_Y(\varphi).$$
Firstly, note that the $\mathfrak{g}/\mathfrak{h}$, which we view as an $H$-representation via the adjoint action, corresponds to the twist $\wedge_0^2= V \otimes \nu^{-1}$ of $V$. Moreover,  an application of Proposition \ref{proposition quotient of L functions} to the $H$-equivariant closed embedding $Y \subset V$, defined by the vanishing locus of the quadratic form $Q$ of \eqref{equation quadratic form Q} gives \begin{align*}
    L_Y(\varphi) &= \Delta^{\frac{\varepsilon_Y - {\rm dim}(Y)}{4}}(\mathrm{det}(V) \circ \varphi)(\partial)^{-1/2} (\nu \circ \varphi)(\partial)  \cdot \frac{L(1/2,\varphi,V )}{L(1,\varphi,\nu^2)} \\ &= \Delta^{-1/4}(\nu \circ \varphi)(\partial)^{-3/2}\cdot \frac{L(1/2,\varphi,V )}{ L(1,\varphi,\nu^2)}.
\end{align*}
The result follows.
\end{proof}
Let $\varphi$ be as in Proposition \ref{proposition spectral X period}, in particular $\varphi$ factors through (a conjugate of) $H$. Since $$L_{\check{M}}(\varphi) =\Delta^{-5/4}  \cdot L(1, \varphi, \wedge^2_0),$$ the Shalika normalized spectral period (as in Conjecture \ref{conjecture ratio of periods}) is given by

\begin{align}\label{Shalika normalized spectral period X}
        \frac{L_X(\varphi)}{L_{\check{M}}(\varphi)} &= (\nu \circ \varphi)(\partial)^{-3/2} \cdot \Delta^{-1/4} \cdot \frac{L(1/2, \varphi, \wedge_0^2 \otimes \nu)}{L(1, \varphi, \nu^{2})}.
\end{align}
From the automorphic perspective, if $\varphi = \varphi_{f \times \chi}$ is the $L$-parameter of $f \times \chi$,  where $f$ is a cusp form on $\mathrm{GL}_4$ and $\chi$ is a Hecke character such that  $\omega_f = \chi^2$, then we can calculate $\mathfrak{z}_X$ directly using the central character $\omega_f \times \chi$ of $f \times \chi$ restricted to the central cocharacter $z \mapsto (1, z^{3})$ into $\check{G}$, dual to $\eta_X$:
\begin{equation}
    \mathfrak{z}_X(\varphi) = (\omega_f \times \chi)(\mathrm{Id}, (\partial^{-1/2})^{3}) = \chi(\partial^{-3/2}).
\end{equation}
Indeed, if $\varphi$ factors through $H$, $\nu \circ \varphi = \chi$.

\subsection{The space \texorpdfstring{$(\check{G}, \check{X})$}{(Gcheck,Xcheck)} and its spectral period} \label{section Gcheck Xcheck}

Dually, we consider the Langlands dual reductive group $\check{G} = \mathrm{GSO}_{3,3}$ of $G$, presented as the quotient
\begin{equation}\label{ses_checkG}
   \check{G} = \frac{\mathrm{GL}_4 \times \mathrm{GL}_1}{(z \mathrm{Id}, z^{-2}: z \in \Gm)}.
\end{equation}
Recall from  \eqref{equation GSp2nbarprime}, the auxiliary quotient group $\overline{H}'$ of $\mathrm{GSp}_4 \times \mathrm{GL}_1$ given by
\begin{align}\label{GSp4embedsintoGSO33}
\overline{H}' = \frac{\mathrm{GSp}_4 \times \mathrm{GL}_1}{(z\mathrm{Id}, z^{-2}: z \in \Gm)} \subset \check{G}
\end{align}
embedded naturally in $\check{G}$ as a subgroup. Moreover, recall that the isogenous cover $H'$ of $\overline{H}'$, given in \eqref{equation GSp4 prime}, is obtained by taking the square root of the character
$(g,x) \longmapsto \mathrm{det}(g)x^2.$ That is, elements of $H'$ can be represented by triples
\begin{equation*}
    H' = \left\{(g,x;y) \, : \,  (g,x) \in \overline{H}'\, \text{ and } y^2 = \mathrm{det}(g)x^2 \, \right\}
\end{equation*}
and $H'$ is equipped with an isogeny
$$\iota: H' \to \check{G} \, \text{ by } (g,x,y) \mapsto (g,x).$$
Note that since the similitude character $\nu(g)$ is a square root of $\mathrm{det}(g)$, the map 
\begin{equation} \label{equation map splitting isogeny}
    j: \overline{H}' \longrightarrow H'
\end{equation}
$$(g,x) \mapsto (g,x; \nu(g)x)$$
defines a section of the isogeny $H' \to \overline{H}' = \mathrm{Im}(\iota)$. Thus, $H'$ is a direct product via the isomorphism
\begin{equation} \label{equation direct product}
    H' \overset{\sim}{\longrightarrow} \overline{H}' \times \{\pm 1\}
\end{equation}
$$(g,x;y) \longmapsto (g,x) \times \frac{y}{\nu(g)x}.$$
We let the group $H'$ act on $Y$ as 
\begin{equation}\label{eq_new_action_H'}
    (g,x;y)\cdot (L,\xi) := (Lg,\xi g\nu(g)^{-1}y).
\end{equation}

Let 
$$P'_0 := \{ (p,x,\mu(p)\nu(p)^{-1}) \in H'\,:\, p \in P, x^2 = \mu(p)^2 \nu(p)^{-4}\} \hookrightarrow P',$$
where $P$, resp. $P'$, denote the Siegel parabolic of $H$, resp. $H'$, and 
$\mu$ denotes the character
\begin{align*}P \twoheadrightarrow \mathrm{GL}_2 \times \mathrm{GL}_1  \xrightarrow[]{(m,\lambda) \to  \mathrm{det}(m)} \mathrm{GL}_1.
\end{align*}
We now consider the homogeneous $H'$-space 
$$\mathring{Y}' := P'_0\backslash H'.$$ 
\begin{lem}\label{Description:Variety:Lemma}
    As varieties, $\mathring{Y} \simeq \mathring{Y}'$. In particular, the homogeneous $H' \times \Ggr$-space $\mathring{Y}'$ is quasi-affine. 
\end{lem}
\begin{proof}
It is evident that the action \eqref{eq_new_action_H'} of $H'$ on $\mathring{Y}$ is transitive, and the stabilizer of the generic point $f_2 \wedge f_1$ is exactly $P'_0$.
\end{proof}

Since $\mathring{Y}'$ is quasi-affine, we may denote by 
$$\mathring{Y}' \hookrightarrow Y' := \spec \Gamma(\mathring{Y}', \mathcal{O}_{\mathring{Y}'})$$
its affine closure with the extended $H'$-action (while $Y \simeq Y'$, we will use separate symbols to remember whether we consider an $H$ or $H'$-action, respectively). We regard $Y'$ as a graded $H'$-action via the scaling action of $\Ggr$ (which coincides with the action of the variable $y$).

Using the map $\iota$ and the $H'$-space $Y'$, we finally define 
$$\check{X} := \mathrm{Ind}_{H'}^{\check{G}}(Y')$$
which is a Deligne--Mumford stack with right $\check{G}$-action. We may use the isomorphism \eqref{equation direct product} to equivalently present $\check{X}$ as the induction of a Deligne--Mumford stack $[Y/\mu_2]$ along the \textit{inclusion} $\overline{H}' \subset \check{G}$
$$\check{X} \simeq \mathrm{Ind}_{\overline{H}'}^{\check{G}}([Y/\mu_2])$$
which is psychologically convenient when computing automorphic and spectral periods associated to $\check{X}$.

Next, we would like to equip $(\check{G}, \check{X})$ with the structure of a graded $\check{G}$-space. Since $\check{X}$ is defined via induction, by the discussion of \S \ref{subsection period induction} it suffices to construct a graded action of $H'$ on $Y'$ and induce. Here we have the analogue of Lemma \ref{eigen:char:G}.
\begin{lem}\label{Lemma Gcheck eigenmeasure}
    There is a $H'\times \Ggr$-eigenvolume form $\omega$ on $\mathring{Y}'$ such that 
    \[((g,x;y),\lambda)^*\omega = \left(y\lambda\right)^3\omega,\]
    where $(g,x;y)\in H'$ and $\lambda\in \Ggr$.
\end{lem}
\begin{proof}
    Note that $P'_0$ acts on the top exterior power of the cotangent fiber at $v_0$ by the modulus character $\delta_{P'_0}$. Indeed, the tangent space at $v_0$ may be identified with the Lie algebra quotient $T_{v_0}\mathring{Y}' \simeq \mathfrak{h}'/\mathfrak{p}'_0$, on which $P_0'$ acts by the \textit{right} adjoint action. Taking the top exterior power and dualizing, we see that $P'_0$ acts by $\delta_{P'_0}$. 

    Now the character $\eta_{Y'}: H' \to \G_m$ defined by $\eta_{Y'}(g,x;y) = y^3$ satisfies $\eta_{Y'}|_{P'_0} = \delta_{P'_0}$. Therefore, according to \cite[\S 4.1]{Sakellaridis:Venkatesh}, there is a $H'$-eigenvolume form on $\mathring{Y}'$ with character $\eta_{Y'}$. Moreover, since the grading torus $\Ggr$ acts exactly as the action of the $y$-component, we obtain the formula in the statement. 
\end{proof}

With the graded $\check{G}$-variety $\check{X}$ well-defined, we may now compute the associated spectral period. As $Y$ and $Y'$ are isomorphic as varieties, it is not surprising that the resulting spectral expressions are closely related. 

Consider $\varphi \times \chi: \Gamma_F \to \check{G}$, and we try to lift it along $\iota$ to obtain some $\varphi_0 \times \chi_1 \times \varepsilon$. What this means is that $\varphi \times \chi = \varphi_0 \times \chi_1$ is valued in $\overline{H}'$, and $\mathrm{det}(\varphi_0)\chi_1^2 = \varepsilon^2$. Thus, the only data we need to provide the lift is $\varepsilon$, in other words, a square root of $\mathrm{det}(\varphi)\chi^2$.

\begin{prop}\label{proposition spectral Xcheck period}
    Let $\varphi \times \chi: \Gamma_F \to \check{G}$ be an unramified $L$-parameter. Then the spectral $\check{X}$-period distinguishes $L$-parameters that lift along the isogeny $\iota: H' \to \check{G}$, in which case if $\varphi \times \chi$ has big enough image so that its only fixed point on $Y$ is the cone point, $L_{\check{X}}(\varphi \times \chi)$ is given by
    $$\Delta^{-3/2}\cdot L(1, \varphi, \wedge_0^2) \cdot \sum_{\varepsilon \in S(\mathrm{det}(\varphi)\chi^2)} \, \varepsilon(\partial^{-1/2})^3 \,  \frac{L(1/2, \varphi \times \chi, \wedge^{\varsigma_0}_0 \otimes \varepsilon ) }{L(1,\varepsilon^2)},$$
    where $S(\mathrm{det}(\varphi)\chi^2)$ denotes the set of unramified square roots of the character $\mathrm{det}(\varphi)\chi^2: \Gamma_F^{\mathrm{ur}} \to \kk^\times$ and $\wedge^{\varsigma_0}_0$ is the irreducible algebraic representation of $\overline{H}'$ given by $\wedge^2_0 \boxtimes \varsigma_0$.
\end{prop}
\begin{proof}
    We compute $L_{\check{X}}$ using Definition \ref{defn_spectral_period_DM_stack}. Assuming that $\varphi \times \chi: \Gamma_F \to \check{G}$ lifts along $\iota$, we know that $\varphi$ has image contained in $\overline{H}'$, and the adjoint representation of $\varphi \times \chi$ on $\check{\g}/\overline{\mathfrak{h}'}$ is precisely $ \wedge^{\varsigma_0}_0$, thus explaining the factor $L(1, \varphi, \wedge_0^2) = L(1, \varphi \times \chi,  \wedge^{\varsigma_0}_0)$. To conclude, we need to compute the spectral period of $Y'$ evaluated at lifts of $\varphi \times \chi$ along $\iota$. Since $Y'$ is a conical closed subvariety in the $H'$-representation $V':=\wedge^{\varsigma_0}_0 \boxtimes \varsigma_1$ defined by the quadratic form $Q$ of \eqref{equation quadratic form Q}, an application of Proposition \ref{proposition quotient of L functions} gives 
    \begin{align*}
        L_{Y'}(\varphi \times \chi \times \varepsilon) = \varepsilon(\partial)^{-3/2}\cdot \Delta^{-1/4} \cdot \frac{L(1/2,\varphi, \wedge^{\varsigma_0}_0 \otimes \varepsilon )}{ L(1,\varepsilon^2)}.
\end{align*}
\end{proof}

Thanks to \eqref{spectral L W check}, the Shalika normalized $\check{X}$-spectral period (in the sense of Conjecture \ref{conjecture ratio of periods}) is given by 
\begin{equation}\label{spectral_side_Xcheck_W_check}
    \frac{L_{\check{X}}(\varphi)}{L_{\check{W}}(\varphi)} =  \frac{\Delta^{-1/4}}{|S(\mathrm{det}(\varphi)\chi^2)|} \sum_{\varepsilon \in S(\mathrm{det}(\varphi)\chi^2)} \, \varepsilon(\partial^{-1/2})^3 \cdot \frac{L(1/2, \varphi,\wedge^{\varsigma_0}_0 \otimes \varepsilon)}{L(1,\varepsilon^2)},
\end{equation}
assuming that $\varphi$ lifts along the isogeny $\iota: H' \to \check{G}$. Note that  $|S(\mathrm{det}(\varphi)\chi^2)| = |\mathrm{Hom}(\Gamma_F^{\mathrm{ur}}, \mu_2)| $, which follows by \eqref{torsors_L_Wcheck}.

\section{Automorphic preliminaries}

In order to compute the automorphic local periods of \S \ref{subsec_automorphic_side_G} and \S \ref{subsec_automorphic_side_Gcheck}, we need a Casselman--Shalika formula for the unramified Shalika periods $P_W$ and $P_M$ on $G$ and $\check{G}$, respectively. While the formula for $\check{G}$ is essentially given by the $\GL_4$-case of \cite[Theorem 2.1]{SakellaridisCSShalika}, the corresponding formula for $G$ requires additional work and constitutes one of the new technical contributions of this manuscript. In \S\ref{subsection Shalika for GSpin6}, we relate the period $P_W$ on $G$ to Shalika periods on $\GL_4 \times \GL_1$, and then apply \cite{SakellaridisCSShalika} to derive the desired formula (see Propositions \ref{Shalika:GSpin:In:GL} and \ref{Final:Formula:Shalika:GSpin}). As part of this argument, we establish Theorem \ref{conjecture stacky Shalika dual pair}.

\subsection{Sakellaridis' Casselman--Shalika formula on \texorpdfstring{$\mathrm{GSO}_{6}$}{GSO(6)}} \label{Subsection:CS:Sakellaridis}
Let $v$ be a place of $F$. Since $$\check{G} = \frac{\mathrm{GL}_4 \times \mathrm{GL}_1}{(z \mathrm{Id}, z^{-2}: z \in \Gm)},$$
we can identify any generic unramified representation of $\check{G}(F_v)$ with $\pi_v \otimes \chi_v$, where $\pi_v$ is a generic unramified representation $\GL_4(F_v)$ and $\chi_v$ is an unramified character  of $\GL_1(F_v)$ such that $\omega_{\pi_v} \chi_v^{-2}=1$. We are therefore left to spell out a formula for the unramified Shalika functional attached to $\pi_v$ and $\chi_v$.

We recall that, if $S:=U \GL_2$ denotes the Shalika subgroup of $\GL_4$, a local Shalika functional of $\pi_v$ is an element of \begin{equation}\label{Hom:Space:Shalika}
{\rm Hom}_{S(F_v)}(\pi_v, \Psi_v^{-1} \times \chi_v),\end{equation}
where $\Psi_v : U(F_v) \simeq {\rm Mat}_2(F_v) \to \kk^\times$ is given by $t \mapsto \psi_v({\rm Tr}(t))$. Recall, from \S \ref{subsec_add_chars}, that $\psi_v$ has conductor $\varpi_v^{-2m_v}\mathfrak{o}_v$, thus it might be ramified.

The ${\rm Hom}$-space \eqref{Hom:Space:Shalika} is at most one dimensional; more precisely, a generic representation $\pi_v$ of $\GL_4(F_v)$ has a non-trivial Shalika functional with respect to the character $\chi_v$ if and only if it has central character $\chi_v^2$ and its Langlands parameter factors through $\mathrm{GSp}_4(\kk)$ with similitude character $\chi_v$ (see, for instance, \cite[Theorem 1.5 (i)]{Gan:Takeda:JM}). This implies that the only generic, unramified representations with non-trivial Shalika functional are isomorphic to 
\[\mathrm{Ind}_{B(F_v)}^{\GL_4(F_v)}(\xi_1\otimes\xi_2\otimes(\chi_v\xi_2^{-1})\otimes(\chi_v\xi_1^{-1})),\]
where $\mathrm{Ind}_{B(F_v)}^{\GL_4(F_v)}$ denotes normalized induction from the upper triangular Borel subgroup.  

When these conditions are in place, after choosing a spherical vector $\phi_0$ for $\pi_v$ such that $\phi_0(1)=1$, we denote by $P_{M,v}$ a generator of \eqref{Hom:Space:Shalika}, normalized so that $$P_{M,v}\left(\pi_v \left[\begin{smallmatrix}
      \varpi_v^{-2m_v}I_2 & \\ & I_2
\end{smallmatrix}\right]\phi_0\right)=1.$$   In what follows, the Shalika function attached to $\phi_0$, $\chi_v$, and $\psi_v$ is defined as \begin{equation}\label{new:def:spher:Shal}\Omega_{\phi_0,v}^{\chi_v} : \GL_4(F_v) \to \kk,\;\;
g\mapsto P_{M,v}(\pi_v(g) \phi_0).
\end{equation}
When clear from the context, we suppress $\phi_0$ and $\chi_v$ from the notation and call it by $\Omega_{v}$. Moreover, in the following we denote by $\Omega_{v}^{\rm ur}$ the unramified Shalika function associated to $\phi_0$, $\chi_v$ and  an unramified additive character $\psi_v^{\rm ur}$, normalized so that $\Omega_{v}^{\rm ur}(1)=1$.

A  special case of the main result of \cite{SakellaridisCSShalika} provides a formula of $\Omega_{v}^{\rm ur}$ evaluated at $$\varpi_v^{(n,n)} : = \left[\begin{smallmatrix}
    \varpi_v^n I_2 & \\ & I_2
\end{smallmatrix}\right],$$ with $n \in \ZZ$.  In particular, it is shown to be non-zero only if $n \geq 0$. We now record the formula of \emph{loc. cit.}

\begin{prop}\label{Sakellaridis_prop}
Let $\xi = \xi_1\otimes\xi_2\otimes(\chi_v\xi_2^{-1})\otimes(\chi_v\xi_1^{-1})$ be a character of the maximal diagonal split torus of $\GL_4(F_v)$ and consider $\pi_v$ the unramified subquotient of the normalized induction $\mathrm{Ind}_{B(F_v)}^{\GL_4(F_v)}\xi$. The unramified Shalika function $\Omega_{v}^{\rm ur}$ attached to $\pi_v$ and $\chi_v$ satisfies 
$$ \Omega_{v}^{\rm ur}( \varpi_v^{(n,n)}) = 
    \frac{q_v^{-2n}}{1+ q_v^{-1}} \cdot \frac{ \mathcal{A}\Big( e^{\rho+n(\alpha_1+\alpha_2)} \cdot \prod_{\alpha \in \Phi_{\GSp_4}^{+,s}} (1 - q_v^{-1} e^{-\alpha}) \Big) ({\rm Sat}_v)}{{\mathcal{A}\left(e^{\rho}\right)}({\rm Sat}_v)},$$
    where we have used the following set of notations:
    \begin{itemize}
        \item $\rho = \frac{1}{2}\sum_{\alpha \in \Phi_{\GSp_4}^+} \alpha$ is the half sum of positive roots of $\mathrm{GSp}_4$;
        \item ${\rm Sat}_v$ denotes the Satake parameter of $\pi_v$;
        \item $\mathcal{A}(\cdot)$ denotes the alternator with respect to the Weyl group of $\GSp_4$; 
        \item $\Phi_{\mathrm{GSp}_4}^{+,s}$ denotes the set of positive short roots of $\mathrm{GSp}_4$.
    \end{itemize}
    Moreover, $e^{\alpha}(\mathrm{Sat}_v) = \xi(a_{\alpha})$ where, if $\alpha$ is the root $\mathrm{diag}(t_1,\cdots,t_4)\mapsto t_it_j^{-1}$, then $a_{\alpha}$ is the diagonal element with $\varpi_v$ on the $i$th row, $\varpi_v^{-1}$ on the $j$th row and $1$'s otherwise, and $e^{n(\alpha_1+\alpha_2)}(\mathrm{Sat}_v) = \xi(\varpi_{v}^{(n,n)}) =(\xi_1\xi_2)(\varpi_v^n)$. 
\end{prop}
\begin{proof}
This is essentially \cite[(78)]{SakellaridisCSShalika}. Observe that, despite the fact that the formula of \emph{loc.cit.} is proven in the case where $\chi_v$ is trivial, the result extends to arbitrary unramified characters $\chi_v$. Indeed, the calculations of Propositions 8.1 and 8.2 of \emph{loc.cit.} are identical as long as $\chi_v^2 = \omega_{\pi_v}$ (which is in any case a necessary condition to have a non-trivial Shalika model).
Thus, we obtain an identical formula as the one of \cite[(74)]{SakellaridisCSShalika}. Note that the embedding of maximal diagonal tori $T_{\GSp_4(\kk)} \hookrightarrow T_{\GL_4(\kk)}$ induces a surjection $\mathrm{X}^*(T_{\GL_4(\kk)}) \to \mathrm{X}^*(T_{\GSp_4(\kk)})$. This allows us to write each term $\xi(a_\alpha)$ as $e^\lambda({\rm Sat}_v)$ for some character $\lambda \in \mathrm{X}^*(T_{\GSp_4(\kk)})$. 
\end{proof}
 \begin{rem}
We observe that, in \cite{BFGsplitorthogonal}, the authors also proved a formula similar to that of Proposition \ref{Sakellaridis_prop} for Shalika models of $\mathrm{PGSO}_6$.
 \end{rem}

Now, fix an unramified representation $\pi_v$ of $\GL_4(F_v)$ whose Satake parameter $\mathrm{Sat}_v$ lies in $\GSp_4(\kk)$. 
For $(a,b;c) \in \Z_{\geq 0}^2 \times \Z$, with $a \geq b \geq 0$, let $V_{(a,b;c)}$ be the irreducible representation of $\GSp_4(\kk)$ with highest weight $a \alpha_1 + b \alpha_2 + c \alpha_0$: \begin{align}\label{highest_weight_GSp4}
(a,b;c) \,:\, {\rm diag}(x_1,x_2,\nu x_2^{-1},\nu x_1^{-1}) \mapsto x_1^a x_2^b \nu^{c}.
\end{align} 
When $a + b \equiv 0 $ [mod $2$], the representation $V_{(a,b;-\frac{a+b}{2})}$ has trivial central character. Set
\begin{align*}
    s_{(a,b;c)} := \mathrm{Tr}(V_{(a,b;c)}|\mathrm{Sat}_v).
    \end{align*}
We can write \begin{align}\label{playing_with_traces}
   s_{(a,b;c)} = s_{(0,0;c)}s_{(a,b;0)}=  \chi_{v}(\varpi_v)^{c}s_{(a,b;0)},
    \end{align} where $\chi_{v}(\varpi_v)$ is the similitude of $\mathrm{Sat}_v$. Also, for any character $\eta\in \mathrm{X}^*(T_{\GSp_4}(\kk))$, with $T_{\GSp_4}$ denoting the maximal diagonal torus of $\GSp_4$, we let $\mathcal{A}(\eta) := \mathcal{A}(e^{\eta})({\rm Sat}_v)$. 
\begin{lem}\label{lemma:simplificationCSstep1}
    Suppose that $\pi_v$ appears as an irreducible subquotient of \[\mathrm{Ind}_{B(F_v)}^{\GL_4(F_v)}(\xi_1\otimes\xi_2\otimes(\chi_v\xi_2^{-1})\otimes(\chi_v\xi_1^{-1})).\] Then, we have
    \[s_{(a,b;c)} =  \frac{\mathcal{A}(\rho+(a,b;c))}{\mathcal{A}(\rho)}.\]
\end{lem}
\begin{proof}
    This follows from the Weyl's character formula for $\GSp_4$ (see, for instance, \cite[\S 24]{FultonHarris}).
\end{proof}

\begin{prop}\label{CS:Formula:Shalika:simplified}
For $n > 0$, we have
$$\Omega_v^{\rm ur}(\varpi_v^{(n,n)}) =  q_v^{-2n} \, \bigg( s_{(n,n;0)} - \chi_v(\varpi_v) s_{(n-1,n-1;0)}q_v^{-1}\bigg).$$
\end{prop}
\begin{proof}
Let us denote $$\widetilde{\mathcal{A}}(\rho+(n,n;0)):= \mathcal{A}\Big( e^{\rho+n(\alpha_1+\alpha_2)} \cdot \prod_{\alpha \in \Phi_{\GSp_4}^{+,s}} (1 - q_v^{-1} e^{-\alpha}) \Big) ({\rm Sat}_v).$$ Expanding out the definition of the alternator (and writing $t = q_v^{-1}$ to save some space), we have
\begin{align*}
    &\widetilde{\mathcal{A}}(\rho+(n,n;0)) = \mathcal{A}((n+2,n+1;0)-3\alpha_0/2)-t\mathcal{A}((n+1,n;0)-\alpha_0/2)\\
    & \quad \quad \quad \quad \quad \quad \quad \,\, -t\mathcal{A}((n+1,n+2;0)-3\alpha_0/2)+t^2\mathcal{A}((n,n+1)-\alpha_0/2). 
\end{align*}
Since the Weyl element $ \omega_{\alpha_1-\alpha_2}$ swaps $\alpha_1\leftrightarrow\alpha_2$ and fixes $\alpha_0$, one can reorder the third and fourth alternators to get 
\begin{align*}
   \mathcal{A}((n+2,n+1&;0)-3\alpha_0/2)-t\mathcal{A}((n+1,n;0)-\alpha_0/2)\\ +t&\mathcal{A}((n+2,n+1;0)-3\alpha_0/2)-t^2\mathcal{A}((n+1,n;0)-\alpha_0/2)\\ &= (t+1)(\mathcal{A}((n+2,n+1;0)-3\alpha_0/2)-t\mathcal{A}((n+1,n;0)-\alpha_0/2)) \\&= (t+1)(\mathcal{A}(\rho+(n,n;0))-t\mathcal{A}(\rho+(n-1,n-1;1))).
\end{align*}
Applying Proposition \ref{Sakellaridis_prop} and Lemma \ref{lemma:simplificationCSstep1}, we have that
\begin{align*}
    \Omega_v^{\rm ur}(\varpi^{(n,n)})&= q_v^{-2n}\, \frac{\mathcal{A}(\rho+(n,n;0)) - q_v^{-1}\mathcal{A}(\rho+(n-1,n-1;1))}{\mathcal{A}(\rho)} \\ 
&=    q_v^{-2n} \, \bigg( s_{(n,n;0)} - \chi_v(\varpi_v) s_{(n-1,n-1;0)}q_v^{-1}\bigg).
\end{align*}
\end{proof}

 \begin{rem}\label{Vanishing:Character}
Note that the formula of Proposition \ref{CS:Formula:Shalika:simplified} holds formally when $n = 0$, for the following reason.  Let $s_1$ and $s_2$ be the Weyl elements associated to the simple roots $\alpha_1-\alpha_2$ and $2\alpha_2-\alpha_0$ respectively. The vector $\rho+(-1,-1;1)$ is invariant under the action of $s_2s_1s_2$, implying that
\[ s_{(-1,-1;1)} :=\frac{\mathcal{A}(\rho+(-1,-1;1))}{\mathcal{A}(\rho)} = 0.\]
Thus, when $n=0$, we recover $\Omega_v^{\rm ur}(1) = 1$, as expected.
 \end{rem}
 
\subsection{The Shalika period on \texorpdfstring{$\mathrm{GSpin}_6$}{GSpin(6)}} \label{subsection Shalika for GSpin6}
Throughout this section, we regard \(\GL_2\) as a subgroup of \(\GL_4\) via the diagonal embedding
    \[L := \left\{\left[\begin{smallmatrix}m& \\ &m\end{smallmatrix}\right],\;m\in \mathrm{GL}_2\right\}.\]
If $F_1$ and $F_2$ are automorphic forms on $\GL_2$, with central characters satisfying $\omega_{F_1} = \omega_{F_2}^{-1}$ so that $F_1F_2$ descends to the central quotient, we define
\[\langle F_1, F_2\rangle := \int_{Z_{\GL_2}(\A)\setminus [\GL_2]}F_1(g)F_2(g)dg,\]
whenever the integral is convergent. 
\begin{prop}\label{Auxiliary:Lemma:Spectral:Expansion}
    Let $\mathrm{Eis}(l,f_s)$ be an Eisenstein series of $\GL_2$ attached to a section $f_s\in \mathrm{Ind}_{B_{\GL_2}(\A)}^{\GL_2(\A)}(\chi_1|\cdot|^{s}\boxtimes \chi_2|\cdot|^{-s})$, with $s\in i\R$ and $\chi_1,\chi_2$ unitary characters. Then
    \[\int_{[\SL_2]}\mathrm{Eis}(l,f_s)\, dl = 0.\]

\end{prop}
\begin{proof}
   First note that the restriction of $\mathrm{Eis}(l,f_s)$ to $\SL_2(\A)$ is an Eisenstein series of $\SL_2$ associated to the induced representation $I(\chi_1\chi_2^{-1},2s) := \mathrm{Ind}_{B_{\SL_2}(\A)}^{\SL_2(\A)}(\chi_1\chi_2^{-1}|\cdot|^{2s})$. We first prove that the integral over $[\SL_2]$ of these Eisenstein series is absolutely convergent.
    Following \cite[\S 1.1]{Harder:Chevalley}, there exists a positive constant $ c \in \R$ such that $$\SL_2(\A) = \SL_2(F) B(c) K,$$ where $K = \SL_2(\mathfrak{o})$ and 
    \[B(c) := \{b= nt \in B_{\SL_2}(\A),\; \text{ with }\;t = \mathrm{diag}(x,x^{-1})\in T_{\SL_2}(\A),\;|x|^{2}\geq q^{-c}\}.\]
    Therefore, we have
    \[\int_{[\SL_2]}\mathrm{Eis}(l,f_s)dl = \int_{(B(c)\cap \SL_2(F))\setminus B(c)}\int_{K}\mathrm{Eis}(bk,f_s)dk \, db.\]
    Let $A(c) :=  B(c) \cap T_{\SL_2}(\A)$. Writing $b=na$, with $n\in N_{\GL_2}(\A)$ and
$a=\mathrm{diag}(x,x^{-1})\in T_{\SL_2}(\A)$, the Haar measure decomposes as
\[
db=\delta_B(a)^{-1}\,dn\,da,
\, \text{with }
\delta_B(a)=|x|^2.
\]
Therefore, since $(B(c)\cap \SL_2(F)) = B_{\SL_2}(F)$, collapsing the integral with respect to $N=N_{\GL_2}$ we get 
    \[\int_{B(c)\cap \SL_2(F)\backslash B(c)}\int_{K}\mathrm{Eis}(bk,f_s)dk \, db = \int_{ (A(c)\cap \SL_2(F))\backslash A(c)}\int_{K} \delta_B(a)^{-1} \mathrm{Eis}_{N}(ak,f_s)dk \, da.\]
    Explicitly, the constant term of the Eisenstein series along $[N]$ is $$\mathrm{Eis}_{N}(ak,f_s) = f_s(ak)+(M(s)f_s)(ak),$$ where $M(s)$ is the classical intertwining operator \[M(s):I(\chi_1\chi_2^{-1},2s)\to I(\chi_1^{-1}\chi_2,-2s).\]
    Since $K$ is compact, one reduces to show that the integral over $A(c)$ is absolutely convergent. By an explicit computation, we have that $f_s(ak)+(M(s)f_s)(ak)$ equals
    \[\chi_1(x)\chi_2^{-1}(x)|x|^{2s+1}f_s(k)+\chi_1^{-1}(x)\chi_2(x)|x|^{-2s+1}M(s)f_s(k),\]
    where we have denoted $a = {\rm diag}(x,x^{-1})$. Let us factor 
    \[T_{\SL_2}(F)\backslash A(c) = \bigsqcup_{n\leq c} C^n,\]
    where $C^n := F^{\times}\backslash \{x\in \A^{\times}\,:\, |x|^2 =  q^{-n}\}$. Then, using that $\delta_B(a)=|x|^2$, we have \begin{align*}\left|\int_{T_{\SL_2}(F)\setminus A(c)}\chi_1(x)\chi_2(x)^{-1}|x|^{2s-1}dx\right| &\leq \sum_{n\leq c}\int_{C^n}|x|^{-1}dx\\&= \mathrm{vol}(C^0)\sum_{n\leq c}q^{n},
    \end{align*}
    where we have used $\mathrm{vol}(C^0)$ is finite. Similarly, a bound for the term corresponding to $(M(s)f_s)(ak)$ is obtained.  Since $\sum_{n\leq c}q^{n}$ converges, we thus obtain the absolute convergence of the integral. This in turn implies that integrating along $[\SL_2]$ gives an element
    \[\int_{[\SL_2]}\mathrm{Eis}(l,\cdot)dl\in\mathrm{Hom}_{\SL_2(\A)}(I(\chi_1\chi_2^{-1},2s),\kk).\]
    However, this ${\rm Hom}$ space is trivial since $I(\chi_1\chi_2^{-1},2s)$ is irreducible. This concludes the proof.
   \end{proof}
Let us introduce the notation for subgroups that are involved in the evaluation of the automorphic period $P_W$ on $\mathrm{GSpin}_6$ introduced in \S \ref{subsec:stackyJS}. We denote by $S' = UL' \subset \GL_4\times \GL_1$, where $U \simeq \mathrm{Mat}_2$ is the maximal unipotent subgroup of the Siegel parabolic of $\GL_4$ and $L'$ is the image of $\GL_2$ under the map
    \begin{align}\label{Embedding:L:Shalika}
        \GL_2&\to \GL_4\times \GL_1,\\
        m &\mapsto \left(  \left[\begin{smallmatrix} m& \\ &m\end{smallmatrix}\right],\mathrm{det}(m)\right).\nonumber
    \end{align}
Note that $S' \subset \mathrm{GSpin}_6\simeq \{(g,\lambda)\in \GL_4\times \GL_1\,:\,\mathrm{det}(g)\lambda^{-2} = 1\}$. Moreover, denote by $$\widetilde{L} = \{\left[\begin{smallmatrix} m& \\ &m\end{smallmatrix}\right],\;m\in \SL_2\},$$  by $\widetilde{L}'$ the image of $\SL_2$ into $\mathrm{GSpin}_6$ under the map \eqref{Embedding:L:Shalika}, and $\widetilde{S}' := U\widetilde{L}'$. The automorphic period $P_W$ on $\mathrm{GSpin}_6$ is given by
\begin{equation}\label{defPWGSpin6}P_W(f) =  \Delta^{-\frac{7+\varepsilon_\Sh}{4}}\int_{[U\widetilde{L}']}f(ule^{\check{\lambda}_\Sh}(\partial^{1/2}))\Psi(u)du \,dl,\end{equation}
following \S \ref{subsec:stackyJS}. In what follows, we study the value of $P_W$ on a cusp form restricted from $\mathrm{GL}_4 \times \mathrm{GL}_1$. 

\begin{prop}\label{Spectral:Expansion:Shalika}
    Let $\Pi\boxtimes\chi$ be a cuspidal automorphic representation of $\GL_4\times\GL_1$. Given a cusp form $F \in \Pi$, we have 
    \[P_W(F \times \chi|_{\mathrm{GSpin}_6(\A)}) = \frac{1}{|\Sigma(\omega_\Pi^{-1}\chi^{-2})|} \cdot \sum_{\substack{\varepsilon: [\GL_1] \to \kk^\times\\\varepsilon^2 = \omega_\Pi^{-1}\chi^{-2}}}P_M(F \times (\varepsilon\chi)^{-1}),\]
    where $\omega_\Pi$ is the central character of $\Pi$ and $\Sigma(\omega_\Pi^{-1}\chi^{-2})$ denotes the set of Hecke characters $\varepsilon:[\mathrm{GL}_1] \to \kk^\times$ squaring to $\omega_{\Pi}^{-1}\chi^{-2}$.
\end{prop}
Note that since all those $\varepsilon$ satisfying $\varepsilon^2 = \omega_\Pi^{-1}\chi^{-2}$ have bounded ramification, the sum on the right hand side is finite.
\begin{proof}
    Starting with an arbitrary cusp form $F \in \Pi$, we restrict $F\boxtimes \chi$ to $S'(\A)$ and analyze its spectral expansion. Since $F\boxtimes \chi|_{Z_L'} = \omega_\Pi\chi^{2}$, the relevant part of the $\GL_2$-spectrum has central character $\omega_\Pi^{-1}\chi^{-2}$. Then $F\boxtimes \chi|_{S'(\A)}$ can be spectrally expanded as 
    \begin{small}
    \[\sum_{\substack{\psi\in L^2([U])\\f\in L^2_{\mathrm{disc}}([\GL_2], \, \omega_\Pi^{-1}\chi^{-2})\\\psi,f,\; \mathrm{orthonormal}}}\langle F_{\psi}\boxtimes\chi,f\rangle \overline{\psi} \overline{f}+\sum_{(\chi_1,\chi_2)\in \mathcal{U}}\int_{-\infty}^{\infty}\sum_{\substack{f_s\in I(\chi_1,\chi_2,is)}}\langle F_{\psi},\mathrm{Eis}(\cdot,f_{is})\rangle \mathrm{Eis}(\cdot,f_{is})ds\,\overline{\psi}.\]
    \end{small}
  Here, $f$ ranges through an orthonormal basis of the discrete spectrum of $L^2([\mathrm{GL}_2])$ with central character $\omega_\Pi^{-1}\chi^{-2}$, while $F_\psi$ denotes the Fourier coefficient obtained by integrating $F$ over $[U]$ against $\psi$ using the self-dual measure on $[U]$.  Given a pair of unitary characters $(\chi_1,\chi_2)$ of of $[\G_m]$ and $s \in \R$, we have written 
    $$I(\chi_1,\chi_2,is) := \mathrm{Ind}_{B_{\GL_2}(\A)}^{\GL_2(\A)}(\chi_1|\cdot|^{is}\boxtimes \chi_2|\cdot|^{-is}).$$ 
Then, the set $\mathcal{U}$ denotes the set of such pairs $(\chi_1,\chi_2)$ so that $I(\chi_1,\chi_2,is)$ has central character $\omega_\Pi^{-1}\chi^2$.
The terms in this expansion are all well defined: the function $F$ is a cusp form on $\GL_4$ and we are restricting to the reductive subgroup $\GL_2$, therefore all the integrals are absolutely convergent. Moreover, since $[U]$ is compact abelian, we have the orthogonal decomposition
    \[L^2([U]) = \bigoplus_{\psi:[U]\to \kk^{\times}}\psi.\]
    We now integrate a right translation of $F\boxtimes\chi|_{S'}$ by $\partial^{1/2}$ over $[\widetilde{S}'] \simeq [U \SL_2]$ against the Shalika character $\Psi$, which gives a scalar multiple of $P_W$
    \begin{align}  \int_{[U \SL_2 ]}F(ul\partial^{1/2})\Psi(u) d^\psi u \, dl &= \Delta^{\frac{7 + \varepsilon_\Sh-2\mathrm{dim}(U)}{4}}P_W(F\boxtimes \chi|_{\mathrm{GSpin_6(\A)}})\\
    &= \Delta^{7/4} \cdot  P_W(F\boxtimes \chi|_{\mathrm{GSpin_6(\A)}}),
    \end{align}
    where the right translation by $\partial^{1/2}$ implicitly means the right translation by $e^{\check{\lambda}_\Sh}(\partial^{1/2})$, where the global power of $\Delta^{7/4}$ comes from
\begin{equation} \label{equation stupid 7/4}
    \frac{7+\varepsilon_\Sh-2\mathrm{dim}(U)}{4} = \frac{7+ 8-2 \cdot 4}{4} = \frac{7}{4}.
\end{equation}
    Using the spectral expansion above, we obtain that $P_W(F\boxtimes \chi|_{\mathrm{GSpin_6(\A)}})$ can be compared to the spectral expansion
    \begin{align*}
 &\sum_{\substack{\varepsilon: [\GL_1] \to \kk^\times\\\varepsilon^2 = \omega_\Pi^{-1}\chi^{-2}}}\langle r_{\partial^{1/2}} \cdot F_{\Psi}\boxtimes\chi,\varepsilon\rangle\\
 &+\sum_{(\chi_1,\chi_2)\in \mathcal{U}}\int_{-\infty}^{\infty}\sum_{\substack{f_s \in I(\chi_1,\chi_2,is)}}\langle r_{\partial^{1/2}} \cdot F_{\Psi},\mathrm{Eis}(\cdot,f_{is})\rangle \int_{[\SL_2]}\mathrm{Eis}(l,f_{is})dlds.
\end{align*}
     To arrive at the preceding formula, we have used the following facts. First, $\int_{[\SL_2]}f(r)dr = 0$ for every cuspidal automorphic form $f$ of $\GL_2$. Second, the remaining part of the discrete spectrum of $\GL_2$, i.e., the residual spectrum, consists of characters of $[\G_m]$ composed with the determinant, hence their integral over $[\SL_2]$ equals the Tamagawa number of $\SL_2$, which is $1$. Finally, $\int_{[U]}\overline{\psi}(u)\Psi(u)du \neq 0$ if and only if $\psi = \Psi$, in which case it contributes $\mathrm{vol}([U]) = \Delta^{\frac{\mathrm{dim}(U)}{2}} = \Delta^2$. Dropping the Eisenstein terms by applying Proposition \ref{Auxiliary:Lemma:Spectral:Expansion}, and identifying $\langle r_{\partial^{1/2}} \cdot F_{\Psi}\boxtimes \chi, \varepsilon\rangle = \Delta^{7/4} \cdot P_{M}(F \times \varepsilon^{-1}\chi^{-1})$, lets us conclude that 
      \begin{align}    P_W(F\times\chi|_{\mathrm{GSpin}_6(\A)}) &= \Delta^{-7/4} \cdot \sum_{\substack{\varepsilon: [\GL_1] \to \kk^\times\\\varepsilon^2 = \omega_\Pi^{-1}\chi^{-2}}}\langle r_{\partial^{1/2}} \cdot F_{\Psi}\boxtimes\chi,\varepsilon\rangle \nonumber\\
      &= \frac{1}{|\Sigma(\omega_\Pi^{-1}\chi^{-2})|} \cdot \sum_{\substack{\varepsilon: [\GL_1] \to \kk^\times\\\varepsilon^2 = \omega_\Pi^{-1}\chi^{-2}}}P_M(F \times (\varepsilon\chi)^{-1}). \label{final eq 5.8}
      \end{align} 
\end{proof}
\begin{rem}
   \textit{A priori}, even if one assumes that $\Pi \boxtimes \chi$ is unramified in Proposition \ref{Spectral:Expansion:Shalika}, it may not be the case that the twisted representation $\Pi \boxtimes \chi\varepsilon$ is, since only $\varepsilon^2 = \omega_\Pi^{-1}\chi^{-2}$ is guaranteed to be unramified. However we will soon find out that, for our purposes, these $\varepsilon$'s can always be taken to be unramified, so the reader may wish to assume this while reading.
\end{rem} 

Applying Jacquet--Shalika's duality for the period $P_M$, we may establish the period duality for $P_W$ as enunciated in \S \ref{section Shalika model}.
\begin{cor}[Theorem \ref{conjecture stacky Shalika dual pair}]\label{theorem stacky Shalika dual pair}
    Let $\mu$ be an unramified Whittaker normalized cusp form on $G = \mathrm{GSpin}_6$ with $L$-parameter $\varphi_\mu$ valued in $\check{G} = \mathrm{GSO}_6$. Consider the (Hamiltonian) actions $G \acts W$ and $\check{G} \acts \check{W}$ defined in \S \ref{subsec:stackyJS}. Then
    $$P_W(\mu) = L_{\check{W}}(\varphi_\mu).$$
\end{cor}
\begin{proof}
    By the preceding proposition, we may write the automorphic $W$-period of a cusp form $f$ on $G$ as a sum over $\varepsilon$-twisted usual Shalika periods on the group $\mathrm{GL}_4 \times \Gm$. The latter, by Jacquet--Shalika's Theorem \ref{theorem Jacquet-Shalika}, may be evaluated as a sum over $\varepsilon$-twisted $L$-values: fixing a preferred lift  $\widetilde{\varphi}_\mu: \Gamma_F \to \mathrm{GL}_4 \times \Gm$ of the $L$-parameter $\varphi_\mu$ of $\mu$, it is the sum over $\varepsilon$ of the $L$-values $L(1,\widetilde{\varphi}_\mu \otimes \varepsilon, \check{\mathfrak{g}}/\overline{\mathfrak{h}}') = L(1, \varphi_\mu, \wedge^2_0)$. Thus,  $P_W(\mu)$ is precisely the spectral $\check{W}$-period of $\varphi_\mu$ given in \eqref{spectral L W check}, as explained in the paragraph following the statement of Conjecture \ref{conjecture stacky Shalika dual pair}.
\end{proof}

We continue on to finer information regarding the local spherical functions associated to the periods $P_W$ and $P_M$. To this end, fix an unramified cuspidal automorphic representation $\sigma$ of $\mathrm{GSpin}_6$. Since the morphism $\mathrm{GSpin}_6\to \GL_4\times\GL_1$ is injective, \cite[Theorem 1.1.1]{Labesse:Sch} implies that there exists a cuspidal automorphic representation $\Pi\boxtimes\chi$ of $\GL_4\times\GL_1$ such that $\sigma\subset \Pi\boxtimes\chi|_{\mathrm{GSpin}_6(\A)}$. We fix such a representation $\Pi\boxtimes\chi$ for the rest of the section and deduce in the following lemma that, if it were to be ramified, it must be so in a very mild way. 
\begin{lem}\label{Description:Representations:Restriction}
    Let $v$ be a place of $F$. Then there exists an unramified representation $\widetilde{\Pi}_v\boxtimes\widetilde{\chi}_v$ of $\GL_4(F_v)\times\GL_1(F_v)$ and a character $\eta_v:F_v^{\times}\to \kk^{\times}$ such that $\sigma_v$ occurs in the restriction to $\mathrm{GSpin}_6(F_v)$ of $\widetilde{\Pi}_v\boxtimes\widetilde{\chi}_v$ and
    \[\Pi_v\boxtimes\chi_v\simeq (\widetilde{\Pi}_v\otimes \eta_v\circ \mathrm{det})\boxtimes\widetilde{\chi}_v\eta_v^{-2}.\]
\end{lem}

\begin{proof}
    By the proof of \cite[Proposition 5.3.3]{Labesse:Sch}, the unramified (irreducible) representation $\sigma_{v}$ occurs in the restriction to $\mathrm{GSpin}_6(F_v)$ of some unramified representation $\widetilde{\Pi}_v\boxtimes\widetilde{\chi}$ of $\GL_4(F_v)\times \GL_1(F_v)$. On the other hand, by \cite[\S 6.1]{Asgari:Choiy} the representations of $\GL_4(F_v)\times \GL_1(F_v)$ whose restriction to $\mathrm{GSpin}_6(F_v)$ contains $\sigma_v$ differ from one another up to twisting by a character on $$ (\GL_4(F_v)\times\GL_1(F_v)) / \mathrm{GSpin}_6(F_v) \simeq F_v^\times.$$  Explicitly, each is obtained by twists
    \[(\Pi_v\boxtimes\chi_v) \otimes \eta_v :=(\Pi_v\otimes \eta_v\circ\mathrm{det})\boxtimes \chi_v\eta_v^{-2},\]
    for some character $\eta_v:F_v^{\times}\to \kk^{\times}$. Since both restrictions of $\Pi_v\boxtimes \chi_v$ and $\widetilde{\Pi}_v\boxtimes\widetilde{\chi}_v$ contain $\sigma_v$, we conclude that there exists a character $\eta_v$ satisfying the statement of the lemma.
\end{proof}
\begin{rem} Note that, when $\Pi_v\boxtimes\chi_v$ is unramified, we can take $\eta_v = 1$.
\end{rem}
\begin{lem}\label{TadicLemma}
    Let $\Pi_v\boxtimes\chi_v$ be an (irreducible) generic representation of $\GL_4(F_v)\times\GL_1(F_v)$, the restriction $\Pi_v\boxtimes\chi_v|_{\mathrm{GSpin}_6(F_v)}$ is multiplicity free, containing at most one unramified constituent.  
\end{lem}
\begin{proof}
  By \cite[Lemma 2.1]{Tadic} and the discussion after it, we know that $\Pi_v\boxtimes\chi_v|_{\mathrm{GSpin}_6(F_v)}$ decomposes as 
  \[\Pi_v\boxtimes\chi_v|_{\mathrm{GSpin}_6(F_v)} = m\bigoplus_{\sigma\in \mathcal{O}(\Pi_v\boxtimes\chi_v)}\sigma,\]
  with $\mathcal{O}(\Pi_v\boxtimes\chi_v)$ being the set of equivalence classes of all representations of $\mathrm{GSpin}_6(F_v)$ which are isomorphic to a subrepresentation of $\Pi_v\boxtimes\chi_v|_{\mathrm{GSpin}_6(F_v)}$ and $m\in \Z_{>0}$ denoting the common multiplicity of each $\sigma$ in $\Pi_v\boxtimes\chi_v$. Applying \cite[Proposition 2.8]{Tadic}, we have that $m =1$. Now, observe that, by \cite[Lemmas 2.1 and 2.4]{Gelbart:Knapp}, we have the equality $\mathcal{O}(\Pi_v\boxtimes\chi_v) = \mathcal{O}((\Pi_v\boxtimes\chi_v) \otimes \nu)$, for any character $\nu$ on $$ (\GL_4(F_v)\times\GL_1(F_v)) / \mathrm{GSpin}_6(F_v).$$
  Moreover, by \cite[\S 6.1]{Asgari:Choiy}, $\mathcal{O}(\Pi_v\boxtimes\chi_v)$ corresponds to the $L$-packet for $\mathrm{GSpin}_6$ of parameter ${\rm pr} \circ \varphi$, where $\varphi$ is the $L$-parameter of $\Pi_v\boxtimes\chi_v$ and ${\rm pr}$ is the projection from $\GL_4 \times \GL_1$ to $\mathrm{GSO}_{6}$ given by \eqref{ses_checkG}.  Hence, by this discussion and the proof of Lemma \ref{Description:Representations:Restriction}, any unramified $\sigma \in \mathcal{O}(\Pi_v\boxtimes\chi_v)$ appears in the $L$-packet of the same unramified $L$-parameter. As the latter contains a unique unramified consituent, the result follows. 
\end{proof}
\begin{rem}\label{remark:on:spherical:vectors}
    Lemma \ref{TadicLemma} implies that, when $\Pi_v\boxtimes\chi_v$ is unramified, its spherical vector gives a spherical vector of the unique unramified representation in the restriction $\Pi_v\boxtimes\chi_v|_{\mathrm{GSpin}_6(F_v)}$. 
\end{rem}
Let $\mu_v\in \sigma_v$ be a spherical vector. By Lemma \ref{Description:Representations:Restriction} and Remark \ref{remark:on:spherical:vectors}, there exists a spherical vector $\widetilde{F}_v\boxtimes\widetilde{\chi}_v$ in the unramified representation $\widetilde{\Pi}_v\boxtimes\widetilde{\chi}_v$  such that 
\begin{equation}\label{Factorization:mu}\mu_v = (\widetilde{F}_v\otimes\eta_v\circ \mathrm{det})\boxtimes\widetilde{\chi}_v\eta_v^{-2},\end{equation}
where the latter can be seen inside $\Pi_v\boxtimes\chi_v$ via the isomorphism  \begin{equation}\label{iso_after_factorization}
    \Pi_v\boxtimes\chi_v\simeq (\widetilde{\Pi}_v\otimes \eta_v\circ \mathrm{det})\boxtimes\widetilde{\chi}_v\eta_v^{-2}.
\end{equation}

Finally we arrive at the analogue of \eqref{equation expliciting Shalika normalization} for $P_M$, where a global period is directly computed in terms of local spherical functions.

\begin{prop}\label{Shalika:GSpin:In:GL}
    Let $\mu \in \sigma$ be a Shalika normalized cusp form (in the sense that $P_W(\mu) = 1$ as in Definition \ref{def shalikanormalized}), then
    \begin{align*}
        P_W(\sigma(h,a)\mu) &= \Delta^{-7/4}(\omega_\Pi\chi^{2})(\partial^{1/2})\,\Omega_\mu^\circ \sum_{ \varepsilon \in S(\omega_{\Pi}^{-1}\chi^{-2})}\prod_{v} \,  \widetilde{\chi}_v(a_v\partial_v^{-1})\Omega_v^{{\rm ur},\varepsilon_v}(h_v),
    \end{align*}
    where \begin{itemize}
        \item $S(\omega_{\Pi}^{-1}\chi^{-2}):= \{\text{unramified characters } \varepsilon\,:\, \varepsilon^2 = \omega_\Pi^{-1}\chi^{-2} \},$ 
        \item $\Omega_\mu^\circ = \frac{\Delta^{7/4} \,\omega_\Pi^{-1}(\partial^{1/2}) (\chi^{-1}\widetilde{\chi})(\partial)}{  |S(\omega_{\Pi}^{-1}\chi^{-2})|}$ is the normalizing constant ensuring $P_W(\mu) = 1$,
        \item  $\Omega_v^{{\rm ur},\varepsilon_v}$ is the unramified Shalika function attached to $\widetilde{F}_v$,  $(\widetilde{\chi}_v\varepsilon_v)^{-1}$, and $\psi_v^{\mathrm{ur}}$, normalized so that $\Omega_v^{{\rm ur},\varepsilon_v}\left(1\right)=1$, with $\widetilde{F}_v \boxtimes \widetilde{\chi}_v$ the vector corresponding to $\mu_v$ via \eqref{Factorization:mu}.
    \end{itemize}
\end{prop}

\begin{proof}
By the discussion above, given $\mu$, there is a factorizable cusp form $F\boxtimes \chi \in \Pi\boxtimes\chi$ such that any local component satisfies \eqref{Factorization:mu}. Observe that, a priori, $F\boxtimes \chi$ might be a ramified vector in $\Pi\boxtimes\chi$: its local components are obtained by twists of spherical vectors $\widetilde{F}_v \boxtimes \widetilde{\chi}_v$ via the (possibly ramified) character $\eta_v$ as in \eqref{iso_after_factorization}. Moreover, for every $g_v = (h_v,a_v)\in \mathrm{GSpin}_6(F_v)$, one has that 
    \begin{align*}\sigma_v(g_v)\mu_v &= (\widetilde{\Pi}_v\otimes (\eta_v\circ\mathrm{det})    )(h_v)  \widetilde{F}_v  \boxtimes \widetilde{\chi}_v(a_v)\eta_v^{-2}(a_v)\\ &= (\widetilde{\Pi}_v(h_v)\widetilde{F}_v)\boxtimes \widetilde{\chi}_v(a_v),\end{align*}
    where we used that $\eta_v$ is trivial on $\mathrm{GSpin}_6(F_v)$. 
    
   A similar computation to the one of \eqref{equation expliciting Shalika normalization} lets us write 
   \begin{align}
       P_W(\sigma(h,a)\mu)  &= \Delta^{-\frac{7+\varepsilon_\Sh }{4}}\int_{[U\widetilde{L}']}\mu(ul(h,a)e^{\check{\lambda}_\Sh}(\partial^{1/2}))\Psi(u)du \,dl \nonumber \\ &=\Delta^{-\frac{7+\varepsilon_\Sh }{4}}\int_{[U\widetilde{L}']}(F\boxtimes \chi)(ul  (h\partial^{1/2} \mathrm{Id},a\partial)a_0 )\Psi(u)du \,dl \nonumber \\ &= \Delta^{-\frac{7+\varepsilon_\Sh }{4}} (\omega_\Pi\chi^{2})(\partial^{1/2})\int_{[U\widetilde{L}']}(F\boxtimes \chi)(ul   (h,a)a_0)\Psi(u)du \,dl \nonumber \\ &= \Delta^{-7/4} (\omega_\Pi\chi^2)(\partial^{1/2})\int_{[U\widetilde{L}']}(F\boxtimes \chi)(ul  (h,a)a_0 )\Psi(u)d^\psi u \,dl, \label{PW first manipulations}
   \end{align}
     where recall that $a_0 = (a_\partial^{-1} , \partial^{-1})$  is the element introduced in \S \ref{subsubsecn=2} and where we have written $$ \Delta^{\frac{-7 -\varepsilon_\Sh + 2\mathrm{dim}(U)}{4}} = \Delta^{-7/4}.$$  
     Mimicking the computations that led to \eqref{final eq 5.8} in the proof of Proposition \ref{Spectral:Expansion:Shalika}, we thus have 
    \begin{align} P_W(\sigma(h,a)\mu)  &= \Delta^{-7/4} (\omega_\Pi\chi^2)(\partial^{1/2}) \nonumber \\  &\cdot \sum_{\substack{\varepsilon: [\GL_1] \to \kk^\times\\\varepsilon^2 = \omega_\Pi^{-1}\chi^{-2}}}   \int_{[U][\mathrm{PGL}_2]} F(u \left[\begin{smallmatrix}
        g & \\ & g 
    \end{smallmatrix}\right]  h a_\partial^{-1} ) \chi(\partial^{-1} \, a) \chi\varepsilon({\rm det}(g)) \Psi(u) d g d^\psi u \nonumber \\   
    &=  \Delta^{-7/4} (\omega_\Pi\chi^{2})(\partial^{1/2}) \Omega_\mu^\circ \nonumber \\ &\cdot  \sum_{\substack{\varepsilon: [\GL_1] \to \kk^\times\\\varepsilon^2 = \omega_\Pi^{-1}\chi^{-2}}} \prod_v \widetilde{\chi}(a_v \partial_v^{-1}) P_{M,v}(\widetilde{\Pi}_v( h_v a_{\partial_v}^{-1}) \widetilde{F}_v \times (\varepsilon_v\widetilde{\chi}_v)^{-1}), \label{eq for normalized things}
    \end{align}
    where in the first equality we used the  embedding of the Shalika subgroup $U \GL_2$ into $\mathrm{GSpin}_6$ of \eqref{Embedding:L:Shalika}, 
      $a_{\partial_v}$ is the $v$-component of the element $a_{\partial}$ introduced in \S \ref{subsubsecn=2}, and to establish the last equality we used the fact that the multiplicity of the local Shalika functionals at each place is at most one. Here, $\Omega_\mu^\circ$ is a global constant to be determined shortly when $\mu$ is Shalika normalized (in the sense of $P_W(\mu) = 1$, as in Definition \ref{def shalikanormalized}), and $P_{M,v}$ are the local Shalika functionals of \S\ref{Subsection:CS:Sakellaridis}.

    Note that each Shalika functional appearing in the right hand side of \eqref{eq for normalized things} coincides with the unramified Shalika function $\Omega_{\widetilde{F}_v}^{\varepsilon_v^{-1}\widetilde{\chi}_v^{-1}}$ for $\GL_4$, as we now explain. Firstly, if $v$ is a place at which $\Pi_v\boxtimes\chi_v$ is ramified, we have an identification of intertwining spaces
    \[ \mathrm{Hom}_{S'(F_v)}(\Pi_v\boxtimes \chi_v\varepsilon_v,\Psi_{v}) = \mathrm{Hom}_{S'(F_v)}(\widetilde{\Pi}_v\boxtimes \widetilde{\chi}_v\varepsilon_v,\Psi_{v}), \]
    where recall that $S'$ embeds in $\GL_4\times \GL_1$ via \eqref{Embedding:L:Shalika} and, using Lemma \ref{Description:Representations:Restriction}, recall that $\widetilde{\Pi}_v \boxtimes \widetilde{\chi}_v$ denotes an unramified representation containing $\sigma_v$ such that 
    \[\Pi_v\boxtimes\chi_v\varepsilon_v\simeq (\widetilde{\Pi}_v\otimes \eta_{v}\circ \mathrm{det})\boxtimes\widetilde{\chi}_{v}\varepsilon_v\eta_{v}^{-2}.\]
    In establishing the equality above, we have used that the twist by $\eta_{v}$ is trivial once restricted to $S'(F_v)$, see \eqref{Embedding:L:Shalika}. Hence, $\varepsilon_v$ can be assumed to be unramified. This is because 
\[ \mathrm{Hom}_{S'(F_v)}(\widetilde{\Pi}_v\boxtimes \widetilde{\chi}_v\varepsilon_v,\Psi_{v}) = \mathrm{Hom}_{U(F_v)\GL_2(F_v)}(\widetilde{\Pi}_v,\Psi_{v} \times (\widetilde{\chi}_{v}\varepsilon_v)^{-1}), \]
and, since $\widetilde{\Pi}_v$ and $\widetilde{\chi}_{v}$ are unramified, the latter is zero if 
$\varepsilon_v$ is ramified. Thus, the Shalika function $$P_{M,v}((\widetilde{\Pi}_v(h_va_{\partial_v}^{-1})\widetilde{F}_v) \times (\widetilde{\chi}_{v}\varepsilon_v)^{-1})$$ must coincide with $\Omega_{\widetilde{F}_v}^{(\widetilde{\chi}_v\varepsilon_v)^{-1}}(h_va_{\partial_v}^{-1})$. 
In particular, since 
$$\Omega_{\widetilde{F}_v}^{(\widetilde{\chi}_v\varepsilon_v)^{-1}}(h_v a_{\partial_v}^{-1})= \Omega_v^{{\rm ur},\varepsilon_v}\left (h_v\right),$$ with $\Omega_{v}^{{\rm ur},\varepsilon_v}$ the unramified Shalika function associated to $\tilde{F}_v$, $(\widetilde{\chi}_{v}\varepsilon_v)^{-1}$, and the unramified additive character $\psi_v^{\rm ur}$ (normalized so that $\Omega_{v}^{{\rm ur},\varepsilon_v}(1)=1$), 
we can write \eqref{eq for normalized things} as \begin{align}\label{eq for normalized things 2}P_W(\sigma(h,a)\mu)  &=\Delta^{-7/4} (\omega_\Pi\chi^{2})(\partial^{1/2})\, \Omega_\mu^\circ \sum_{\substack{\varepsilon: [\GL_1] \to \kk^\times \text{ unr. }\\\varepsilon^2 = \omega_\Pi^{-1}\chi^{-2}}}\prod_{v}\widetilde{\chi}_v(a_v  \varpi_v^{-2m_v})\Omega_v^{{\rm ur},\varepsilon_v}\left (h_v\right).
    \end{align} 
Finally, we compute the normalizing constant $\Omega_\mu^\circ$ when $\mu$ is Shalika normalized. To do this, we set $(h,a) = \mathrm{id}$ in \eqref{eq for normalized things 2} to get $$ 1  = \Delta^{-7/4} (\omega_\Pi\chi^{2})(\partial^{1/2})\,\Omega_\mu^\circ \, \widetilde{\chi}(\partial^{-1}) |S(\omega_{\Pi}^{-1}\chi^{-2})|,$$
hence $$ \Omega_\mu^\circ = \frac{\Delta^{7/4} \,\omega_\Pi^{-1}(\partial^{1/2}) (\chi^{-1}\widetilde{\chi})(\partial)}{  |S(\omega_{\Pi}^{-1}\chi^{-2})|},$$ as desired. 
\end{proof}

\begin{rem}
    Observe that $S(\omega_{\Pi}^{-1}\chi^{-2})$ is a torsor under the group of unramified quadratic Hecke characters
    \[\mathrm{Hom}([\Gm],\pm 1)\simeq \mathrm{Hom}(\mathrm{Pic}(\Sigma)(\F_q),\pm 1),\]
    and thus it is a finite set.
\end{rem}

\begin{rem}\label{Langlands:Parameters:GSpin}
    Let $\varphi_{\Pi_v}\times\chi_v$ denote the $L$-parameter attached to $\Pi_v\boxtimes\chi_v$. According to \cite[\S 6.1]{Asgari:Choiy}, the Langlands parameter of $\sigma_v$ is obtained by composition of $\varphi_{\Pi_v} \times\chi_v$ with the natural projection
    \[\mathrm{pr}:\GL_4(\kk)\times \GL_1(\kk)\to (\GL_4(\kk)\times \GL_1(\kk))/(z\mathrm{Id},z^{-2}).\]
    Furthermore, given any character $\eta_v$, one has
    \[\mathrm{pr}\circ (\varphi_{(\Pi_v\otimes\eta_v\circ\mathrm{det})}\times\chi_v\eta^{-2}_v) = \mathrm{pr}\circ (\varphi_{\Pi_v} \times\chi_v).\]
\end{rem}
From now on,  denote by $\mathrm{Sat}_{\sigma_v}$ and $\mathrm{Sat}_{\widetilde{\Pi}_v}$  the Satake parameters of $\sigma_v$ and $\widetilde{\Pi}_v$, respectively. Recall, from \S \ref{Subsection:CS:Sakellaridis}, that  $\Omega_v^{{\rm ur},\varepsilon_v}$ is non-trivial if and only if $\widetilde{\Pi}_v$ has central character $(\widetilde{\chi}_v\varepsilon_v)^{-2}$ (which is true by assumption on $\varepsilon_v$) and  $\mathrm{Sat}_{\widetilde{\Pi}_v}\in \GSp_4(\kk)$ with similitude equal to $(\widetilde{\chi}_v\varepsilon_v)^{-1}(\varpi_v)$. Consequently by Lemma \ref{Description:Representations:Restriction} we have that
\[\mathrm{Sat}_{\sigma_v} = \mathrm{pr}(\mathrm{Sat}_{\widetilde{\Pi}_v}\times  \widetilde{\chi}_v(\varpi_v)) \] lies in the image of $\overline{H}'(\kk)$, defined in \eqref{GSp4embedsintoGSO33}. Moreover, any element in the preimage of $\mathrm{Sat}_{\sigma_v}$ via $\iota$ is of the form $$\mathrm{Sat}_{\widetilde{\Pi}_v}\times \widetilde{\chi}_v(\varpi_v) \times   \tilde{\varepsilon}_v(\varpi_v) \in H'(\kk),$$ for a given square-root $\tilde{\varepsilon}_v(\varpi_v)$ of ${\rm det}( \mathrm{Sat}_{\widetilde{\Pi}_v})\widetilde{\chi}_v^2(\varpi_v) = \varepsilon^{-2}_v(\varpi_v)$. We fix $$\mathrm{Sat}_{\sigma_v}^\iota :=\mathrm{Sat}_{\widetilde{\Pi}_v}\times \widetilde{\chi}_v(\varpi_v) \times   \varepsilon^{-1}_v(\varpi_v) \in H'(\kk)$$ a preferred lift of $\mathrm{Sat}_{\sigma_v}$.

Observe that any irreducible algebraic representation of $\overline{H}'$ can be written as $$V_{(a,b;c)}\boxtimes \varsigma_d,$$ where  $V_{(a,b;c)}$ is the irreducible representation of $\GSp_4(\kk)$ of highest weight $a \alpha_1 + b \alpha_2 + c \alpha_0$ and recall that $\varsigma_d$ is the weight $d$ character (as in \S \ref{Subsection:Groups:Notation}), satisfying $$a+b+2c-2d = 0.$$ 
In particular, $a+b \equiv 0$ [mod $2$] and $d = c + (a + b)/2.$ If $d' \in \Z$, we can further consider the $H'$-representation $$V_{(a,b;c)}\boxtimes \varsigma_d \boxtimes  \varsigma_{d'},$$ and denote 
\[s_{(a,b;c)}^{\varsigma_d,\varsigma_{d'}} := \mathrm{Tr}(V_{(a,b;c)}\boxtimes \varsigma_d\boxtimes  \varsigma_{d'}|\mathrm{Sat}_{\sigma_v}^\iota).\]
Similarly to \eqref{playing_with_traces}, we have \begin{align}
s_{(a,b;c)}^{\varsigma_d,\varsigma_{d'}} &= \varsigma_d(\mathrm{Sat}_{\sigma_v}^\iota)\varsigma_{d'}(\mathrm{Sat}_{\sigma_v}^\iota) s_{(a,b;c)} = (\widetilde{\chi}_v^d\varepsilon_v^{-d'})(\varpi_v) s_{(a,b;c)} \nonumber \\ &=  (\widetilde{\chi}_v^d\varepsilon_v^{-d'})(\varpi_v) s_{(0,0;c)} s_{(a,b;0)} = \varepsilon_v^{-d' - c}(\varpi_v)\widetilde{\chi}_v^{d-c} (\varpi_v)  s_{(a,b;0)}, \label{Relation:Between:Traces:GSpin}
\end{align}
with  $s_{(a,b;c)} := \mathrm{Tr}(V_{(a,b;c)}|\mathrm{Sat}_{\widetilde{\Pi}_v}).$

\begin{prop}\label{Final:Formula:Shalika:GSpin}
     For every $n \geq 0$, we have that 
    \[\Omega_v^{{\rm ur},\varepsilon_v}(\varpi_v^{(n,n)}) = (\widetilde{\chi}_v\varepsilon_v)^{-n}(\varpi_v)q_v^{-2n} \, \bigg(s_{(n,n;-n)}^{\varsigma_0,\varsigma_0}   - s_{(n-1,n-1;1-n)}^{\varsigma_{0},\varsigma_0} q_v^{-1}\bigg).\]
\end{prop}
\begin{proof}
    By Proposition \ref{CS:Formula:Shalika:simplified}
    and Remark \ref{Vanishing:Character}, if $n \geq 0$, 
     $\Omega_v^{{\rm ur},\varepsilon_v}(\varpi_v^{(n,n)})$ equals
    \begin{align*}
 q_v^{-2n} \, \bigg( s_{(n,n;0)} - \widetilde{\chi}_v^{-1}\varepsilon_v^{-1}(\varpi_v) s_{(n-1,n-1;0)}q_v^{-1}\bigg).
 \end{align*} 
The result then follows from \eqref{Relation:Between:Traces:GSpin}.
\end{proof}
 
\begin{rem}
    Notice that the right hand side of the formula of Proposition \ref{Final:Formula:Shalika:GSpin} does not depend on the lift $\mathrm{Sat}_{\sigma_v}^{\iota}$  chosen, as $\mathrm{Sat}_{\widetilde{\Pi}_v}\in \GSp_4(\kk)$ has similitude equal to $(\widetilde{\chi}_v\varepsilon_v)^{-1}(\varpi_v)$.
\end{rem}

\section{Singular relative Langlands duality}

In this final section, our aim is to prove the main result of the present manuscript. 
\begin{thm} \label{theorem main duality}
    Let $G = \mathrm{GSpin}_6$ and let $\check{G} = \mathrm{GSO}_6$ be its Langlands dual group. Let $(G,X)$ and $(\check{G}, \check{X})$ be the graded actions defined in \S \ref{Section:The:Space:G:X} and \S \ref{section Gcheck Xcheck}, respectively. Then $(G,X)$ and $(\check{G},\check{X})$ are weakly numerically dual with discrepancy $-5$ in both directions. 
\end{thm}
Theorem \ref{theorem main duality} will be a combination of the equality of periods
$$P_X = \Delta^{-5/4} \cdot L_{\check{X}}/L_{\check{W}}  \quad \text{ and } \quad P_{\check{X}} = \Delta^{-5/4} \cdot L_X/L_{\check{M}}$$
applied to Shalika normalized cusp forms on $G$ and $\check{G}$, respectively. These two statements appear below as Theorem \ref{thm:Final:PX} and Theorem \ref{thm PXcheck final}, respectively. We remind the reader that the notion of numerical duality should be interpreted in the sense of Remark \ref{remark convergence in numerical duality}. 

\subsection{\texorpdfstring{$(G,X)$}{(G,X)} on the automorphic side}\label{subsec_automorphic_side_G} 
Let $\sigma$ be a cuspidal tempered unramified representation of $G(\mathbb{A})$ and let $\mu \in \sigma$ be a factorizable unramified cusp form. By Lemma \ref{lemma automorphic induction}, we write 
\[P_{X}(\mu) = C_Y\int_{[H]}\mu(h)|\eta_{Y}(h)|^{1/2}\sum_{y\in \mathring{Y}(F)}\Phi_{Y}(y\partial^{1/2}h)dh,\]
where $\Phi_{Y}\in \mathcal{S}(Y(\mathbb{A}))$ is the characteristic function of the $\mathfrak{o}$-points of the variety $Y$ (with respect to the integral structure given in \eqref{Integral:Str:Y}), and where the constant $C_Y$, introduced in  \eqref{equation betaX}, is explicitly
\[C_Y = \Delta^{\frac{\mathrm{dim}(Y) - \mathrm{dim}(H)}{4}}|\eta_{Y}(\partial^{1/2})|^{1/2} = \Delta^{-5/2}.\]

\begin{lem}\label{lemma_1_on_G_side}
    We have $$P_X(\mu) =  C_Y\int_{\A^\times \cap \mathfrak{o}}|x|^3\int_{[P_0]}\mu(p h_{x\partial^{-1/2}} )|\eta_{Y}(p)|^{1/2} dp \, dx,$$ where for any $t \in \mathbb{A}^\times$ we have denoted $h_t := {\rm diag}(1,t^{-1},t,1)$ and $dp$ denotes the left invariant Haar measure on $P_0(\mathbb{A})$.
\end{lem}
\begin{proof}
    By Proposition \ref{Quotient:Equals:First:Variety} and the fact that $\mathrm{Ker}(\mathrm{H}^1(F, P_0)\to \mathrm{H}^1(F, H)) =0$, one has 
    \[\mathring{Y}(F) \simeq P_0(F)\setminus H(F).\]
    Therefore, by unfolding the sum with the integral, one obtains 
    \[P_{X}(\mu) = C_Y\int_{P_0(\A)\setminus H(\A)}\int_{[P_0]}\mu(ph)|\eta_{Y}(ph)|^{1/2}\Phi_{Y}(v_0\partial^{1/2}ph)dp \, dh,\]
    where $v_0 = (\langle f_2, f_1 \rangle, f_2\wedge f_1)\in \mathring{Y}(F)$, $dp$ denotes the left invariant Haar measure on $P_0(\mathbb{A})$, and by $d h$ we mean the product of  $| \eta_Y(h)|^{-1}$ with the measure attached to the eigenvolume form $\omega$ of Lemma \ref{eigen:char:G}. Since $P_0$ acts trivially on $v_0$, one obtains that $P_{X}(\mu)$ equals 
    \[C_Y\int_{P_0(\A)\setminus H(\A)}\Phi_Y(v_0\partial^{1/2}h)|\eta_{Y}(h)|^{1/2}\int_{[P_0]}\mu(ph)|\eta_{Y}(p)|^{1/2}dp \, dh.\]
Since $\partial^{1/2}$ acts on $v_0$ as the diagonal element $h_{\partial^{1/2}}:={\rm diag}(1, \partial^{-1/2} , \partial^{1/2} , 1)$, we make the change of variables $h_{\partial^{1/2}} h \mapsto h$ to get 
  \[C_Y\int_{P_0(\A)\setminus H(\A)}\Phi_Y(v_0 h)|\eta_{Y}(h)|^{1/2}\int_{[P_0]}\mu(ph_{\partial^{1/2}}^{-1} h)|\eta_{Y}(p)|^{1/2}dp \, dh,\]
  since $\eta_Y(h_{\partial^{1/2}})=1$. Now, recall that $P_0$ sits inside the Siegel parabolic $P$ of $H$; therefore, using the Iwasawa decomposition $H(\A) = P(\A) H(\mathfrak{o})$, the fact that $\mu$ is unramified, and that  $H(\mathfrak{o})$ preserves $\mathbf{Y}(\mathfrak{o})$  we reduce the integral to 
     \[C_Y\int_{P_0(\A)\setminus P(\A)}\Phi_Y(v_0 h)|\eta_{Y}(h)|^{1/2} \delta_P^{-1}(h)\int_{[P_0]}\mu(ph_{\partial^{1/2}}^{-1} h)|\eta_{Y}(p)|^{1/2} dp \, d_r h,\]
     where $d_r h$ now denotes the right invariant Haar measure on $P(\mathbb{A})$. To justify the appearance of $d_r h$, we have used the fact that the measure attached to the eigenvolume form $\omega$ restricted to $P(\A)$ equals $|\eta_Y(h)| \delta_P(h)^{-1} d_r h$. Consider the isomorphism \begin{align*}  \mathbf{G}_m &\to  P_0 \backslash P,\\
   x &\mapsto h_x = {\rm diag}(1,x^{-1},x,1).\end{align*}
    Then, since $\eta_{Y}(h_x) = 1$, the integral equals 
\[C_Y\int_{\A^\times}\Phi_Y(xv_0) \delta_P^{-1}(h_x)\int_{[P_0]}\mu(p h_{x\partial^{-1/2}})|\eta_{Y}(p)|^{1/2} dp \, dx.\]
Finally, recall that the integral structure $\mathbf{Y}$ of $Y$ given in \eqref{Integral:Str:Y} is determined by the choice of an admissible lattice $V_{\mathfrak{o}}$ of the irreducible $H$-representation $V$ of highest weight $(1,1;0)$, which is invariant under the action of $H({\mathfrak{o}})$. Thus, under the identification of Proposition \ref{Quotient:Equals:First:Variety}, an element $h\in P_0(\A) \backslash H(\A)$ is in $\mathbf{Y}(\mathfrak{o})$ if and only if  $ v_0 h \in \mathbf{Y}(\mathfrak{o})$. In particular, $h_x \in P_0(\A)\setminus P(\A)$ satisfies the latter if $x \in \A^\times \cap \mathfrak{o}$. Thus, we can write   
$$P_X(\mu) =  C_Y\int_{\A^\times \cap \mathfrak{o}}\delta_P^{-1}(h_x)\int_{[P_0]}\mu(p h_x h_{\partial^{1/2}}^{-1})|\eta_{Y}(p)|^{1/2} dp \, dx.$$
Finally, since $\delta_P^{-1}(h_x) = |x|^3$, we obtain the desired equality.
\end{proof}

Let $$I_{P_0}(\mu)(g):= \int_{[P_0]}\mu(p g)|\eta_{Y}(p)|^{1/2} dp$$ be the inner integral arising in Lemma \ref{lemma_1_on_G_side}.  In the following proposition, we manipulate $I_{P_0}(\mu)$ to show that $P_X(\mu)$ unfolds to the Shalika period.

\begin{lem}\label{lemma_2_on_G_side}
We have $$I_{P_0}(\mu)(g) =\Delta^{3/2} \int_{\A^\times} |y|^{-3/2} \Omega^{\rm tw}_\mu((a_y,y)  g)  dy, $$
where $a_y := {\rm diag}(y,y,1,1)$ and 
$$ \Omega^{\rm tw}_\mu(g) :=  \int_{[U][\mathrm{SL}_2]} \, \mu\left( u\left[\begin{smallmatrix} m& \\ &J_2{}^tm^{-1}J_2\end{smallmatrix}\right] g\right)\psi_1(u) \, d^\psi u \, dm,$$
with $\psi_1:[U]\to \C^\times$ the character defined by
    $\left[\begin{smallmatrix}I_2&X\\ &I_2\end{smallmatrix}\right]\mapsto   \psi\left(\mathrm{Tr}\left(\left[\begin{smallmatrix}1& \\ &-1\end{smallmatrix}\right]X\right)\right)$.
\end{lem}
    \begin{proof}
    Let $\widetilde{P} = \widetilde{M} U$ be the parabolic subgroup of $\GL_4$ with Levi component $\widetilde{M}\simeq \GL_2\times \GL_2$ and unipotent radical $U \simeq {\rm Mat}_2$. By the short exact sequence \eqref{ses_G}, there is a parabolic subgroup $Q$ of $G$ with maximal unipotent isomorphic to $U$.  Write $P_0 = N M_0$, hence identify $d p$ with $ \delta_{P_0}(m)^{-1} d u \,  d m$ (for instance, see \cite[IV (9)]{Cartier}).  we collapse $I_{P_0}(\mu)(g)$ over $[N]$ and Fourier expand the cusp form $\mu$ over $U$, to write
    $$\int_{[N]} \mu(n g) dn = \Delta^{3/2} \int_{[N]} \mu(n g) d^\psi n = \Delta^{3/2} \sum_{\alpha \in F} f_{U,\psi_\alpha}(g),$$
    where we have used the self-dual measure $d^\psi n$ to Fourier expand, and we denoted  by $\psi_\alpha:[U/N]\to \C^\times$ the character
    \[n(X):= \left[\begin{smallmatrix}I_2&X\\ &I_2\end{smallmatrix}\right]\mapsto \psi_{\alpha}(n(X)):= \psi(\mathrm{Tr}(A_{\alpha}X)), \text{ with } A_{\alpha} = \left[\begin{smallmatrix}\alpha& \\ &-\alpha\end{smallmatrix}\right],\]
     and $\mu_{U,\psi_\alpha}(g) :=  \int_{[U]}\mu(ug)\psi_\alpha(n)d^\psi n$. Thus, since $\delta_{P_0}(p) = |\eta_{Y}(p)|$, we have
    $$I_{P_0}(\mu)(g) = \Delta^{3/2} \int_{[M_0]} \delta^{-1/2}_{P_0}(m)\sum_{\alpha \in F} \mu_{U,\psi_\alpha}(mg) dm.$$ 
By direct computation, the group $M_0(F)$ acts on these characters as 
    \[\psi_{\alpha}(m_0 n(X)m_0^{-1}) = \psi_{\lambda^{-1}{\rm det}(m)\alpha}(n(X)) = \psi_{\lambda \alpha}(n(X)),\]
    where we have denoted $m_0 = m_0(m,\lambda)$ and used the equality $\lambda^2 = {\rm det}(m)$. Thus, the group $M_0(F)$ acts on the set of such characters with two orbits, represented by $\psi_0$ and $\psi_1$, with stabilizers $M_0(F)$ and 
    \[L(F) = \left\{\left[\begin{smallmatrix} m& \\ &J_2{}^tm^{-1}J_2\end{smallmatrix}\right]: m\in \SL_2(F)\right\},\]
    respectively. By cuspidality of $\mu$, $\mu_{\psi_0}$ vanishes, thus we can write 
    $$ I_{P_0}(\mu)(g) = \Delta^{3/2}\int_{L(F)\backslash M_0(\A)} \delta^{-1/2}_{P_0}(m) \mu_{U,\psi_1}(m g) dm.$$ 
   Collapse the sum over $[L]$ to write 
    $$I_{P_0}(\mu)(g) = \Delta^{3/2} \cdot \int_{L(\A)\backslash M_0(\A)} \delta^{-1/2}_{P_0}(m) \Omega^{\rm tw}_\mu( m g)  dm.$$ 
    Finally, fix the isomorphism \begin{align*}  \mathbf{G}_m &\to  L  \backslash M_0,\\
   y &\mapsto a_y = {\rm diag}(y,y,1,1),\end{align*}
   and use the fact that $|\eta_{Y}(a_y)| = \delta_{P_0}(a_y) = |y|^3$ (see Proposition \ref{eigen:char:G}) to write
   \begin{align*}
       I_{P_0}(\mu)(g) &= \Delta^{3/2}\int_{\A^\times} \delta_{P_0}(a_y)^{-1/2}  \Omega_\mu^{\rm tw}((a_y,y)  g)   dy\\ &=\Delta^{3/2} \int_{\A^\times} |y|^{-3/2} \Omega_\mu^{\rm tw}((a_y,y)  g)   dy.
   \end{align*} 
\end{proof}
We now recall some of the notation introduced in \S \ref{subsection Shalika for GSpin6}. By \cite[Theorem 1.1.1]{Labesse:Sch}, there is a cuspidal representation $\Pi\boxtimes\chi$ of $\GL_4\times \GL_1$ so that $\sigma$ is a subrepresentation of $\Pi\boxtimes\chi|_{G(\A)}$. Moreover, by Lemma \ref{Description:Representations:Restriction},  for each place $v$ of $F$ there exist an unramified representation $\widetilde{\Pi}_v\boxtimes\widetilde{\chi}_v$ and a character  $ \eta_v$ on $$ (\GL_4(F_v)\times\GL_1(F_v)) / G(F_v) \simeq F_v^\times$$
such that  \begin{equation*}\Pi_v\boxtimes\chi_v\simeq (\widetilde{\Pi}_v\otimes\eta_v)\boxtimes\eta_v^{-2}\widetilde{\chi}_v.\end{equation*} 
As in Remark \ref{remark:on:spherical:vectors} and Proposition \ref{Shalika:GSpin:In:GL}, we choose a factorizable cusp form $f\boxtimes\chi\in \Pi\boxtimes\chi$ such that $f_v\boxtimes\chi_v = \widetilde{f}_v\otimes\eta_v\boxtimes\widetilde{\chi}_v\eta_v^{-2}$ with $ \widetilde{f}_v\boxtimes\widetilde{\chi}_v$ a spherical vector in $\widetilde{\Pi}_v\boxtimes\widetilde{\chi}_v$, whose restriction to $G(F_v)$ equals $\mu_v$.

\begin{prop}\label{Prop_Shalika_twisted}  Let $\mu\in \sigma$ be a factorizable unramified Shalika normalized tempered cusp form. We have  
\[\Omega_\mu^{\rm tw}( a_yh_{x\partial^{-1/2}},y) = \Omega^{\circ}_{\mu}\sum_{ \varepsilon \in S(\omega_{\Pi}^{-1}\chi^{-2})}\prod_{v}\widetilde{\chi}_v(y_v)\Omega_{v}^{\mathrm{ur}, \varepsilon_v}(a_{y_v\partial_v}  h_{x_v\partial^{-1/2}_v}),\]
    where \begin{itemize}
        \item $S(\omega_{\Pi}^{-1}\chi^{-2})= \{\text{unramified characters } \varepsilon\,:\, \varepsilon^2 = \omega_\Pi^{-1}\chi^{-2} \},$ 
        \item $\Omega_\mu^\circ = \frac{\Delta^{7/4} \,\omega_\Pi^{-1}(\partial^{1/2}) (\chi^{-1}\widetilde{\chi})(\partial)}{  |S(\omega_{\Pi}^{-1}\chi^{-2})|}$ is the global constant appearing in Proposition \ref{Shalika:GSpin:In:GL},
        \item  $\Omega_v^{{\rm ur},\varepsilon_v}$ is the unramified Shalika function attached to $\widetilde{F}_v$,  $(\widetilde{\chi}_v\varepsilon_v)^{-1}$, and $\psi_v^{\mathrm{ur}}$, normalized so that $\Omega_v^{{\rm ur},\varepsilon_v}\left(1\right)=1$.
    \end{itemize}
\end{prop}
\begin{proof}
 Let $w_0 = {\rm diag}(1,1,-1,1) \in \GL_4(F)$. Proceeding as in Proposition \ref{Shalika:GSpin:In:GL}, we have \begin{align*}\Omega_\mu^{\rm tw}( a_yh_{x\partial^{-1/2}},y) &= \int_{[U][\mathrm{SL}_2]} \, (\widetilde{f}\boxtimes\widetilde{\chi} )\left( \left( u\left[\begin{smallmatrix} m& \\ &J_2{}^tm^{-1}J_2\end{smallmatrix}\right] a_yh_{x\partial^{-1/2}},y\right)\right)\psi_1(u) \, d^\psi u \, dm \\  
  & =  \widetilde{\chi}(y) \int_{[U][\mathrm{SL}_2]} \, \widetilde{f} \left( u\left[\begin{smallmatrix} m& \\ &m \end{smallmatrix}\right] w_0 a_yh_{x\partial^{-1/2}}\right)\Psi(u) \, d^\psi u \, dm \\  & = \widetilde{\chi}(y\partial) \Delta^{7/4} (\omega_\Pi\chi^{2})(\partial^{-1/2}) P_W( \widetilde{\Pi}(w_0 a_y  h_{x\partial^{-1/2}}a_{\partial})\widetilde{f}\boxtimes\widetilde{\chi}),
 \end{align*}
where, for the first equality, we have used that $h_{x\partial^{-1/2}}$ and $a_y$ embed into $G$ as 
 $h_{x\partial^{-1/2}} \mapsto \left( h_{x\partial^{-1/2}}, 1\right)$ and
 $a_y\mapsto \left(a_y , y\right)$, while the second one follows from the fact that,  for any $m \in \SL_2$, $$w_0\left[\begin{smallmatrix} m& \\ &J_2{}^tm^{-1}J_2\end{smallmatrix}\right]w_0^{-1} =  \left[\begin{smallmatrix} m& \\ &m\end{smallmatrix}\right],$$ that $\psi_1(w_0 u w_0^{-1}) = \Psi(u)$. Finally, for the third equality, we used \eqref{PW first manipulations} and took into account the missing twist by $a_0=(a_{\partial}^{-1},\partial^{-1})$ in $\Omega_\mu^{\rm tw}$. By Proposition \ref{Shalika:GSpin:In:GL} and the fact that $w_0$ commutes with $ a_y  h_{x\partial^{-1/2}}a_{\partial}$, we thus have that $\Omega_\mu^{\rm tw}( a_yh_{x\partial^{-1/2}},y)$ equals  \[\Omega^{\circ}_{\mu}\sum_{ \varepsilon \in S(\omega_{\Pi}^{-1}\chi^{-2})}\prod_{v}\widetilde{\chi}_v(y_v)\Omega_{v}^{\mathrm{ur}, \varepsilon_v}(a_{y_v\partial_v}  h_{x_v\partial^{-1/2}_v}),\]
 as desired. 
\end{proof}

\begin{rem}
    Note that the fact that $\mu$ is Shalika normalized implies that $$\Omega_\mu^{\rm tw}(a_0) = \Omega^{\circ}_{\mu},$$ 
   as one might expect. Indeed, we invite the reader to compare this with $\Omega_f^\circ$ in  \eqref{equation Omegaf0} and $W_f^0$ in \cite[\S 2.2]{CV}. 
\end{rem}

\begin{rem}
     If $F$ contains a primitive fourth root of unity, then $\Omega_\mu^{\rm tw}(g)$ easily corresponds to a constant multiple of  $P_W(\sigma( g a_0) \mu)$, by using conjugation by $(w_0, \sqrt{-1}) \in G(F)$.
\end{rem}

Combining the results above, we may express the period $P_X$ as a finite sum of Eulerian integrals.
\begin{thm}\label{final_unfolding_thm_G_side}  
    Let $\mu\in \sigma$ be a factorizable unramified Shalika normalized tempered cusp form on $G$. Then
   $$P_X(\mu)= C_Y'\cdot \Omega_\mu^\circ \cdot \sum_{ \varepsilon \in S(\omega_{\Pi}^{-1}\chi^{-2})}\prod_{v} Z(\widetilde{f}_v,\widetilde{\chi}_v\varepsilon_v) ,$$
   where
   \begin{itemize}
        \item $C_Y' = C_Y \cdot \Delta^{3/2} = \Delta^{-1},$
        \item $\Omega_\mu^\circ$ is the global constant of Proposition \ref{Prop_Shalika_twisted}, and
       \item $Z(\widetilde{f}_v,\widetilde{\chi}_v\varepsilon_v)$ is the local zeta integral 
       $$\sum_{n\geq 0}\sum_{a \geq  0} |\varpi_v^{3(n - a +3m_v)/2}|\varepsilon^{-1}_v(\varpi_v^{n-m_v})\widetilde{\chi}_v^{-1}(\varpi_v^{2m_v-a} )\Omega_{v}^{\mathrm{ur}, \varepsilon_v}({\varpi_v^{(a,a)}}).$$
   \end{itemize}
   
\end{thm} 
\begin{proof}
    Using Lemmas \ref{lemma_1_on_G_side}, \ref{lemma_2_on_G_side}, and Proposition \ref{Prop_Shalika_twisted}, we get
    \begin{align*}
        P_X(\mu) =  C_Y' \cdot \Omega_\mu^\circ \cdot &\sum_{ \varepsilon \in S(\omega_{\Pi}^{-1}\chi^{-2})}\prod_{v}  \int_{F_v^\times \cap \mathfrak{o}_v}|x_v|^3  \int_{F_v^\times} |y_v|^{-3/2}\widetilde{\chi}_v(y_v)\Omega_{v}^{\mathrm{ur}, \varepsilon_v}(a_{y_v\partial_v}  h_{x_v\partial_v^{-1/2}}) d y_v d x_v.
    \end{align*}
    Moreover, if $b_t : = {\rm diag}(t,1,t,1)$, we have \begin{align*}
        \Omega_{v}^{\mathrm{ur}, \varepsilon_v}(a_{y_v\partial_v}  h_{x_v\partial_v^{-1/2}}) &= (\widetilde{\chi}_v\varepsilon_v)^{-1}(x_v\partial_v^{-1/2})\Omega_{v}^{\mathrm{ur}, \varepsilon_v}(b_{x_v\partial_v^{-1/2}}^{-1}a_{y_v\partial_v}  h_{x_v\partial_v^{-1/2}}) \\ &= (\widetilde{\chi}_v\varepsilon_v)^{-1}(x_v\partial_v^{-1/2})\Omega_{v}^{\mathrm{ur}, \varepsilon_v}( a_{y_v x_v^{-1}\partial_v^{3/2}}).
    \end{align*}
    Thus, $P_X(\mu)$ equals  \begin{align*}
           &C_Y' \, \Omega_f^\circ  \!\!\!\!\!\!\sum_{ \varepsilon \in S(\omega_{\Pi}^{-1}\chi^{-2})}\prod_{v}  \int_{F_v^\times \cap \mathfrak{o}_v}|x_v|^3\varepsilon^{-1}_v(x_v\partial^{-1/2}_v)   \int_{F_v^\times} |y_v|^{-3/2}\widetilde{\chi}_v(y_vx_v^{-1} \partial^{1/2}_v)\Omega_{v}^{\mathrm{ur}, \varepsilon_v}(a_{y_v x_v^{-1}\partial^{3/2}_v}) d y_v d x_v.
    \end{align*}
    We now examine each local integral separately.
 Using $F_v^{\times}\cap \mathfrak{o}_v = \bigsqcup_{n\geq 0}\varpi_v^{n}\mathfrak{o}_v^{\times}$ and that all the data are unramified, each local integral is 
    \begin{align*}
    \sum_{n\geq 0}&|\varpi_v^{3n}|\varepsilon^{-1}_v(\varpi_v^{n-m_v})\sum_{r \in \ZZ}|\varpi_v^{-3r/2}|\widetilde{\chi}_v(\varpi_v^{r-n+m_v} )\Omega_{v}^{\mathrm{ur}, \varepsilon_v}(a_{\varpi_v^{r-n+3m_v}}).
    \end{align*}
    Since the conductor of $\psi_v^{\mathrm{ur}}$ is $\mathfrak{o}_v$, by standard arguments we see that 
    \[\Omega_{v}^{\mathrm{ur}, \varepsilon_v}(a_{\varpi_v^{s}}) = 0,\]
    whenever $s< 0$. Thus, $\Omega_{v}^{\mathrm{ur}, \varepsilon_v}(a_{\varpi_v^{r-n+3m_v}}) = 0$ if $r < n-3m_v$ and we can write the integral as 
    \begin{align*}
    \sum_{n\geq 0}\sum_{r \geq n-3m_v } |\varpi_v^{3n-3r/2}|\varepsilon^{-1}_v(\varpi_v^{n-m_v})\widetilde{\chi}_v(\varpi_v^{r-n+m_v} )\Omega_{v}^{\mathrm{ur}, \varepsilon_v}(a_{\varpi_v^{r-n+3m_v}}),
    \end{align*}
    We make the change of variables $a  = r-n+3m_v$ to write the sums as
    \begin{align*}
    \sum_{n\geq 0}\sum_{a \geq  0} |\varpi_v^{3(n - a +3m_v)/2}|\varepsilon^{-1}_v(\varpi_v^{n-m_v})\widetilde{\chi}_v^{-1}(\varpi_v^{2m_v-a} )\Omega_{v}^{\mathrm{ur}, \varepsilon_v}({\varpi_v^{(a,a)}}).
    \end{align*}
    \end{proof}

\subsubsection{The local computation}
We now compute the local integrals of Theorem \ref{final_unfolding_thm_G_side}. Recall that  $\Omega_{v}^{\mathrm{ur}, \varepsilon_v}$ is non-trivial if and only if $\widetilde{\Pi}_v$ has central character $(\widetilde{\chi}_v\varepsilon_v)^{-2}$ (which is true by assumption on $\varepsilon_v$) and  $\mathrm{Sat}_{\widetilde{\Pi}_v}\in \GSp_4(\kk)$ with similitude equal to $(\widetilde{\chi}_v\varepsilon_v)^{-1}(\varpi_v)$. Thus, 
\[\mathrm{Sat}_{\sigma_v} = \mathrm{pr}(\mathrm{Sat}_{\widetilde{\Pi}_v}\times  \widetilde{\chi}_v(\varpi_v))\]
lies in the image of $\overline{H}'(\kk)$, defined in \eqref{GSp4embedsintoGSO33}. Fix $\mathrm{Sat}_{\sigma_v}^\iota$ a preferred lift of $\mathrm{Sat}_{\sigma_v}$, which is of the form $$\mathrm{Sat}_{\sigma_v}^\iota :=\mathrm{Sat}_{\widetilde{\Pi}_v}\times \widetilde{\chi}_v(\varpi_v) \times   \varepsilon^{-1}_v(\varpi_v) \in H'(\kk).$$ 
As in the discussion preceding Proposition \ref{Final:Formula:Shalika:GSpin}, denote 
\[s_{(a,b;c)}^{\varsigma_d,\varsigma_{d'}} = \mathrm{Tr}(V_{(a,b;c)}\boxtimes \varsigma_d\boxtimes  \varsigma_{d'}|\mathrm{Sat}_{\sigma_v}^\iota).\]
Recall that $\wedge_0^2$ is the $5$-dimensional representation of $\GSp_4(\kk)$ of highest weight $(1,1;-1)$ and 
$\wedge_0^{\varsigma_{0}} = \wedge_0^2\boxtimes \varsigma_0$ is the corresponding irreducible algebraic representation of $\overline{H}'(\kk)$. Moreover, recall that $\sigma_v$ is assumed to be tempered.
\begin{lem}\label{lemma_on_Lstd_G_side} For ${\rm Re}(s) \gg 0$, we have an equality  
     $$\frac{L(s,\sigma_v,\wedge_0^{\varsigma_{0}} \otimes \varepsilon_v^{-1})}{L(2s,\varepsilon_v^{-2})} =\sum_{k = 0}^\infty q_v^{-ks}\varepsilon_v^{-k}(\varpi_v) s_{(k,k;-k)}^{\varsigma_0,\varsigma_0}.$$  
\end{lem}
\begin{proof}
    Write $s_{(k-2i,k-2i;-k+2i)}=\mathrm{Tr}(V_{(k-2i,k-2i;-k+2i)}|\mathrm{Sat}_{\widetilde{\Pi}_v})$. Thanks to \cite[Lemma 6.2]{CauchiGutiCS}, we have \begin{align*}
        L(s,\sigma_v,\wedge_0^{\varsigma_{0}} \otimes \varepsilon_v^{-1}) &= L(s,\widetilde{\Pi}_v,\wedge_0^2\otimes \varepsilon_v^{-1}) \\ &= \sum_{k=0}^\infty \sum_{i=0}^{\lfloor k/2\rfloor}q_v^{-ks} \varepsilon_v^{-k}(\varpi_v) s_{(k-2i,k-2i;-k+2i)}  \\ &=\sum_{k=0}^\infty \sum_{i=0}^{\lfloor k/2\rfloor}q_v^{-ks} \varepsilon_v^{-k}(\varpi_v) s_{(k-2i,k-2i;-k+2i)}^{\varsigma_0,\varsigma_0} ,
    \end{align*} 
    where, for the last equality, we have used  \eqref{Relation:Between:Traces:GSpin} to deduce that $$s_{(k-2i,k-2i;-k+2i)}^{\varsigma_0,\varsigma_0} = s_{(k-2i,k-2i;-k+2i)}.$$
    Use the Cauchy product formula to further write the sum as $$\left( \sum_{k=0}^\infty q_v^{-2ks}\varepsilon_v^{-2k}(\varpi_v)\right) \left( \sum_{k = 0}^\infty q_v^{-ks}\varepsilon_v^{-k}(\varpi_v) s_{(k,k;-k)}^{\varsigma_0,\varsigma_0} \right).$$
The result then follows.
\end{proof}

\begin{thm}\label{local_Zeta_G_side_thm}
We have $$Z(\widetilde{f}_v,\widetilde{\chi}_v\varepsilon_v) = \gamma_v\frac{L(1/2,\sigma_v,\wedge_0^{\varsigma_{0}} \otimes \varepsilon_v^{-1})}{L(1,\varepsilon_v^{-2})},$$
where $\gamma_v : = | \partial_v|^{9/4}( \varepsilon_v \widetilde{\chi}_v^{-2})(\partial_v^{1/2})$.
\end{thm}

\begin{proof}
    By Proposition \ref{Final:Formula:Shalika:GSpin}, for every $a \geq 0$, we have
    \[\Omega_v^{{\rm ur},\varepsilon_v}(\varpi_v^{(a,a)}) = (\widetilde{\chi}_v\varepsilon_v)^{-a}(\varpi_v)q_v^{-2a} \, \bigg(s_{(a,a;-a)}^{\varsigma_0,\varsigma_0}   - s_{(a-1,a-1;1-a)}^{\varsigma_0,\varsigma_0} q_v^{-1}\bigg),\]
    where recall \[s_{(a,a;-a)}^{\varsigma_0,\varsigma_0} = \mathrm{Tr}(V_{(a,a;-a)}\boxtimes \varsigma_0\boxtimes \varsigma_0|\mathrm{Sat}_{\sigma_v}^\iota).\]
We plug the formula into $Z(\widetilde{f}_v,\widetilde{\chi}_v\varepsilon_v)$ to write the sum as 
  $$| \partial_v|^{9/4}( \varepsilon_v \widetilde{\chi}_v^{-2})(\partial_v^{1/2}) \sum_{n\geq 0}\sum_{a \geq  0} q_v^{-3n/2 -a/2}\varepsilon_v(\varpi_v^{-n-a})\bigg(s_{(a,a;-a)}^{\varsigma_0,\varsigma_0} - s_{(a-1,a-1;1-a)}^{\varsigma_0,\varsigma_0} q_v^{-1}\bigg).$$

We first analyze 
$$I_1 := \sum_{a \geq  0} q_v^{-a/2}\varepsilon_v^{-a}(\varpi_v)\bigg(s_{(a,a;-a)}^{\varsigma_0,\varsigma_0} - s_{(a-1,a-1;1-a)}^{\varsigma_0,\varsigma_0} q_v^{-1}\bigg).$$
Use that $s_{(-1,-1;1)}^{\varsigma_0,\varsigma_0}=0$ to write 
\begin{align*}
    I_1 &=\sum_{a \geq  0} q_v^{-a/2}\varepsilon_v^{-a}(\varpi_v)s_{(a,a;-a)}^{\varsigma_0,\varsigma_0} - \sum_{a \geq  0} q_v^{-a/2-1} \varepsilon_v^{-a}(\varpi_v) s_{(a-1,a-1;1-a)}^{\varsigma_0,\varsigma_0}  \\ &=\sum_{a \geq  0} q_v^{-a/2}\varepsilon_v^{-a}(\varpi_v)s_{(a,a;-a)}^{\varsigma_0,\varsigma_0} - \varepsilon_v^{-1}(\varpi_v)q_v^{-3/2} \sum_{a \geq  0} q_v^{-(a-1)/2} \varepsilon_v^{1-a}(\varpi_v) s_{(a-1,a-1;1-a)}^{\varsigma_0,\varsigma_0} \\ 
    &= (1 - \varepsilon_v^{-1}(\varpi_v)q_v^{-3/2})\sum_{a \geq  0} q_v^{-a/2}\varepsilon_v^{-a}(\varpi_v)s_{(a,a;-a)}^{\varsigma_0,\varsigma_0}.
\end{align*}
By Lemma \ref{lemma_on_Lstd_G_side}, we obtain 

$$I_1 = \frac{L(1/2,\sigma_v,\wedge_0^{\varsigma_{0}} \otimes \varepsilon_v^{-1})}{L(1,\varepsilon_v^{-2})L(3/2,\varepsilon_v^{-1})}, $$
thus 
\begin{align*}
Z(\widetilde{f}_v,\widetilde{\chi}_v\varepsilon_v) &=\gamma_v \frac{L(1/2,\sigma_v,\wedge_0^{\varsigma_{0}} \otimes \varepsilon_v^{-1})}{L(1,\varepsilon_v^{-2})L(3/2,\varepsilon_v^{-1})} \sum_{n \geq 0}  q_v^{-3n/2}\varepsilon_v(\varpi_v^{-n}) \\ 
&= \gamma_v \frac{L(1/2,\sigma_v,\wedge_0^{\varsigma_{0}} \otimes \varepsilon_v^{-1})}{L(1,\varepsilon_v^{-2})L(3/2,\varepsilon_v^{-1})} \sum_{n \geq 0}  q_v^{-3n/2}\varepsilon_v(\varpi_v^{-n}) \\ &=  \gamma_v\frac{L(1/2,\sigma_v,\wedge_0^{\varsigma_{0}} \otimes \varepsilon_v^{-1})}{L(1,\varepsilon_v^{-2})},
\end{align*}
where we denoted $\gamma_v : = | \partial_v|^{9/4}( \varepsilon_v \widetilde{\chi}_v^{-2})(\partial_v^{1/2})$.
\end{proof}

Combining Theorems \ref{final_unfolding_thm_G_side} and \ref{local_Zeta_G_side_thm}, we have the final expression for the period $P_X$.
 \begin{thm}\label{thm:Final:PX}
    Let $\mu\in \sigma$ be a factorizable unramified Shalika normalized tempered cusp form  on $G$. Then,
   $$P_X(\mu)=\frac{\Delta^{-3/2} (\omega_\Pi^{-1}\chi^{-2})(\partial^{1/2})}{|S(\omega_{\Pi}^{-1}\chi^{-2})|}\sum_{ \varepsilon \in S(\omega_{\Pi}^{-1}\chi^{-2})} \varepsilon(\partial^{1/2})  \frac{L(1/2,\sigma,\wedge_0^{\varsigma_{0}} \otimes \varepsilon^{-1})}{L(1,\varepsilon^{-2})}.$$    
 \end{thm}

In particular, since $\varepsilon \in S(\omega_{\Pi}^{-1}\chi^{-2})$, switching $\varepsilon$ to $\varepsilon^{-1}$ for aesthetic reasons, the formula can be read as 
$$P_X(\mu)=\frac{\Delta^{-3/2}}{|S(\omega_{\Pi}\chi^2)|}\sum_{ \varepsilon \in S(\omega_{\Pi}\chi^{2})} \varepsilon(\partial^{-1/2})^3  \frac{L(1/2,\sigma,\wedge_0^{\varsigma_{0}} \otimes \varepsilon)}{L(1,\varepsilon^2)}.$$ 
where $S(\omega_\Pi \chi^2)$ now denotes the set of unramified square roots of the character $\omega_\Pi \chi^2$. Comparing with the spectral side \eqref{spectral_side_Xcheck_W_check}, we obtain the following. 

\begin{cor}
Let $\mu$ be a factorizable unramified Shalika normalized tempered cusp form on $G$. Then we have the period identity
    $$P_X(\mu) = \Delta^{-5/4} \cdot L_{\check{X}}(\varphi_\mu)/L_{\check{W}}(\varphi_\mu)$$
with discrepancy $-5$.
\end{cor} 
\begin{proof}
    Combining Proposition \ref{proposition spectral Xcheck period} and \eqref{spectral_side_Xcheck_W_check}, we have 
    \begin{align*}
        \frac{L_{\check{X}}(\varphi_\mu)}{L_{\check{W}}(\varphi_\mu)} &= \frac{\Delta^{-1/4} }{|S(\omega_{\Pi}\chi^2)|} \sum_{ \varepsilon \in S(\omega_{\Pi}\chi^2)}  \varepsilon(\partial^{-1/2})^3 \frac{L(1/2,\sigma,\wedge_0^{\varsigma_{0}} \otimes \varepsilon)}{L(1,\varepsilon^{2})}.
    \end{align*}
Hence, we get $$ P_X(\mu) = \Delta^{-5/4} \cdot L_{\check{X}}(\varphi_\mu)/L_{\check{W}}(\varphi_\mu).$$
\end{proof}

\subsection{\texorpdfstring{$(\check{G}, \check{X})$}{(Gcheck,Xcheck)} on the automorphic side}
\label{subsec_automorphic_side_Gcheck}

Let $\Pi$ be a cuspidal unramified tempered automorphic representation of $\check{G}(\A)$. By \eqref{ses_checkG}, we can identify $\Pi$ with $\pi \otimes \chi$, where  $\pi$ is a cuspidal unramified automorphic representation of $\GL_4(\A)$ and $\chi$ is an unramified character $[\mathbf{G}_m]\to \C^\times$ such that $\omega_\pi \chi^{-2} = 1$. Given any cusp form $f \in \pi$, Definition \ref{definition DM automorphic period} and Lemma \ref{example DM stack periods split isogeny} allow us to write 
 \begin{equation}\label{first_step_ind_unf_n}
    P_{\check{X}}(f \otimes \chi) = C_{Y'}\int_{[H']} \, (f \otimes \chi)(\iota(h)) \, |\eta_{Y'}(h)|^{1/2}\sum_{y \in \mathring{Y}'(F)} \Phi_{Y'}(y \partial^{1/2}h) \, dh,
\end{equation}
where $\Phi_{Y'} \in \mathcal{S}(Y'(\mathbb{A}))$ is the Schwartz function of $\mathfrak{o}$-points, and we have denoted by $C_{Y'}$ the overall normalizing constant
$$C_{Y'} = \Delta^{\frac{\mathrm{dim}(Y') - \mathrm{dim}(H')}{4}}|\eta_{Y'}(\partial^{1/2})|^{1/2} = \Delta^{-5/2},$$
where for the latter equality we have used Lemma \ref{Lemma Gcheck eigenmeasure}. 
\begin{lem}
    We have
    $$P_{\check{X}}(f \otimes \chi) = C_{Y'}'\chi(\partial)^{-1/2}  \bigg[\int_{\mathbb{A}^\times \cap \mathfrak{o}} |x|^{3/2} \,\chi(x) \, dx \bigg]\bigg[\int_{[P'_0]} (f \otimes \chi)(\iota(p)a_0)|\eta_{Y'}(p)|^{1/2} dp\bigg],$$
    where $C_{Y'}' =\Delta^{-19/4}$, $a_0$ is the element introduced in \eqref{Definition:a0}, and $dp$ denotes the left invariant Haar measure on $P'_0(\mathbb{A})$.
\end{lem}
\begin{proof}
    By Lemma \ref{Description:Variety:Lemma} and the fact that $\mathrm{K}^1_{P_0', H'} =0$, one has
    $$\mathring{Y}'(F) \simeq P_0'(F) \backslash H'(F),$$
    hence 
    $$ P_{\check{X}}(f \otimes \chi) = C_{Y'}\int_{[H']} \, (f \otimes \chi)(\iota(h)) \, |\eta_{Y'}(h)|^{1/2}\sum_{h' \in P_0'(F) \backslash H'(F)} \Phi_{Y'}(v_0 h'  \partial^{1/2}h) \, dh,$$
  where $v_0 = (\langle f_1, f_2\rangle, f_1 \wedge f_2) \in \mathring{Y}'(F)$. Observe that $\partial^{1/2} \in \Ggr(\mathbb{A})$ acts on $v_0$ as the element 
  $$h_{\partial^{1/2}} = \left(\left[\begin{smallmatrix} \partial & & & \\ & \partial& & \\ & & 1 & \\ & & & 1\end{smallmatrix}\right], \partial^{1/2},\partial^{3/2}\right) \in P'(\mathbb{A}),$$  which is designed so that $\iota(h_{\partial^{1/2}}) = a_0^{-1} \cdot ({\rm Id},\partial^{1/2})$, with $a_0$ the element introduced in \eqref{Definition:a0}. Hence, we can write  
    \begin{align*}
        P_{\check{X}}(f \otimes \chi) &=  C_{Y'}\int_{[H']} \, (f \otimes \chi)(\iota(h)) \, |\eta_{Y'}(h)|^{1/2}\sum_{h' \in P_0'(F) \backslash H'(F)} \Phi_{Y'}(v_0 h_{\partial^{1/2}} h' h) \, dh \\ 
        &= C_{Y'} \int_{P'_0(\mathbb{A}) \backslash H'(\mathbb{A})} \Phi_{Y'}(v_0 h_{\partial^{1/2}}  h)|\eta_{Y'}(h)|^{-1/2} \int_{[P'_0]} (f \otimes \chi)(\iota(ph))|\eta_{Y'}(p)|^{1/2} \, dp \, d_\omega h, 
    \end{align*}
where $dp$ denotes the left invariant Haar measure on $P_0'(\mathbb{A})$ and $d_\omega h$ denotes the measure associated with the eigenvolume form $\omega$ of  Lemma \ref{Lemma Gcheck eigenmeasure}. Since $h_{\partial^{1/2}}$ normalizes $P_0'$ and  $$|\eta_{Y'}(h_{\partial^{1/2}})|^{-1/2}=|\partial|^{-9/4} = \Delta^{9/4},$$  we can make the change of variables $h_{\partial^{1/2}}h \to h$  to get
    \begin{align*}
        &C_{Y'}' \int_{P_0'(\mathbb{A})\backslash H'(\mathbb{A})}\Phi_{Y'}(v_0h)|\eta_{Y'}(h)|^{-1/2} \int_{[P_0']}(f \otimes \chi)(\iota(p h^{-1}_{\partial^{1/2}} h))|\eta_{Y'}(p)|^{1/2} \, dp \, d_\omega h\\
        &= C_{Y'}'\chi(\partial)^{-1/2}\int_{P_0'(\mathbb{A})\backslash H'(\mathbb{A})}\Phi_{Y'}(v_0h)|\eta_{Y'}(h)|^{-1/2} \int_{[P_0']}(f \otimes \chi)(\iota(p) a_0 \iota(h))|\eta_{Y'}(p)|^{1/2} \, dp \, d_\omega h,  \end{align*}
    where $C_{Y'}' = \Delta^{-9/4} \cdot C_{Y'} = \Delta^{-19/4}$ as in the statement of the lemma.

Now, recall that $P_0'$ sits inside the Siegel parabolic $P'$ of $H'$; therefore, using the Iwasawa decomposition $H'(\mathbb{A}) = P'(\mathbb{A})H'(\mathfrak{o})$, the fact that $f$ is unramified, and that $H'(\mathfrak{o})$ preserves $\mathbf{Y}'(\mathfrak{o})$ we reduce the domain of integration on the outer integral to obtain
   $$C_{Y'}'\chi(\partial)^{-1/2}\int_{P_0'(\mathbb{A})\backslash P'(\mathbb{A})} \Phi_{Y'}(v_0h)|\eta_{Y'}(h)|^{-1/2}\int_{[P_0']}(f \otimes \chi)((\iota(p) a_0 \iota(h)))|\eta_{Y'}(p)|^{1/2} \, dp \,   |\eta_{Y'}(h)|dh,$$
where $dh$ denotes the left invariant Haar measure on $P'(\mathbb{A})$.  In order to parametrize the outer domain of integration, consider the morphism 
    \begin{align*} \label{equation central torus}
        P'&\longrightarrow \mathbf{G}_m,\\
        (p,x,y)&\mapsto \frac{\nu(p)y}{\mu(p)}\nonumber
    \end{align*}
    It is surjective and has kernel $P'_0$, with section 
    \begin{align*}
        \mathbf{G}_m&\to P'_0\setminus P',\, x\mapsto (1,x,x).
    \end{align*}
Thus, we recognize then that the outer integral is a 1-dimensional multiplicative adelic integral which can be completely separated
$$C_{Y'}'\chi(\partial)^{-1/2}\int_{\mathbb{A}^\times} \Phi_{Y'}(v_0 x)|x|^{3/2}\chi(x) \, dx\int_{[P_0']}(f \otimes \chi)(\iota(p) a_0) |\eta_{Y'}(p)|^{1/2} \, dp.$$
The result follows from the fact that $\Phi_{Y'}(v_0 x) = 0$ unless $x \in \A^\times \cap \mathfrak{o}$.
\end{proof}

The first integral resulting from the preceding lemma, over the 1-dimensional multiplicative region $\mathbb{A}^\times \cap \mathfrak{o}$, is an integral of Tate type which can be evaluated as an abelian $L$-function
$$\int_{\mathbb{A}^\times \cap \mathfrak{o}}|x|^{3/2} \, \chi(x) \, dx = L(3/2, \chi).$$
To continue, we focus on unfolding the inner $[P_0']$-integral further, and we set
$$I_{P_0'}(f \otimes \chi) = \int_{[P'_0]} (f \otimes \chi)(\iota(p)a_0)|\eta_{Y'}(p)|^{1/2} dp.$$
We will unfold $I_{P_0'}$ using the Shalika period. 
\begin{lem}
    We have
    $$I_{P_0'}(f \otimes \chi) = \Delta^{3/2}\int_{\mathbb{A}^\times} \Omega_{f \times \chi}(\iota(a_y)a_0)\delta_{P_0'}^{-1/2}(a_y) \, dy,$$
    where $$a_y := \left(\begin{bmatrix} \mathrm{Id}_{2} & 0\\ 0 & y\cdot\mathrm{Id}_{2}\end{bmatrix}, y^{-2},y^{-1}\right)$$ and $\Omega_{f \times \chi}$ is the global Shalika function
    $$\Omega_{f \times \chi}(m_1,m_2) := \int_{[\mathrm{PGL}_2]}f_{U,\psi_1}\left ( \left[\begin{smallmatrix} g & 0\\ 0 & \mathrm{det}(g) J_2 {}^{t}g^{-1}J_2\end{smallmatrix}\right] m_1 a_\partial^{-1} \right)  \chi( \mathrm{det}(g)^{-1} m_2 )\, dg.$$
\end{lem}
\begin{proof}
    Consider the decomposition of $P'_0$ into its unipotent radical $U_0'$ and its maximal reductive subgroup $M_0'$. We can write
    $$I_{P_0'}(f \otimes \chi) = \int_{[U_0'][M'_0]}(f \otimes \chi)(\iota(um)a_0)|\eta_{Y'}(m)|^{1/2}\delta_{P_0'}^{-1}(m) \, d(u,m),$$
    where the factor $\delta_{P_0'}^{-1}(m)$ appeared because $d p$ is the left invariant Haar measure (see \cite[IV (9)]{Cartier}).
    For simplicity, denote temporarily $(f \otimes \chi)$ by $\mathcal{F}$. We may Fourier expand along the codimension 1 abelian subgroup $[U'_0] \subset [\mathrm{Mat}_2]$ and appeal to cuspidality to write
    $$\int_{[U'_0]}\mathcal{F}(\iota(um) a_0) \, du = \Delta^{u'_0/2}\int_{[U_0']}\mathcal{F}(\iota(um)a_0) \, d^\psi u = \, \Delta^{u'_0/2}\sum_{\alpha \in F^\times} \, \mathcal{F}_{U,\psi_\alpha}(\iota(m) a_0)$$
   where  $u_0' = \mathrm{dim}(U_0')=3$, we have switched to the self-dual measure $d^\psi u$ in order to Fourier expand, and we used the notation
    $$\mathcal{F}_{U,\psi_\alpha}(g) := \int_{[U]} \, \mathcal{F}(ug) \, \psi_\alpha(u) \, d^\psi u.$$
    We also recall that $\psi_\alpha:[U/U_0']\to \C^\times$ denotes the character
    \[n(X):= \left[\begin{smallmatrix}I_2&X\\ &I_2\end{smallmatrix}\right]\mapsto \psi_{\alpha}(n(X))= \psi(\mathrm{Tr}(A_{\alpha}X)), \text{ with } A_{\alpha} = \left[\begin{smallmatrix}\alpha& \\ &-\alpha\end{smallmatrix}\right].\]
    Now, concretely, $M'_0$ is given by 
    $$M'_0 = \bigg\{\left(\left[\begin{smallmatrix} g & 0\\ 0 & \lambda J_2 {}^{t}g^{-1}J_2\end{smallmatrix}\right],x, \mathrm{det}(g)\lambda^{-1}\right) \in H'\,:\, x^2=\mathrm{det}(g)^2 \lambda^{-4}\bigg\}$$
    whose $F$-points act transitively on the set of $\psi_\alpha$'s, with the stabilizer of the $\alpha = 1$ term being the $F$-points of the subgroup $R \subset M'_0$ given by $$\bigg\{\left(\left[\begin{smallmatrix} g & 0\\ 0 & \mathrm{det}(g) J_2 {}^{t}g^{-1}J_2\end{smallmatrix}\right],\pm \mathrm{det}(g)^{-1}, 1\right) \in H'\bigg\}.$$
    Returning back to the computation of $I_{P_0'}$, since $f$ and $\chi$ are unramified we may then unfold to obtain
    \begin{align*}
        I_{P_0'}(f \otimes \chi) &= \Delta^{3/2}\int_{R(\mathbb{A})\backslash M'_0(\mathbb{A})}\int_{[R]}(f_{U,\psi_1} \otimes \chi)(\iota(rm)a_0)|\eta_{Y'}(m)|^{1/2}\delta_{P_0'}^{-1}(m) \, dr \, dm\\
        &=  \Delta^{3/2}\int_{R(\mathbb{A})\backslash M'_0(\mathbb{A})}\int_{[\mathrm{PGL}_2]}\!\!\! \!\! f_{U,\psi_1}\left ( \left[\begin{smallmatrix} g & 0\\ 0 & \mathrm{det}(g) J_2 {}^{t}g^{-1}J_2\end{smallmatrix}\right] m_1 a_\partial^{-1} \right)  \chi( \mathrm{det}(g)^{-1} m_2 )\delta_{P_0'}^{-1/2}(m) \, dg \, dm\\ 
        &= \Delta^{3/2}\int_{R(\mathbb{A})\backslash M'_0(\mathbb{A})} \Omega_{f \times \chi}(\iota(m)a_0) \delta_{P_0'}^{-1/2}(m) \, dm,
    \end{align*}
    where we denoted by $(m_1,m_2)$ the coordinates of $\iota(m)$ and we used the fact that  the characters $|\eta_{Y'}|$ and $\delta_{P_0'}$ are equal when restricted to $P_0'$.
    In order to simplify the outer domain of integration, note that there is a map
$$R \backslash M'_0 \longrightarrow \Gm$$
$$\left(\begin{bmatrix} g & 0\\ 0 & \lambda g^*\end{bmatrix},x,\mathrm{det}(g)\lambda^{-1}\right)\longmapsto \mathrm{det}(g)\lambda^{-1},$$
where $g^* = J_2 {}^t g^{-1} J_2$, exhibiting $R$ as a normal subgroup. Using the section 
$$a: y \longmapsto a_y = \left(\begin{bmatrix} \mathrm{Id}_{2} & 0\\ 0 & y\cdot\mathrm{Id}_{2}\end{bmatrix}, y^{-2},  y^{-1}\right),$$
we finally obtain the desired expression in the lemma
$$I_{P_0'}(f \otimes \chi) = \Delta^{3/2}\int_{\mathbb{A}^\times} \Omega_{f \times \chi}(\iota(a_y)a_0)\delta_{P_0'}^{-1/2}(a_y) \, dy.$$

\end{proof}

Combining the preceding calculations, we obtain an Eulerian product expression for the automorphic period $P_{\check{X}}(f \otimes \chi)$.
\begin{thm} \label{theorem Gcheck global integral}
    Let $f \otimes \chi$ be a factorizable unramified Shalika normalized tempered cusp form on $\check{G}$. Then
    $$P_{\check{X}}(f \otimes \chi) = C_{Y'}'' \cdot \chi(\partial^{-1/2} )\cdot \Omega_f^\circ \cdot L( 3/2, \chi) \cdot \prod_v \, Z_v(f_v \times \chi_v)$$
    where 
    \begin{itemize}
        \item $C_{Y'}'' = \Delta^{-13/4}$,
        \item $\Omega_f^\circ = \Delta^{7/4}\chi(\partial)^{-1}$ is the global constant appearing in \eqref{equation Omegaf0},
        \item $L( 3/2, \chi)$ is the Tate $L$-function of the character $\chi$,
        \item and $Z_v(f_v \times \chi_v)$ is the local expression
    $$Z_v(f_v \times \chi_v) = \sum_{k \geq 0} \,  \Omega_{f,v}^{\mathrm{ur}} \left (\left[\begin{smallmatrix} \varpi_v^{k} \cdot\mathrm{Id}_2 & 0\\ 0 & \mathrm{Id}_2 \end{smallmatrix} \right] \right)\delta_{P_0'}^{1/2}(a_{\varpi_v^k}).$$
    \end{itemize}
\end{thm}
\begin{proof}

First, we calculate the global constant prefactor. Multiplying 
$$C_{Y'}' \cdot \Delta^{3/2} = \Delta^{-13/4},$$ we obtain the description of $C_{Y'}''$ in the statement. Note that the adelic integral $I_{P_0'}(f \otimes \chi)$ splits as an Eulerian product of local integrals as $\Omega_{f \times \chi}$ factors into a product of local unramified Shalika functions. Indeed, proceeding as in Proposition \ref{Prop_Shalika_twisted}, we get 
\begin{align*}
    \Omega_{f \times \chi}(\iota(a_y)a_0) &= \int_{[\mathrm{PGL}_2]}(f_{U,\psi_1} \otimes \chi)\left(\left[\begin{smallmatrix} g & 0\\ 0 & \mathrm{det}(g) g^* \end{smallmatrix} \right]\left[\begin{smallmatrix} \mathrm{Id}_2 & 0\\ 0 & y \cdot \mathrm{Id}_2 \end{smallmatrix} \right] a_\partial^{-1},\mathrm{det}(g)^{-1}y^{-2}\right ) \, dg \\ 
    &=  \int_{[U][\mathrm{PGL}_2]}f\left(u\left[\begin{smallmatrix} y^{-1}g & 0\\ 0 & g \end{smallmatrix} \right] a_\partial^{-1} \right)   \chi^{-1}\left(\mathrm{det}(g)\right ) \Psi(u) \, d^\psi u dg 
\end{align*}
where in the second equality we used $\omega_\pi = \chi^2$ and that conjugating by $w_0 = {\rm diag}(1,1,-1,1) \in \GL_4(F)$ identifies $$w_0\left[\begin{smallmatrix} g& \\ &\mathrm{det}(g) g^* \end{smallmatrix}\right]w_0^{-1} =  \left[\begin{smallmatrix} g& \\ &g \end{smallmatrix}\right],\,\,\psi_1(w_0 u w_0^{-1}) = \Psi(u).$$
 Thus, by \eqref{equation expliciting Shalika normalization}, we can write 

$$ \Omega_{f \times \chi}(\iota(a_y)a_0) = \Omega_f^\circ \cdot  \prod_v \Omega_{f,v}^{\rm ur} \left (\left[\begin{smallmatrix} y^{-1} \cdot\mathrm{Id}_2 & 0\\ 0 & \mathrm{Id}_2 \end{smallmatrix} \right] \right),$$
with $\Omega_f^\circ = \Delta^{\frac{7}{4}} \chi^{-1}(\partial)$ the constant of \eqref{equation Omegaf0}. This let us write $\Delta^{-3/2} I_{P_0'}(f \otimes \chi)$ as 
    \begin{align*}
        &\Omega_f^\circ \cdot \prod_v \int_{F_v^\times}  \Omega_{f,v}^{\rm ur} \left (\left[\begin{smallmatrix} y^{-1} \cdot\mathrm{Id}_2 & 0\\ 0 & \mathrm{Id}_2 \end{smallmatrix} \right] \right)\delta_{P_0'}^{-1/2}(a_y) \, dy \\
        &=\Omega_f^\circ \cdot  \sum_{k \in \Z} \, \Omega_{f,v}^{\mathrm{ur}} \left (\left[\begin{smallmatrix} \varpi_v^{-k} \cdot\mathrm{Id}_2 & 0\\ 0 & \mathrm{Id}_2 \end{smallmatrix} \right] \right)\delta_{P_0'}^{-1/2}(a_{\varpi_v^k}) \\ 
         &=\Omega_f^\circ \cdot  \sum_{k \geq 0} \, \Omega_{f,v}^{\mathrm{ur}} \left (\left[\begin{smallmatrix} \varpi_v^{k} \cdot\mathrm{Id}_2 & 0\\ 0 & \mathrm{Id}_2 \end{smallmatrix} \right] \right)\delta_{P_0'}^{1/2}(a_{\varpi_v^k}), 
    \end{align*}
 where in the last step we observed that $\Omega_{f,v}^{\mathrm{ur}} \left (\left[\begin{smallmatrix} \varpi_v^{-k} \cdot\mathrm{Id}_2 & 0\\ 0 & \mathrm{Id}_2 \end{smallmatrix} \right] \right) = 0$ for $k \geq 0$, and made a substitution $k \mapsto -k$ to simplify notation slightly.
\end{proof}

\subsubsection{The local computation} It follows from Theorem \ref{theorem Gcheck global integral} that we can assume the unramified representation $\pi_v$ of $\GL_4(F_v)$ to have Satake parameter $\mathrm{Sat}_{\pi_v}$ in $\GSp_4(\kk)$ with similitude character $\chi_v$. Denote $$s_{(a,b;c)} = \mathrm{Tr}(V_{(a,b;c)}|\mathrm{Sat}_{\pi_v}).$$
The Casselman--Shalika formula of Proposition \ref{CS:Formula:Shalika:simplified} applied to $\Omega_{f,v}^{\mathrm{ur}} $ allows us to write
\begin{equation}
    Z_v(f_v \times \chi_v) = \sum_{k \geq 0} q_v^{-2k} \delta_{P_0'}^{k/2}(a_{\varpi_v}) \bigg(s_{(k,k;0)} - s_{(k-1,k-1;0)}\chi_v(\varpi_v)q_v^{-1}\bigg).
\end{equation}
Using the fact that  $s_{(-1,-1;0)} = 0$, we may compute the preceding expression as 
\begin{align*}
     \sum_{k\geq 0}q_v^{-2k}\delta_{P_0'}^{k/2}(a_{\varpi_v})  s_{(k,k; 0)} - \chi_v(\varpi_v)q_v^{-3}\delta_{P_0'}^{1/2}(a_{\varpi_v}) \bigg(\!\sum_{k\geq 0}q_v^{-2(k-1)}\delta_{P_0'}^{(k-1)/2}(a_{\varpi_v}) s_{(k-1,k-1; 0)}\!\bigg).
\end{align*}
Thus, we conclude that 
$$Z_v(f_v \times \chi_v)
 = (1 - \chi_v(\varpi_v)\delta_{P_0'}^{1/2}(a_{\varpi_v})q_v^{-3})\sum_{k\geq 0}q_v^{-2k}\delta_{P_0'}^{k/2}(a_{\varpi_v})   s_{(k,k; 0)}.$$
Towards the evaluation of our local zeta integral $Z_v(f_v \times \chi_v)$, we first establish the following computation, analogous to Lemma \ref{lemma_on_Lstd_G_side}.

\begin{lem}\label{lemma_on_std_n} For ${\rm Re}(s) \gg 0$, we have 
     $$\frac{L(s,\pi_v,\wedge_0^2 \otimes \chi_v)}{L(2s,\chi_v^{2})} = \sum_{k \geq 0}q_v^{-ks}s_{(k,k;0)} ,$$  
    where recall that $\wedge_0^2$ is the $5$-dimensional representation of $\GSp_4(\kk)$ of highest weight $(1,1;-1)$ and $s_{(a,b;c)} = \mathrm{Tr}(V_{(a,b;c)}|\mathrm{Sat}_{\pi_v})$.
\end{lem}
\begin{proof}
Thanks to \cite[Lemma 6.2]{CauchiGutiCS}, we can write $$L(s,\pi_v,\wedge_0^2 \otimes \chi_v) = \sum_{k=0}^\infty \sum_{i=0}^{\lfloor k/2\rfloor}q_v^{-ks} s_{(k-2i,k-2i;2i)} .$$
    Moreover, using the  Cauchy product formula, the right hand side equals $$\left( \sum_{k=0}^\infty q_v^{-2ks}\chi_v({\varpi_v})^{2k} \right) \left( \sum_{k = 0}^\infty q_v^{-ks} s_{(k,k;0)} \right).$$
The result then follows.
\end{proof}

Applying the preceding lemma to the last expression of $Z_v(f_v \times \chi_v)$, we obtain 
\begin{equation}\label{we are surfing baby}
    Z_v(f_v \times \chi_v) = \frac{L(2+\langle a, \rho_{P_0'}\rangle, \pi_v, \wedge_0^2 \otimes \chi_v)}{L(4+2\langle a, \rho_{P_0'}\rangle, \chi^2_v) L(3+ \langle a, \rho_{P_0'}\rangle, \chi_v)}.
\end{equation}
where $a: \Gm \to M_0'$ is the cocharacter defining $y \mapsto a_y$ and $\rho_{P_0'}$ denotes the half sum of positive roots that belong to $P_0'$, so that $\delta_{P_0'}^{1/2}(a_{\varpi_v}) = q_v^{-\langle a, \rho_{P_0'}\rangle}$. 
\begin{thm}\label{thm PXcheck final}
     Let $f \otimes \chi$ be a factorizable unramified Shalika normalized tempered cusp form on $\check{G}$. Then
     $$P_{\check{X}}(f \otimes \chi) =  \chi(\partial^{-3/2} ) \cdot \Delta^{-3/2} \cdot  \frac{L(1/2,\varphi_f, \wedge_0^2 \otimes \chi)}{L(1, \chi^2)}.$$
\end{thm}
\begin{proof}
Explicitly, we have $$ C_{Y'}'' \cdot \chi(\partial^{-1/2} )\cdot \Omega_f^\circ = \chi(\partial^{-3/2} )\Delta^{-3/2} .$$  Thus combining this with  \eqref{we are surfing baby}, the Tate $L$-function appearing in Theorem \ref{theorem Gcheck global integral}, and the computation $\langle a, \rho_{P_0'}\rangle = -3/2$, we obtain the final expression for $P_{\check{X}}(f \otimes \chi)$.
\end{proof}
 Comparing with the spectral $X$-period, we see that for a Shalika normalized cusp form $f \times \chi$ on $\check{G}$, we have the following period identity.
 \begin{cor}\label{final cor Gcheck side}
    Let $f \otimes \chi$ be a factorizable Shalika normalized unramified tempered cusp form on $\check{G}$. Then we have the period identity $$P_{\check{X}}(f \otimes \chi) = \Delta^{-5/4} \cdot  L_X(\varphi_{f\otimes \chi})/L_{\check{M}}(\varphi_{f\otimes \chi}),$$
    with discrepancy $-5$.
 \end{cor}
 \begin{proof}
 By Proposition \ref{proposition spectral X period} and \eqref{Shalika normalized spectral period X}, we have
 \begin{align*}
      \frac{L_X(\varphi_{f\otimes \chi})}{L_{\check{M}}(\varphi_{f\otimes \chi})} &= (\nu \circ \varphi_f)(\partial)^{-3/2}  \Delta^{-1/4}   \frac{L(1/2, \varphi_f, \wedge_0^2 \otimes \nu)}{L(1, \varphi_f, \nu^{2})} \\ &= \chi(\partial)^{-3/2}\Delta^{-1/4} \frac{L(1/2, \varphi_f, \wedge_0^2 \otimes \chi)}{L(1, \chi^{2})}.
 \end{align*}
Therefore, Theorem \ref{thm PXcheck final} gives
$$P_{\check{X}}(f \otimes \chi) = \Delta^{-5/4} \cdot  \frac{L_X(\varphi_{f\otimes \chi})}{L_{\check{M}}(\varphi_{f\otimes \chi})},$$
as desired.
 \end{proof}
\begin{rem}
  If $f \otimes \chi$ were not Shalika normalized (but still with non-trivial $P_M$-period), the formula of Corollary \ref{final cor Gcheck side} would read as $$P_{\check{X}}(f \otimes \chi) / P_M(f \otimes \chi) = \Delta^{-5/4} \cdot  L_X(\varphi_{f\otimes \chi})/L_{\check{M}}(\varphi_{f\otimes \chi}).$$
\end{rem}

\subsection{Remarks on discrepancies} \label{subsection discrepancy}

Since the first examples of singular relative Langlands duality were discovered, we have understood that discrepancies are typical and unavoidable in the majority of period formulae. Although we currently do not have a satisfying conceptual understanding, by carefully tracing through our (Shalika) normalization scheme and global unfolding, one can start to make some sense of the resulting discrepancies. Our discussion here is definitely not the final word on this topic, but rather a starting point for future investigations. 

Consider the situation of Conjecture \ref{conjecture ratio of periods}, i.e., we attempt to establish a period formula of the form 
\begin{equation}\label{equation discrepancy discussion}
    \frac{P_{M_1}}{P_{M_2}} \overset{?}{=} \frac{L_{\check{X}_1}}{L_{\check{X}_2}},
\end{equation}
where $(\overset{?}{=})$ means we are looking for an equality allowing for discrepancy corrections, under the following assumptions: 
\begin{itemize}
    \item On the automorphic side, we assume that $P_{M_1}$ unfolds to an adelic integral of $P_{M_2}$, to be made precise below. We remind the reader that the periods $P_{M_1}, P_{M_2}$ are defined \textit{with a $\partial^{1/2}$-shift} as in Definition \ref{definition automorphic period} and \eqref{equation theta series}.
    \item On the spectral side, we assume that $\check{X}_2 = H \backslash \check{G}$ for some reductive subgroup $H \subset \check{G}$, and $\check{X}_1$ is the induction from $H$ to $\check{G}$ of some (possibly singular) $H$-variety $Y$. 
    \item Finally, we assume that $P_{M_2} = L_{\check{X}_2}$ is a smooth numerical duality (arising from a hyperspherical dual pair, say) without discrepancies. 
\end{itemize}
One may wish to take the case of $M_2 = \mathrm{Whittaker}$ and $\check{X}_2 = \mathrm{pt}$, which is already interesting.

With the above setting in mind, let us consider first the global constant appearing in the definition of the spectral period: for a graded $\check{G}$-variety $\check{X}$, we may define the
$$\text{spectral discrepancy } \sigma(\check{X}) := \varepsilon_{\check{X}} - \mathrm{dim}(\check{X})$$
as the power of $\Delta^{1/4}$ which appears in Definition \ref{definition spectral period}. In some sense, the spectral discrepancy measures the nonlinearity of $\check{X}$: it vanishes when $\check{X}$ is a vector space with $\Ggr$ acting by scaling. Furthermore, spectral discrepancy behaves nicely with respect to induction: for $H \subset \check{G}$ a reductive subgroup and $Y$ a graded $H$-variety with spectral discrepancy $\sigma(Y)$, we have
$$\sigma(\mathrm{Ind}_H^{\check{G}}(Y)) = \sigma(Y) - \mathrm{dim}(H \backslash \check{G}).$$
In particular, if we consider the ratio of spectral periods on the right hand side of \eqref{equation discrepancy discussion} for $\check{X}_1 = \mathrm{Ind}_H^{\check{G}}(Y)$ and $\check{X}_2 = H \backslash \check{G}$, then the spectral discrepancy subtracts to give a \textit{relative spectral discrepancy}
\begin{equation}
    \sigma_{\check{X}_1/\check{X}_2} := \sigma(\check{X}_1) - \sigma(\check{X}_2) = \sigma(Y)
\end{equation}
as the power of $\Delta^{1/4}$ appearing on the right hand side of \eqref{equation discrepancy discussion}, and we regard it as the spectral discrepancy in the case when we normalize $L_{\check{X}_1}$ by $L_{\check{X}_2}$.

The analogue of the previous paragraph on the automorphic side is more subtle. To start, we may define, for $M = (X,\Psi)$ a graded Hamiltonian $G$-variety, its
$$\text{automorphic discrepancy } \alpha(M) := \mathrm{dim}(X) - \mathrm{dim}(G) - \varepsilon_X,$$
as the power of $\Delta^{1/4}$ in Definition \ref{definition automorphic period}. Now let us make precise the assumption that $P_{M_1}$ unfolds to an adelic integral of $P_{M_2}$: for a cusp form $f$ on $G$ with trivial central character, we assume that there is a combination of unfolding and Fourier--Whittaker expansions that allow us to write 
\begin{equation} \label{equation PM1 unfolds to PM2}
    P_{M_1}(f) = \Delta^{\frac{\alpha(M_1)+ d + 2u }{4}} \cdot \int_{\Xi(\mathbb{A})} \, |\eta_{M_1}(a)|^s \, P_{M_2}^\psi(r_a \cdot f) \, da,
\end{equation}
where $\Xi(\mathbb{A})$ is some adelic reductive subgroup of $G(\mathbb{A})$, $s \in \Q$, $r_a$ denotes right translation by $a \in \Xi(\mathbb{A})$, and $P_{M_2}^\psi$ is defined by the same formula as $P_{M_2}$ but using the self-dual measure on the unipotent regions of integration, and dropping all the global constants; in other words, $$P_{M_2}^\psi = \Delta^{\frac{-\alpha(M_2) + 2 \, \mathrm{dim}(U)}{4}}P_{M_2}$$ as automorphic distributions, where we write $M_2 = (U L \backslash G, \Psi_U) $, with $U$ (resp. $L$) a possibly trivial unipotent (resp. reductive) subgroup of $G$ and $\Psi_U$ an affine line bundle encoded by a generic character on $[U]$. The numbers $d$ and $u$ appearing in \eqref{equation PM1 unfolds to PM2}, which are the most subtle constants associated with the unfolding process, are described as follows: \begin{itemize}
    \item $u $ is the dimension of unipotent regions of integrations in $P_{M_1}$ on which a Fourier--Whittaker expansion is performed. Switching from the probability measure to the self-dual measure introduces the extra factor of $\Delta^{2u}$ above.
    \item $d$ takes into account the transition from the $\partial^{1/2}$-shift in $\theta_{M_1}$ to an idèle $a_{M_2,0} \in \Xi(\A)$ which effectuates the $\partial^{1/2}$-shift in $\theta_{M_2}$. More precisely, if we let $h_{\partial^{1/2}}$ be any element in $\Xi(\A)$ which acts on the preferred base point of $M_1$ as $\partial^{1/2}$ and such that $h_{\partial^{1/2}}^{-1}$ equals $a_{M_2,0}$ modulo the center of  $G(\A)$, then  $$ \Delta^{d/4} = |\eta_{M_1}( h_{\partial^{1/2}}^{-1}) |^s.$$
\end{itemize}
Since $M_2$ is hyperspherical, it is reasonable to write $\Omega_{f,v}^{\mathrm{ur}}$ as the unramified local spherical function \cite{SakellaridiSsphericalFunctions} associated to $M_2$ for each place $v$, normalized to have value 1 at the identity.
By multiplicity one of the local models in which $\Omega_{f,v}^{\mathrm{ur}}$ lives, we have
\begin{equation} \label{equation global spherical function and local spherical functions}
    P_{M_2}^\psi(r_a \cdot f) =  \Omega_f^\circ \cdot \prod_v \, \Omega_{f,v}^{\mathrm{ur}}(a_v)
\end{equation}
for some constant $\Omega_f^\circ$ depending on the normalization of $f$ (hence on $a_{M_2,0}$). Substituting \eqref{equation global spherical function and local spherical functions} into the integral \eqref{equation PM1 unfolds to PM2} and using the Casselman--Shalika type formulae of Sakellaridis in \textit{op. cit.}, one can in principle relate $P_{M_1}(f)$ directly to an Eulerian product of graded traces of the Satake parameters of $f$, which one hopes to compare to the spectral period $L_Y$. 

Focusing on the global constants that appear, we see that for $f$ to be $M_2$-normalized, the right hand side of \eqref{equation global spherical function and local spherical functions} needs to read $\Delta^{\frac{-\alpha(M_2)+2 \mathrm{dim}(U)}{4}}$ when we take $a = \mathrm{id}$: indeed, $M_2$-normalization means that $P_{M_2}(f) = 1$. Since we have chosen $\Omega_{f,v}^{\mathrm{ur}}(1) = 1$, this means that 
\begin{equation*}
    \Omega_f^\circ = \Delta^{\frac{-\alpha(M_2)-2\mathrm{dim}(U)}{4}}.
\end{equation*}
Substituting \eqref{equation global spherical function and local spherical functions} into \eqref{equation PM1 unfolds to PM2} with the assumption that $f$ is $M_2$-normalized, we obtain
\begin{equation*}
    P_{M_1}(f) = \Delta^{\frac{\alpha(M_1)-\alpha(M_2) + d +  2 (u - \mathrm{dim}(U))}{4}} \cdot  \prod_v \, \int_{\Xi(F_v)}|\eta_{M_1}(a_v)|^s \, \Omega_{f,v}^{\mathrm{ur}}(a_v) \, d a_v.
\end{equation*}
In other words, the correct notion of \textit{relative automorphic discrepancy} of the ratio of automorphic periods $M_1$ by $M_2$ is not simply $\alpha(M_1) - \alpha(M_2)$, but rather it is corrected by $d+ 2(u-\mathrm{dim}(U))$
\begin{equation}
    \alpha_{M_1/M_2} := \alpha(M_1)-\alpha(M_2) + d + 2 (u - \mathrm{dim}(U)).
\end{equation} 

Finally, it is clear that for equation \eqref{equation discrepancy discussion} to be free of discrepancy, it is equivalent to ask for the following condition to hold. 
\begin{defn} \label{definition discrepancy equation} Let $M_1, M_2$ be a pair of graded $G$-actions, and $\check{X}_1, \check{X}_2$ be a pair of graded $\check{G}$-actions, as in \eqref{equation discrepancy discussion}. Then we say that the \textit{discrepancy equation is satisfied} if
    \begin{equation}\label{equation discrepancy equation}
   \alpha_{M_1/M_2} = \sigma_{\check{X}_1/\check{X}_2}.
\end{equation}
\end{defn}
That is, we demand that the relative automorphic and spectral discrepancies match. 

\begin{example} \label{example Whittaker normalization 2}
    Analyzing \eqref{equation discrepancy equation} for Example \ref{example Whittaker normalization}, the (smooth) Hecke period, we see that the discrepancy equation is satisfied if
    $$\alpha(T \backslash \mathrm{PGL}_2) -\alpha(\mathrm{Whitt})+ d + 2(u - \mathrm{dim}(U)) = 0.$$
    This is indeed satisfied, with the following calculations:
    $$\alpha(T\backslash \mathrm{PGL}_2) = -\mathrm{dim}(T) = -1$$
    $$\alpha(\mathrm{Whitt}) = -\mathrm{dim}(U) - \langle 2\rho, 2\check{\rho}\rangle = -3$$
    where $2\rho$ (resp. $2\check{\rho}$) is the positive root (resp. coroot), whose associated unipotent subgroup is $U$. In this case, $u = 0$ since there is no unipotent integral in the toric period whose probability measure we need to correct to the self-dual measure. Finally, $d = 0$ as $s=0$.
\end{example}

\begin{example}
    We analyze another smooth case, that of the Godement--Jacquet period for $G = \mathrm{GL}_n^{\times 2}$ acting on $M_1 = T^*\mathrm{Mat}_n$, normalized by the $L^2$-period $M_2 = T^*\mathrm{GL}_n$ (see \cite[\S B.3.1]{CV}). The automorphic discrepancies of $M_1$ and $M_2$ are
    $$\alpha(M_1) = n^2 - 2n^2 - n^2 = -2n^2, \, \alpha(M_2) = n^2 - 2n^2 = -n^2.$$
    In this case, $h_{\partial^{1/2}}^{-1}= a_{M_2,0} = (\mathrm{Id}_n,\partial^{-1/2} \mathrm{Id}_n)$, $s=1/2$, and $$ |\eta_{M_1}(a_{M_2,0}) |^{1/2} = \Delta^{\frac{n^2}{4}} \Rightarrow d = n^2.$$ 
    In this case there are no unipotent regions of integration, so the relative automorphic discrepancy is simply the difference
    $$\alpha_{M_1/M_2} = -2n^2 - (-n^2) + n^2=0.$$
    On the spectral side, we have $\check{X}_1 = \mathrm{Ind}_{\mathrm{GL}_n^\mathrm{diag}}^{\check{G}}\mathbf{A}^n$ and $\check{X}_2 = \mathrm{GL}_n^\mathrm{diag} \backslash \check{G}$, so the relative spectral discrepancy is
    $$\sigma_{\check{X}_1/\check{X}_2} = \sigma_{\mathbf{A}^n} = 0,$$
    showing that the discrepancy equation is satisfied.
\end{example}

\begin{example}
    In the cases of principal interest in \cite{CV}, the relative automorphic discrepancy (measured with respect to the Whittaker period) is given by $2\gimel$ in Lemma 7.2 of \textit{op. cit}. In this case, as in the preceding example, we have $u = \mathrm{dim}(U)-1$, reflecting the fact that the automorphic period of interest unfolds ``in one step" to the Whittaker period.
\end{example} 

\begin{example}
    Finally we consider the most important case of our present discussion, that is, the computation of $P_{\check{X}}$ (where $\check{X}$ is the singular variety defined in \S \ref{section Gcheck Xcheck}) normalized by the Shalika period as in Definition \ref{def shalikanormalized}. In other words, we set $M_2 = M$ (for $n = 2$) , and while $M_1$ strictly speaking does not exist as a Hamiltonian action, it should be the cotangent bundle over the singular variety $\check{X}$. Using the ingredients
    $$\alpha(M_1) = \mathrm{dim}(\check{X}) - \mathrm{dim}(\check{G}) - 3 = -10,$$
    $$\alpha(M_2) =  -\mathrm{dim}(U \mathrm{PGL}_2) - \varepsilon_\Sh = -15,$$
    $$s=-1/2,\,\,\,|\eta_{M_1}(h_{\partial^{1/2}}^{-1}) |^{-1/2} = \Delta^{-\frac{9}{4}},$$
    $$u = 3 = \mathrm{dim}(U) -1,$$
    we have that the relative automorphic discrepancy
    $$\alpha_{M_1/M_2} = -10 - (-15) +(-9) + 2(3-4) = -6,$$
    as was seen in the automorphic formula of Theorem \ref{thm PXcheck final}. On the spectral side, we are considering the spectral period quotient $L_X/L_{\check{M}}$, where $X$ is the singular variety defined in \S \ref{Section:The:Space:G:X}. Since $X$ is a $Y$-fiber bundle over $\check{M}$, the relative spectral discrepancy calculates easily as
    $$\sigma(Y) = \varepsilon_Y - \mathrm{dim}(Y) = 3-4=-1.$$
    The difference between the two preceding equations gives rise to a discrepancy of $-5$.  
\end{example}

We remark that even in the smooth setting, equation \eqref{equation discrepancy equation} can already be useful in providing an ansatz for the type of spectral period one represents with $P_{M_1}$ and assuming an unfolding to $M_2$ being the Whittaker period: indeed, in this situation we would expect $P_{M_1}$ to distinguish a subgroup $H \subset \check{G}$ such that $\alpha_{M_1/M_2} = \mathrm{dim}(H) - \mathrm{dim}(\check{G})$.

\bibliographystyle{alpha}
\bibliography{MainCCGT}
\end{document}